\documentclass{article}

\usepackage[utf8]{inputenc}
\usepackage[T1]{fontenc}
\usepackage{lmodern}
\usepackage{amsfonts,amssymb}
\usepackage[english]{babel}
\usepackage{amsmath}

\usepackage{amsthm}

\usepackage{epsfig}
\usepackage{mathrsfs}
\usepackage{xcolor}
\usepackage{graphicx}
\usepackage{dsfont}
\usepackage{enumerate}
\usepackage{cases}
\usepackage{bbm}
\usepackage{hyperref}
\usepackage{stmaryrd}

\def\RR{\mathbb{R}}

\def\div{\mathrm {div}\,}

\def\of{\bar{f}}
\def\bO{\bar{\Omega}}
\def\dO{\pa\Omega}

\def\Delsv{\big(-\Delta_v\big)^s}

\def\Dels{\big(-\Delta\big)^s}

\def\DelsSR{(-\Delta)_{\text{\tiny{SR}}}^s}

\def\iRn{\int_{\RR^d}}
\def\iRd{\int_{\RR^d}}

\def\d{\, {\rm{d}} }

\def\ds{\displaystyle}

\def\bv{\underline{v}}
\def\bz{\underline{z}}
\def\bw{\underline{w}}

\def\pa{\partial}
\def\na{\nabla}
\def\eps{\varepsilon}

\def\HSRs{\mathcal{H}_{\textrm{\tiny{SR}}}^s }
\def\HSR2s{\mathcal{H}_{\textrm{\tiny{SR}}}^{2s} }

\def\Lsr*{\mathscr{L}_{\textrm{\tiny{SR}}}^*}
\def\ds{\displaystyle}

\def\Hdiff0{\mathcal{H}_{\textrm{\tiny{diff}},0}^s }

\theoremstyle{plain}
\newtheorem{thm}{Theorem}[section]
\newtheorem*{thm*}{Theorem}
\newtheorem{lemma}[thm]{Lemma}
\newtheorem*{lemma*}{Lemma}
\newtheorem{prop}[thm]{Proposition}
\newtheorem*{prop*}{Proposition}
\newtheorem{defi}{Definition}[section]
\newtheorem*{defi*}{Definition}

\newtheorem{rmq}[thm]{Remark}

\makeatletter
\newtheorem*{rep@theorem}{\rep@title}
\newcommand{\newreptheorem}[2]{%
\newenvironment{rep#1}[1]{%
 \def\rep@title{#2 \ref{##1}}%
 \begin{rep@theorem}}%
 {\end{rep@theorem}}}
\makeatother

\newreptheorem{prop}{Proposition}{\bf}{\rm}
\newreptheorem{thm}{Theorem}{\bf}{\rm}
\newreptheorem{cor}{Corollary}{\bf}{\rm}
\newreptheorem{lemma}{Lemma}{\bf}{\rm}

\title{Anomalous diffusion limit of kinetic equations in spatially bounded domains}
\date{\today}
\author{Ludovic Cesbron\footnote{DPMMS, Center for Mathematical Sciences, University of Cambridge,  
Wilberforce Road, Cambridge CB3 0WB UK. Email: lpc31@cam.ac.uk } }

\makeindex
\begin{document}
\maketitle

\begin{abstract}
This paper is devoted to the anomalous diffusion limit of kinetic equations with a fractional Fokker-Planck collision operator in a spatially bounded domain. We consider two boundary conditions at the kinetic scale: absorption and specular reflection. In the absorption case, we show that the long time/small mean free path asymptotic dynamics are described by a fractional diffusion equation with homogeneous Dirichlet-type boundary conditions set on the whole complement of the spatial domain. On the other hand, specular reflections will give rise to a new operator which we call specular diffusion operator and write $\DelsSR$. This non-local diffusion operator strongly depends on the geometry of the domain and includes in its definition the interaction between the diffusion and the boundary. We consider two types of domains: half-spaces and balls in $\RR^d$. In these domains, we prove properties of the specular diffusion operator and establish existence and uniqueness of weak solutions to the associated heat-type equation.
\end{abstract}


\tableofcontents

\section{Introduction}

Because of the non-local nature of fractional diffusion, it is not clear how it should interact with a boundary. The confinement of non-local diffusion processes raises a lot of questions and has received a growing interest in recent years from both the points of view of stochastic analysis, see for instance \cite{Bogdan+2003} \cite{Chen02}, and partial differential equations, see e.g. \cite{GuanMa05}, \cite{Felsinger13}, \cite{Mou15},\cite{DiPierro17}. The purpose of this paper is to derive such confinements. Our approach consists in considering the confined non-local diffusion equation as an anomalous limit of confined kinetic equations. As a result, the interaction between the non-local diffusion phenomena and the spatial boundaries will be entirely deduced from the kinetic setting where there are no ambiguities in the definition of boundary conditions. We believe that this method, since it conserves the physical relevance of the kinetic models, is a promising step towards determining fractional equivalents to the Dirichlet and Neumann-type boundary conditions for classical heat equations.\\
More precisely, we present in this paper the derivation of fractional diffusion equations on spatially bounded domain from kinetic equations with a fractional Fokker-Planck collision operator. This setting is particularly relevant due the fact that those kinetic equations feature a non-local collision operator that acts solely on the velocities of the particles which are unbounded. As a result, although we already have an explicit non-local operator at the kinetic scale, its interaction with the spatial boundary will only arise as we look at the anomalous diffusion limit. \\
We investigate the long time/small mean-free-path asymptotic behaviour of the solution of the fractional Vlasov-Fokker-Planck (VFP) equation:
\begin{subequations}{\label{eq:meso0}}
\begin{align}
&\pa_t f +  v\cdot \na_x f  =  \na_v\cdot (vf) -(-\Delta_v)^s f  & \mbox{ in } \RR^+\times\Omega\times\RR^d,\label{eq:vlfp0}\\
&f(0,x,v) = f_{in} (x,v) &\mbox{ in } \Omega\times\RR^d,  \label{eq:vlfpid0}
\end{align}
\end{subequations}
for $s\in(0,1)$ on a smooth convex domain $\Omega$. We introduce the oriented set: 
\begin{equation} \label{def:sigma}
\Sigma_\pm = \lbrace (x,v) \in \Sigma; \pm n(x)\cdot v >0 \rbrace \text{ with } \Sigma = \dO \times \RR^d
\end{equation}
where $n(x)$ is the outgoing normal vector and we denote by $\gamma f$ the trace of $f$ on $\RR^+\times\dO\times\RR^d$. The boundary conditions then take the form of a balance between the values of the traces of $f$ on these oriented sets $\gamma_\pm f := \mathds{1}_{\Sigma_\pm} \gamma f$. We will consider two types of conditions introduced by J. C. Maxwell in the appendix of \cite{Maxwell} in 1879:
\begin{itemize}
\item The absorption boundary condition (also called zero inflow) : for all $(x,v)\in\Sigma_-$
\begin{equation} \label{eq:Aop}
\gamma_-f(t,x,v) = 0
\end{equation}
\item The local-in-velocity reflection operator called {\it specular reflection}: for all $(x,v)\in\Sigma_-$ 
\begin{equation}\label{eq:SRop}
\gamma_-f(t,x,v) = \gamma_+f\big(t,x,\mathcal{R}_x(v)\big)
\end{equation} 
where $\mathcal{R}_x(v) = v-2\big(n(x)\cdot v \big)n(x)$ which is illustrated in Figure \ref{fig:SpecRef}.
\end{itemize}
\begin{figure}[h]
\centering
\includegraphics[width=11cm,height=9cm]{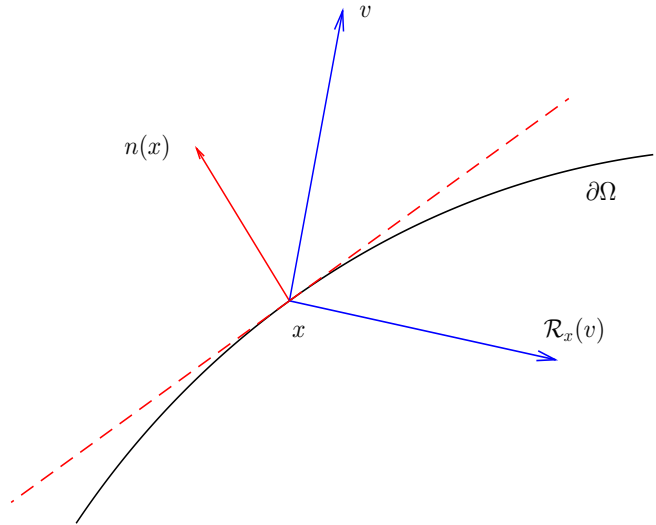}
\caption{Specular reflection operator}
\label{fig:SpecRef}
\end{figure}

The fractional VFP equation models the evolution of the distribution function $f(t,x,v)$ of a cloud of particles in a plasma. The left hand side of \eqref{eq:vlfp0} models the free transport of the particles, while on the right hand side the fractional Fokker-Planck operator 
\begin{equation} \label{def:LFPop}
\mathcal{L}^s f = \na_v \cdot (vf) -(-\Delta_v)^s f
\end{equation}
describes the interactions of the particles with the background. It can be interpreted as a deterministic description of a Langevin equation for the velocity of the particles: $\dot{v}(t) = -v(t) + A(t)$, where $A(t)$ is a white noise. This model describes the evolution of the velocity of a particle as the result of two phenomena, a viscosity-like interaction that causes the velocity to slow down and a white noise that causes it to jump at random times which can be interpreted as the consequence of the interaction between the particle and the background. The classical Fokker-Planck operator corresponds to $s = 1$ and arises when $A(t)$ is a Gaussian white noise. In that case, equilibrium distributions (solutions of $\mathcal{L}^1 M = 0$) are Maxwellian (or Gaussian) velocity distributions: $M = C \exp(-|v|^2/2)$. However, some experimental measurements of particles and heat fluxes in confined plasma point to non-local features and non-Gaussian distribution functions, see section 2 in the introduction of \cite{CesbronThesis} for more details. The introduction of L\'{e}vy statistic in the velocity equation (replacing the Gaussian white noise by L\'{e}vy white noise in the Langevin equation) can be seen as an attempt at taking into account these non-local effects in plasma turbulence. 

In order to study the long time/small mean free path asymptotic behaviour of the solutions of the fractional VFP equation, we introduce the Knudsen number $\eps$ which represents the ratio of the mean-free-path to the macroscopic length scale, or equivalently the ratio of the mean time between two collisions to the macroscopic time scale. We use this $\eps$ to rescale the time variable as
\begin{equation} \label{eq:scale}
t' = \eps^{2s-1} t.
\end{equation}
Moreover, we also introduce $1/\eps$ as a factor of the fractional Fokker-Planck operator to model the mean-free-path growing smaller as a consequence of the number of collisions per unit of time increasing. Hence, we consider the following scaling of \eqref{eq:vlfp0}-\eqref{eq:vlfpid0}:
\begin{subequations}{\label{eq:meso}}
\begin{align}
&\eps^{2s-1}\pa_t f^\eps +  v\cdot \na_x f ^\eps = \frac{1}{\eps} \mathcal{L}^s \big( f^\eps\big) & \mbox{ in } \RR^+\times\Omega\times\RR^d,\label{eq:vlfp}\\
&f^\eps(0,x,v) = f_{in} (x,v) &\mbox{ in } \Omega\times\RR^d.  \label{eq:vlfpid}
\end{align}
\end{subequations}
and investigate the behaviour of the solution $f^\eps$ as $\eps$ goes to $0$.\\

In the non-fractional framework, the first papers concerned with the relation between the VFP equations on the whole space and diffusion equations can be found in the late 70's and early 80's with the work of D'Arruda, Larsen \cite{DArrudaLarsen78} and also Beals, Protopopescu \cite{BealsProtopopescu83} where they prove diffusion limits in a perturbative settings; as well as the work of Bardos, Santos, Sentis \cite{bardos1984diffusion} in 84 where they lay down the theoretical basis for diffusion limits in general. More recently, several works have been concerned with the diffusion limits of the more elaborate Vlasov-Poisson-Fokker-Planck systems. For instance, in 2000, Poupaud and Soler in \cite{PoupaudSoler} prove the diffusion limit under parabolic scaling,  which is exactly \eqref{eq:vlfp} with $s=1$, for a small enough time interval. It is easy to see that their results imply, for the VFP equation, that the solution $f^\eps$ converges, as $\eps$ goes to $0$, to $\rho (t,x) M(v)$ where $M$ is the Maxwellian equilibrium of the Fokker-Planck operator and $\rho$ is the limit of the density $\rho^\eps = \int f^\eps \d v$ and satisfies a Heat equation. Their results were then extended (still in the Poisson case) in 2005 by Goudon \cite{Goudon05} to a global in time convergence in dimension 2 with bounds on the entropy and energy of the initial data as to ensure that they don't develop singularities in the limit system, and later in 2010 by El Ghani and Masmousi \cite{ElGhani10} who proved the global in time convergence in higher dimensions with similar initial bounds. \\

In the fractional framework, Biler, Karch \cite{Biler+2003} and Gentil, Imbert \cite{Gentil+2008} investigate the long-time behaviour of L\'{e}vy-Fokker-Planck equations
\begin{equation} \label{eq:LevyFokkerPlanck}
\pa_t f = \div(f \na\phi ) + \mathcal{I}\big[ f\big] 
\end{equation}
where $\mathcal{I}$ is the infinitesimal generator of a L\'{e}vy process. This family of operators includes the fractional Fokker-Planck operator since the fractional Laplacian of order $s$ is the generator of a particular $2s$-stable L\'{e}vy process whose characteristic exponent is $|\xi |^{2s}$. Biler and Karch prove convergence of the solution of \eqref{eq:LevyFokkerPlanck} to the unique normalised equilibrium of the L\'{e}vy-Fokker-Planck operator, later improved by Gentil and Imbert to exponential convergence in a weighted $L^2$ space where the weight is prescribed by the equilibrium. Their proofs use entropy production methods and a modified logarithmic Sobolev inequality which we will use later on to establish a priori estimates on the solutions of the fractional VFP equation in a similar weighted $L^2$ space. In \cite{GuanMa06} and references within, Guan and Ma give a description of this equilibrium and proofs that it is, in particular, heavy-tailed, as stated below in Proposition \ref{prop:eqLFP}.\\
This characterisation of the equilibrium of the fractional Fokker-Planck operator and the entropy production method allowed the author with A. Mellet and K. Trivisa to establish in \cite{Cesbron12} the anomalous diffusion limit of the fractional VFP equation. More precisely, we proved the following result:
\begin{thm*}[Theorem 1.2 in \cite{Cesbron12}]\label{thm:main}
Assume that $ f_0 \in L^2_{F^{-1}}(\RR^d\times\RR^d)$ where $F(v)$ is the normalised heavy-tailed equilibrium of the fractional Fokker-Planck operator.  
Then, up to a subsequence, the solution $f^\eps$ of the rescaled fractional VFP equation on the whole space \eqref{eq:vlfp}-\eqref{eq:vlfpid} converges weakly in
$L^\infty(0,T;L^2_{F^{-1}}(\RR^d\times\RR^d))$, as $\eps\to 0$ to  $\rho(t,x) F(v)$ where $\rho(t,x)$ solves
\begin{subequations}\label{eq:adiff}
\begin{align} 
&\pa_t \rho + (-\Delta_x)^s \rho = 0 &\mbox{ in } (0,\infty)\times\RR^{d} & \\[5pt]
&\rho(0,x)=\rho_0(x) & \mbox{ in }  \RR^{d} &
\end{align}
\end{subequations}
with $\ds \rho_0(x)=\int f_0(x,v)\, dv$.
\end{thm*}
Note that we use here and throughout this paper the notation $L^2_{\mu} (\RR^d)$ for the $L^2$ space with weight $\mu$. \\
We can see how this result compares to the aforementioned diffusion limit of the classical Vlasov-Fokker-Planck. However, the method used in \cite{Cesbron12} to derive this asymptotic behaviour is quite different from what is done is the non-fractional case, and rests upon the particular structure of the fractional VFP equation. Indeed, and this will be essential for the rest of this paper, if we consider the Fourier transform of \eqref{eq:vlfp0} in $x$ and $v$ (respective Fourier variables $k$ and $\xi$) on $\RR^d\times\RR^d$ we get the following PDE:
\begin{equation} \label{eq:Fourier}
\pa_t \hat{f}(t,k,\xi) + (k-\xi)\cdot \na_\xi \hat{f}(t,k,\xi) = -|\xi|^{2s} \hat{f}(t,k,\xi).
\end{equation}
This PDE is scalar-hyperbolic so if we follow well-chosen {\it characteristic lines}, it becomes an ODE which can be solved explicitly. The main idea of \cite{Cesbron12} is to transpose these {\it characteristic lines} in a non-Fourier setting in order to derive fractional diffusion. The method we developed in the present work is inspired from the same idea but confined to bounded domains.

Kinetic equations on bounded domains, because of their obvious physical relevance, have always received a lot of attention. There have been many works concerning existence of global weak solutions on bounded domains with absorbing-type or reflection-type boundary conditions. We would like to mention the work of Carrillo \cite{Carrillo98} on the VPFP system, as well as the work of Mellet and Vasseur \cite{MelletVasseur} for the VFP equation coupled to compressible Navier-Stokes via drag force, because their techniques could be generalised to the fractional VFP equation with some modifications to handle the non-local property of the diffusion operator and we will indeed follow the line of reasoning of \cite{Carrillo98} to prove well-posedness of the specular diffusion equation in section \ref{sec:wellposedness}.\\

Hydrodynamical and diffusion limits in bounded domains have also been the subject of many works. For instance, in 1987, Degond and Mas-Gallic \cite{Degond1987} established the first rigorous diffusion limit for the (classical) VFP equation in 1 dimension on a bounded domain. This result has been improved many times (cf. references within \cite{WuLinLiu}), and in 2015 Wu, Lin and Liu proved in \cite{WuLinLiu} that the diffusion limit of a VPFP system for multiple species charged particles with reflection boundary conditions is a Poisson-Nernst-Planck system with homogeneous Neumann boundary conditions. Other examples of macroscopic limits are the work Masmoudi and Saint-Raymond who, in 2003, showed in \cite{Masmoudi03} that the Boltzmann equation with Maxwell boundary conditions converges to the Stokes-Fourier system with Navier boundary conditions, or, more recently, the work of Jiang, Levermore and Masmoudi who established in \cite{Jiang09} the acoustic limit for DiPerna-Lions solutions and recovered impermeable boundaries for the acoustic system. \\

Before stating our main results, let us present properly the fractional Laplacian and give some well-known properties of this operator and the associated fractional Fokker-Planck operator.

\subsection{Preliminaries on the fractional Fokker-Planck operator} \label{subsec:def}
The fractional Laplacian can be defined as a pseudo-differential operator of symbol $|\xi|^{2s}$ which can be written in Fourier transform as:
\begin{equation} \label{def:delsFourier}
\mathcal{F} \Big[ \Dels f(\xi) \Big] = |\xi|^{2s}\mathcal{F} [ f ] (\xi).
\end{equation}
Much like the Laplace operator is the infinitesimal generator of a Brownian motion, the fractional Laplacian is the generator of a L\'{e}vy process. More precisely, it is the generator of a L\'{e}vy process  $V_t$ whose transition density $\rho(t,y-x)$ relative to the Lebesgue measure is given in Fourier by:
\begin{equation*}
\iRd e^{i v\cdot \xi} \rho(t,v) \d v = e^{-t |\xi |^{2s}}.
\end{equation*}
The fractional Laplacian can also be written as a singular integral, which will be most useful in the PDE framework:
\begin{equation} \label{def:Dels}
\Dels f(v) = c_{s,d} \text{P.V.} \iRn \frac{f(v)-f(w)}{|v-w|^{d+2s}} \d w
\end{equation}
where $c_{s,d}$ is a constant depending on $s$ and the dimension $d$ given by:
\begin{equation} \label{def:cns}
c_{d,s} = \bigg( \underset{\RR^d}{\int} \frac{1-\cos (\zeta_1)}{|\zeta|^{d+2s}} \,\text{d}\zeta \bigg)^{-1}.
\end{equation} 
The properties of this operator have been studied in 2007 by Silvestre in \cite{Silvestre2007} and more recently by DiNezza, Palatucci and Valdinoci in \cite{DiNezza+} where they focus on the link between $\Dels$ and the fractional Sobolev spaces $H^s(\RR^d)$. \\
As mentioned before, the interaction between the non-locality of the fractional Laplacian and the boundary of a domain raises a lot of questions. In 2003, Bogdan,Burdzy and Chen introduced in \cite{Bogdan03} the notion of reflected $2s$-stable processes, which are the restriction of a $2s$-stable process, such as $V_t$ defined above, to a open set $\Omega$ in $\RR^d$. In particular, they define the killed process, constructed by adding a coffin state $\pa$ to $\RR^d$ and defining $W_t$, the killed process associated with $V_t$, as:
\begin{equation*}
W_t(\omega) = \left| \begin{aligned}  & V_t (\omega)  \mbox{ for } t\leq t_\Omega(\omega) \\
															 & \pa \mbox{ for } t > t_\Omega(\omega)
						  \end{aligned} \right. 
\end{equation*}
where $t_\Omega := \inf \lbrace t>0 : V_t \notin \Omega\rbrace$ is the first exit time. The Dirichlet form of this process on $L^2(\Omega,dx)$ is $(\mathcal{C}, \mathcal{F}^\Omega)$ defined as:
\begin{align*}
&\mathcal{F}^\Omega = \bigg\{ f\in L^2(\RR^d) :  \underset{\RR^d\times\RR^d}{\iint} \frac{\big( f(x)-f(y)\big)^2}{|x-y|^{d+2s}} \d x \d y <\infty \mbox{ and } f = 0 \mbox{ q.e. on } \RR^d\setminus\Omega \bigg\} \\
&\mathcal{C}(f,g) = \frac{1}{2} c_{d,s} \underset{\Omega\times\Omega}{\iint} \frac{\big( f(x)-f(y)\big)\big( g(x)-g(y)\big)}{|x-y|^{d+2s}}  \d x \d y + \underset{\Omega}{\int} f(x) g(x) \kappa_\Omega(x) \d x
\end{align*}
where q.e. means quasi everywhere and $\kappa_\Omega$ is the density of the killing measure of $W_t$ given by:
\begin{equation*}
\kappa_\Omega (x) = c_{d,s} \underset{\RR^d\setminus\Omega}{\int} \frac{1}{|x-y|^{d+2s}} \d y.
\end{equation*}
They also define more general reflected processes by extending the lifetime of the process beyond $t_\Omega$. The killed process has a direct link with the PDE approach to fractional Laplacian on bounded domain. Indeed, in 2014, Felsinger, Kassmann and Voigt considered in \cite{Felsinger13}, the Dirichlet problem for non-local operators which, in case of the fractional Laplacian, reads:
\begin{equation} \label{eq:dirfraclap}
\begin{array}{rll}
\Dels f &= u  \hspace{0.5cm} &\mbox{ in } \Omega \\
f &= g &\mbox{ on } \RR^d\setminus\Omega.
\end{array}
\end{equation}
They introduced the Hilbert space $H_\Omega\big(\RR^d; \frac{1}{|x-y|^{d+2s}} \big)$, which is exactly the space $\mathcal{F}^\Omega$ defined above, provided with the norm $\lVert f\lVert_{L^2(\RR^d)} + \mathcal{C}(f,f)$. They wrote a variational formulation of the Dirichel problem \eqref{eq:dirfraclap} in that Hilbert space and proved existence and uniqueness of solutions. Note that their results actually include a large family of non-local operators, we stated it here for the fractional Laplacian since it is the subject of this paper, but their work goes far beyond. For regularity results on the solutions of the homogeneous Dirichlet problem with fractional Laplacian inside the domain and up to the boundary, we refer the reader to Grubb \cite{Grubb2013} and Ros-Oton, Serra \cite{RosOton14}.

The fractional Fokker-Planck operator $\mathcal{L}^s$ has been introduced as a generalization of the classical Fokker-Planck operator for general L\'{e}vy stable processes in 2000 by Yanovsky, Chechkin, Schertzer and Tur \cite{Yanovsky} and the following year it was derived from the wider class of non-linear Langevin-type equation driven by a L\'{e}vy stable noise by Schertzer Larchevêque Duan Yanovsky and Lovejoy in \cite{Schertzer2001}. \\
In the present paper, the most crucial property of the fractional Fokker-Planck operator will be the fact that its thermodynamical equilibrium is a L\'{e}vy stable distribution i.e. a heavy-tailed distribution, instead of the Maxwellian distribution that arise in the non-fractional setting. The explicit solution in Fourier transform of the equation $\mathcal{L}^s F = 0$ yields the following result
\begin{prop} \label{prop:eqLFP}
For $s\in (0,1)$ and $\nu > 0$, there exists a unique normalized equilibrium distribution function $F(v)$, solution of
\begin{equation}\label{eq:heavytailLFP}\mathcal L^s(F)=\nu \na_v \cdot (vF) -(-\Delta_v)^s F=0, \qquad \int_{\RR^d} F(v)\, dv =1. 
\end{equation}
Furthermore, $F(v)>0$ for all $v$, and $F$ is a  heavy-tailed distribution  function satisfying 
$$ F(v)\sim\frac{C}{|v|^{d+2s}}\qquad \mbox{ as } |v|\to\infty.$$
\end{prop} 
For a more detailed presentation of the equilibrium of $\mathcal{L}^s$ we refer the reader to \cite{AcevesCesbron} and references within.

\subsection{Main Results}

Throughout this paper, for any $T>0$ we write $Q_T = [0,T)\times\bO\times\RR^d$ and $\Sigma= \dO \times \RR^d$ as mentioned in \eqref{def:sigma}. Also, we will write $L^p (\Sigma_\pm)$ the Lebesgue space associated with the norm:
\begin{equation} \label{def:Lpsig}
\lVert \gamma_\pm f \lVert_{L^p(\Sigma_\pm)} = \bigg( \underset{\Sigma_\pm}{\iint} | \gamma_\pm f |^p \big( n(x)\cdot v \big) \d \sigma(x) \d v \bigg)^{1/p}
\end{equation}
As usually in the framework of fractional Vlasov-Fokker-Planck equations we use the following definitions of weak solutions
\begin{defi} \label{def:weaksolDir}
We say that $f$ is a weak solution of the fractional VFP equation with Dirichlet type boundary conditions \eqref{eq:vlfp0}-\eqref{eq:vlfpid0}-\eqref{eq:Aop} on $[0,T)$ if
\begin{equation}\label{eq:regsolA}
\begin{aligned}
&f(t,x,v) \geq 0 \hspace{1cm} \forall (t,x,v)\in [0,T)\times\Omega\times\RR^d\\
& f \in L^2_{t,x} H^s_v(Q_T) = \bigg\{ f\in L^2(Q_T ) , \frac{f(t,x,v)-f(t,x,w)}{|v-w|^{\frac{d+2s}{2}}} \in L^2(Q_T\times\RR^d) \bigg\}
\end{aligned}
\end{equation}
satisfies
\begin{equation} \label{eq:DBCf}
\gamma_\pm f \in L^1\big(0,T;L^1(\Sigma_\pm)\big),\hspace{0.2cm} \text{and} \hspace{0.2cm} \gamma_- f=0 
\end{equation}
and \eqref{eq:vlfp0} holds in the sense that for any $\phi$ such that 
\begin{equation} \label{eq:weakphiA}
\begin{aligned}
&\phi \in C^\infty ( Q_T ) \hspace{1cm} \phi(T,\cdot,\cdot)=0 \\
&\gamma_+\phi(t,x,v) = 0 \hspace{1cm} \forall (t,x,v)\in [0,T)\times\Sigma_+
\end{aligned}
\end{equation} 
we have:
\begin{equation}\label{eq:wfvlfpA}
\begin{aligned}
& \underset{Q_T}{\iiint} f \Big( \pa_t \phi +  v\cdot \na_x \phi - v\cdot \na_v \phi - \Delsv \phi \Big) \d t \d x \d v \\
&+ \underset{\Omega\times\RR^d}{\iint} f_{in}(x,v)\phi(0,x,v) \d x\d v  = 0.
\end{aligned}
\end{equation}
\end{defi} 
In the case of specular reflection, it is well known that reflective boundaries are often responsible for a loss of regularity of the traces of $f$, see \cite{Mischler10}. Hence, we define the following notion of weak solutions:
\begin{defi} \label{def:weaksolSR}
We say that $f$ is a weak solution of \eqref{eq:vlfp0}-\eqref{eq:vlfpid0}-\eqref{eq:SRop} on $[0,T)$ if
\begin{equation}\label{eq:regsolSR}
\begin{aligned}
&f(t,x,v) \geq 0 \hspace{1cm} \forall (t,x,v)\in [0,T]\times\Omega\times\RR^d\\
& f \in L^2_{t,x} H^s_v(Q_T) = \bigg\{ f\in L^2(Q_T ) , \frac{f(t,x,v)-f(t,x,w)}{|v-w|^{\frac{d+2s}{2}}} \in L^2(Q_T\times\RR^d) \bigg\}
\end{aligned}
\end{equation}
and \eqref{eq:vlfp0} holds in the sense that for any $\phi$ such that:
\begin{equation} \label{eq:weakphiSR}
\begin{aligned}
&\phi \in C^\infty ( Q_T ) \hspace{1cm} \phi(T,\cdot,\cdot)=0 \\
&\gamma_+\phi(t,x,v) = \gamma_-\phi\big(t,x,\mathcal{R}_x(v)\big) \hspace{1cm} \forall (t,x,v)\in [0,T)\times\Sigma_+
\end{aligned}
\end{equation}
we have:
\begin{equation}\label{eq:wfvlfpSR}
\begin{aligned}
& \underset{(0,T)\times\Omega\times\RR^d}{\iiint} f \Big( \pa_t \phi +  v\cdot \na_x \phi - v\cdot \na_v \phi - \Delsv \phi \Big) \d t \d x \d v \\
&\hspace{1cm}+ \underset{\Omega\times\RR^d}{\iint} f_{in}(x,v)\phi(0,x,v) \d x\d v = 0.
\end{aligned}
\end{equation}
\end{defi} 

The existence and uniqueness of such weak solutions can be established by adapting the method of Carrillo  in \cite{Carrillo98} or Mellet and Vasseur in \cite{MelletVasseur} in order to handle the non-local property of the diffusion operator. In the whole space, this was done my the author and Aceves-Sanchez in \cite{AcevesCesbron}. We do not dwell on this issue for it is not the focus of this paper.\\
In the first part of this paper, section \ref{section:apriori}, we establish a priori estimates on the weak solutions, in both the absorption and the specular reflection case, using the dissipative property of the fractional Fokker-Planck operator. We then use those estimates to prove convergence of the weak solution of the rescaled fractional VFP equation:
\begin{prop} \label{prop:weakCV}
Let $f_{in}$ be in $L^2_{F^{-1}(v)}(\Omega\times\RR^d)$ and $s$ be in $(0,1)$. The weak solution $f^\eps$ of the rescaled fractional VFP equation \eqref{eq:vlfp}-\eqref{eq:vlfpid} with absorption \eqref{eq:Aop} or specular reflections \eqref{eq:SRop} on the boundary satisfies
\begin{equation} \label{eq:weakCV}
f^\eps (t,x,v) \rightharpoonup \rho(t,x)F(v) \mbox{ weakly in } L^\infty\big(0,T; L^2_{F^{-1}(v)}(\Omega\times\RR^d) \big)
\end{equation}
where $\rho(t,x)$ is the limit of the macroscopic densities $\rho^\eps = \iRd f^\eps \d v$.
\end{prop}
In sections 3 and 4, we establish the anomalous diffusion limits, i.e. we identify the limit $\rho$ as solution to a diffusion equation. The main idea of these proofs is to take advantage of the aforementioned scalar-hyperbolic structure of the fractional VFP equation in Fourier space \eqref{eq:Fourier}. To that end, we introduce an auxiliary problem whose purpose is to construct, from any test function $\psi(t,x)$, a function $\phi^\eps(t,x,v)$ which will be constant along the \textit{characteristic lines} of the fractional VFP equation modified to take into account the boundary conditions, and such that $\lim_{\eps \searrow 0}\phi^\eps(t,x,v) = \psi(t,x)$. For the absorption boundary condition, the auxiliary problem reads for $\psi\in\mathcal{D}([0,T)\times\Omega)$:
\begin{subequations}
\begin{align}
&\eps v\cdot \na_x\phi^\eps - v\cdot \na_v \phi^\eps = 0 \hspace{1cm} &\forall (t,x,v) \in \RR^+\times\Omega\times\RR^d, \label{eq:APabsmain}\\
&\phi^\eps(t,x,0) = \psi(t,x) &\forall (t,x)\in\RR^+\times\Omega, \label{eq:APabsit}\\
&\gamma_+ \phi^\eps(t,x,v) = 0 &\forall (t,x,v)\in\RR^+\times\Sigma_+.\label{eq:APabsbc}
\end{align}
\end{subequations}
We construct a solution of this problem and use it as a test function in the weak formulation of \eqref{eq:vlfp}-\eqref{eq:vlfpid}-\eqref{eq:Aop}. We then show that we can take the limit in the weak formulation to prove:
\begin{thm}\label{thm:mainA}
Assume that $f_{in}$ is in $L^2_{F^{-1}(v)} (\Omega\times\RR^d)$ and $s$ is in $(0,1)$. Then the solution $f^\eps$ of \eqref{eq:vlfp}-\eqref{eq:vlfpid}-\eqref{eq:Aop}, converges weakly in the sense of Proposition \ref{prop:weakCV} to $\rho(t,x) F(v)$ where the extension of $\rho(t,x)$ by $0$ outside of $\Omega$ is a weak solution of
\begin{subequations} \label{eq:FracHeatAbs}
\begin{align}
&\pa_t \rho + \Dels \rho = 0 & (t,x)\in [0,T)\times\Omega  \label{eq:FracHeatAbsmain} \\
&\rho(x,0) = \rho_{in}(x) & x\in\Omega  \label{eq:FracHeatAbsit}\\
&\rho(t,x) = 0 & t\in [0,T), x \in \RR^d\setminus\Omega \label{eq:FracHeatAbsBC}
\end{align}
\end{subequations}
where $\rho_{in}(x) = \int f_{in} \d v$, in the sense that for all $\psi \in \mathcal{D}([0,T)\times\RR^d)$ compactly supported in $\Omega$:
\begin{equation} 
\underset{(0,T)\times\RR^d}{\iint} \rho(t,x) \big( \pa_t \psi(t,x) - \Dels \psi(t,x)\big) \d t \d x + \underset{\RR^d}{\int} \rho_{in}(x) \psi(0,x) \d x = 0.
\end{equation}
\end{thm}
In this macroscopic equation, the extension by $0$ of the function $\rho$ can be interpreted as an extension of \eqref{eq:Aop}, the homogeneous Dirichlet boundary condition in the kinetic equation, to the whole complementary of the domain $\Omega$ as a consequence of the non-local nature of the  fractional Laplacian operator. Note that, as expected, this limit problem is directly related to the killed process of Bogdan, Burdzy and Chen \cite{Bogdan03}, and the Dirichlet problem of Felsinger, Kassmann and Voigt \cite{Felsinger13}.\\
For the specular reflection boundary condition, if we want follow the \textit{characteristic lines} as they reflect on the boundary, we need to reduce, when $s\geq 1/2$, the set of test functions to $\mathfrak{D}_T(\Omega)$ defined as:
\begin{equation} \label{eq:defDT}
\mathfrak{D}_T(\Omega) = \Big\{ \psi\in\mathcal{C}^\infty ([0,T)\times\bO) \text{ s.t. } \psi(T,\cdot)=0 \text{ and } \forall x\in\dO : \na_x \psi(t,x)\cdot n(x) =0 \Big\}.
\end{equation} 
The auxiliary problem reads for $\psi \in\mathcal{C}^\infty ([0,T)\times\bO)$ if $s<1/2$ or in $ \mathfrak{D}_T$ if $s\geq 1/2$:
\begin{subequations}
\begin{align}
&\eps v\cdot \na_x\phi^\eps - v\cdot \na_v \phi^\eps = 0 \hspace{1cm} &\forall (t,x,v) \in \RR^+\times\Omega\times\RR^d, \label{eq:APSRmain}\\
&\phi^\eps(t,x,0) = \psi(t,x) &\forall (t,x)\in\RR^+\times\Omega, \label{eq:APSRit}\\
&\gamma_+ \phi^\eps(t,x,v) = \gamma_-\phi^\eps\big(t,x,\mathcal{R}_x(v)\big) &\forall (t,x,v)\in\RR^+\times\Sigma_+.\label{eq:APSRbc}
\end{align}
\end{subequations}
In order to construct a solution of this auxiliary problem we study geodesic trajectories in a Hamiltonian billiard. These trajectories are given by,  parametrised with $s\in[0,\infty)$
\begin{equation}\label{eq:cc}
  \left\{
      \begin{array}{llll}
        &\dot{x}(s) = \eps v(s) \hspace{15mm} &x(0)=x^{in}\in\Omega,  \\
        &\dot{v}(s) = -v(s) &v(0)=v^{in}\in\RR^d,\\
        & \text{If } x(s)\in\dO \text{ then } v(s^+)= \mathcal{R}_{x(s)}(v(s^{-})),
      \end{array}
    \right.
\end{equation}
as illustrated in Figure \ref{fig:extraj} for example when $\Omega$ is a ball. We construct a function $\eta :\Omega\times\RR^d \mapsto \bO$ that will be constant along those trajectories, defined as $\eta(x^{in},v^{in}) = \lim_{s\rightarrow \infty} x(s)$ which obviously, strongly depends on the geometry of the domain and we will show that it is well defined when $\Omega$ is a half-space or a strongly convex domain. This $\eta$ function allows us to build a solution to the auxiliary problem:
\begin{prop}\label{prop:solAPSR}
If $\Omega$ is either a half-space or smooth and strongly convex, then there exists a function $\eta : \Omega\times\RR^d \rightarrow \bO$ such that 
\begin{equation}
\phi^\eps (t,x,v) = \psi\big(t,\eta(x,\eps v) \big)
\end{equation}
is a solution of the auxiliary problem \eqref{eq:APSRmain}-\eqref{eq:APSRit}-\eqref{eq:APSRbc}.
\end{prop}
Although the regularity of this $\eta$ function is rather simple to study in the half-space, it is much harder to understand in the ball and we will devote Appendix \ref{app:FreeTransport} to this investigation. In fact, it is strongly linked with the free transport equation. Indeed, if we consider the following free transport equation in a ball with specular reflection on the boundary and a homogeneous-in-velocity initial condition: 
\begin{align*}
&\pa_t f  + v\cdot \na_x f = 0 & (t,x,v) \in [0,T)\times\Omega\times\RR^d \\
&f(0,x,v) = \psi(x) & (x,v)\in\Omega\times\RR^d \\
&\gamma_- f(t,x,v) = \gamma_+ f(t,x,\mathcal{R}_x v )  &(t,x,v) \in [0,T)\times \dO \times \lbrace v : v\cdot n(x) <0\rbrace
\end{align*} 
then, using \eqref{eq:APSRmain}-\eqref{eq:APSRit}-\eqref{eq:APSRbc} and Proposition \ref{prop:solAPSR} we can show that a solution of this problem is 
\begin{equation*}
f(t,x,v) = \psi\big( \eta(x, -t v) \big) .
\end{equation*}
As a consequence, the regularity properties of $\eta$ we establish in Appendix \ref{app:FreeTransport} can also be interpreted as a propagation of regularity with respect to the velocity for the free transport equation in a ball with specular reflection on the boundary. Note that the optimal regularity for this problem is an open problem and, to the best of our knowledge, the regularity in velocity that we proved here is the highest known in Sobolev spaces. \\
We are then able to establish the following anomalous diffusion limit.
\begin{thm} \label{thm:mainSR}
Let $\Omega$ be either a half-space or a ball in $\RR^d$ and assume that $f_{in}$ is in $L^2_{F^{-1}(v)} (\Omega\times\RR^d)$ and $s$ is in $(0,1)$. Then the solution $f^\eps$ of \eqref{eq:vlfp}-\eqref{eq:vlfpid}-\eqref{eq:Aop}, converges weakly in the sense of Proposition \ref{prop:weakCV} to $\rho(t,x) F(v)$ where $\rho(t,x)$ satisfies, for any $\psi\in \mathcal{C}^\infty ([0,T)\times\bO)$ if $s<1/2$ and any $\psi \in  \mathfrak{D}_T(\Omega)$ if $s\geq 1/2$:
\begin{equation} \label{eq:FracHeatSRweak}
\underset{(0,T)\times\Omega}{\iint} \rho(t,x) \Big( \pa_t \psi(t,x) - \DelsSR \psi(t,x) \Big) \d t \d x + \underset{\Omega}{\int} \rho_{in}(x)\psi(0,x) \d x = 0.
\end{equation}
where $\rho_{in}(x) = \int f_{in} \d v$ and $\DelsSR$ is defined as:
\begin{equation} \label{def:DelsSR}
\DelsSR \psi(x) = c_{d,s} P.V. \underset{\RR^d}{\int} \frac{\psi(x) - \psi\big(\eta(x,w)\big)}{|w|^{d+2s}} \d w
\end{equation}
\end{thm}
This new operator, which we call \textit{specular diffusion operator}, can be seen as a modified version of the fractional Laplacian where the particles can jump from a position $x$ to a position $y$ in $\Omega$ not only through a straight line (which corresponds to the fractional Laplacian) but also through trajectories that are specularly reflected upon hitting the boundary, and the probability of this jump is $1/|w|^{d+2s}$ where $|w|$ is the length of the trajectory. Note that when $\Omega$ is $\RR^d$, by definition we have $\eta(x,w)= x+w$ so that $\DelsSR$ coincides with the full fractional Laplacian $\Dels$ on $\RR^d$. \\
Theorem \ref{thm:mainSR} can also be proved when $\Omega$ is a stripe $\lbrace x=(x',x_d) \in\RR^d : -1 < x_d < 1\rbrace$ or a cube using arguments from the half-space case in order to handle locally the interaction with the boundary, and from the ball case to handle the multitude of reflections a trajectory in a stripe or a cube may undergo in a finite time. Moreover, in order to extend this theorem to general smooth and strongly convex domains, one only needs to prove that the trajectories described by $\eta$ in that domain satisfy appropriate controls, similar to the ones we state in Lemma \ref{lemma:delsint} in the case of the ball which we prove in Appendix \ref{app:FreeTransport}. The rest of the proof would remain the same.\\
Finally, in the last section of this paper, we focus on the macroscopic equation \eqref{eq:weakFHSR} which we name \textit{specular diffusion equation}. First, we establish properties of the specular diffusion operator $\DelsSR$. Namely, in the half-space we show that it can be written as a kernel operator with a symmetric kernel:
\begin{equation} \label{def:kernelDelSR}
\DelsSR \psi(x) = P.V. \underset{\Omega}{\int} \big( \psi(x) - \psi(y)\big) K_\Omega(x,y) \d y \hspace{0.5cm} \mbox{with} \hspace{0.3cm} K_\Omega(x,y)=K_\Omega(y,x).
\end{equation} 
and such that the kernel is $2s$-singular. Then, in both the half-space and the ball, we show that the operator is symmetric and admits a integration by parts formula. From this formula we derive a scalar product and defined the associated Hilbert space $\HSRs(\Omega)$ in the spirit of the fractional Sobolev spaces in their relation with the fractional Laplacian operators as is presented for instance in \cite{DiNezza+}. We conclude this paper by studying the specular diffusion equation in this setting:
\begin{thm} \label{thm:wellposedness}
Let $\Omega$ be a half-space or a ball in $\RR^d$, $u_{in}$ be in $L^2((0,T)\times\Omega)$ and $s$ be in $(0,1)$. For any $T>0$, there exists a unique weak solution $u \in L^2(0,T;\HSRs(\Omega))$ of
\begin{subequations}
\begin{align}
&\pa_t u + \DelsSR u = 0 & (t,x)\in [0,T)\times\Omega  \label{eq:FracHeatSR}\\
&u(0,x) = u_{in} (x) & x\in\Omega \label{eq:FracHeatSRit}
\end{align}
\end{subequations}
in the sense that for any $\psi\in \mathcal{C}^\infty([0,T)\times\bO)$ if $s<1/2$ and any $\psi \in \mathfrak{D}_T$ is $s\geq 1/2$, $u$ satisfies if $\Omega$ is a half-space:
\begin{equation} \label{eq:weaksolFracHeathalfspace}
\begin{aligned}
&\underset{(0,T)\times\Omega}{\iint} u \pa_t \psi \d t\d x + \underset{\Omega}{\int} u_{in}(x)\psi(0,x) \d x   \\
& -\frac{1}{2} \underset{(0,T)\times\Omega\times\Omega}{\iiint} \big( u(t,x)-u(t,y) \big)\big( \psi(t,x) - \psi(t,y) \big) K(x,y) \d t\d x \d y=0 .
\end{aligned}
\end{equation}
and if $\Omega$ is the unit ball
\begin{equation} \label{eq:weaksolFracHeatball}
\begin{aligned}
&\underset{(0,T)\times\Omega}{\iint} u \pa_t \psi \d t \d x +  \underset{\Omega}{\int} u_{in}(x)\psi(0,x) \d x \\
&-\frac{1}{2} \underset{(0,T)\times\Omega\times\RR^d}{\iiint} \big( u(t,x)-u\big(t,\eta(x,v)\big) \big)\big( \psi(t,x) - \psi\big(t,\eta(x,v)\big) \big) \frac{ \d t \d x \d v }{|v|^{d+2s}} =0 .
\end{aligned}
\end{equation}
Moreover, if $\Omega$ is a half-space or a ball, then the macroscopic density $\rho$ who satisfies \eqref{eq:FracHeatSRweak} for all $\psi\in \mathcal{C}^\infty([0,T)\times\bO)$ is $s<1/2$ and any $\psi \in \mathfrak{D}_T(\Omega)$ if $s\geq 1/2$, is the unique weak solution of \eqref{eq:FracHeatSR}-\eqref{eq:FracHeatSRit}.
\end{thm}
This theorem highlights the fact that the interaction with the boundary in \eqref{eq:FracHeatSR}-\eqref{eq:FracHeatSRit} is contained in the definition of the diffusion operator $\DelsSR$ since we don't need to add a boundary condition in order to have well-posedness. \\
Here again, although we only look at the half-space and the ball, other geometries can be handled by our method such as a stripe or a cube for example. Furthermore, the only obstacle to considering more general domains lies in understanding the function $\eta$ is those domains in order to establish the symmetry of the specular diffusion operator and estimates on its singularity.

\section{A priori estimates} \label{section:apriori}

In order to study the asymptotic behaviour of the weak solution of \eqref{eq:vlfp}-\eqref{eq:vlfpid} with \eqref{eq:Aop} or \eqref{eq:SRop} boundary condition, we need a priori estimates. Those estimates will rely on the following dissipation property of the fraction Fokker-Planck operator $\mathcal{L}^s$
\begin{prop} \label{prop:diss}
For all $f$ smooth enough, if we define the dissipation as:
\begin{equation}\label{def:diss}
\mathcal{D}^s (f) := - \underset{\RR^d}{\int} \mathcal{L}^s(f) \frac{f}{F}\d v 
\end{equation}
then there exists $\theta >0$ such that
\begin{equation}\label{eq:disscontrol}
\mathcal{D}^s(f) = \underset{\RR^d\times\RR^d}{\iint} \frac{\big( f(v) - f(w) \big)^2}{|v-w|^{d+2s}} \frac{\d v \d w}{F(v)} \geq \theta \underset{\RR^d}{\int} \big| f(v) - \rho F(v) \big|^2 \frac{\d v}{F(v)}
\end{equation}
where $\rho=\iRd f(v) \d v$. Note, in particular, that $\mathcal{D}^s(f) \geq 0$.
\end{prop}
\begin{proof}
We introduce the notation $g=f/F(v)$ and notice by expending the divergence and integrating by parts that:\\
\begin{equation*}
\iRn \na_v \cdot(vFg) g \d v = \frac{1}{2} \iRn \na_v \cdot(vF)g^2 \d v.
\end{equation*}
We recall that $F$ satisfies $\mathcal{L}^s(F) = 0$, which means $\na_v \cdot (vF) = (-\Delta)^s (F)$. Together with the symmetry of the fractional Laplacian this yields:
\begin{align*}
\mathcal{D}^s (f)&= -\int_{\RR^d}  \Big( \na_v \cdot (vgF)g -(-\Delta)^s (gF)g \Big)  \d v \\
& = -\int_{\RR^d} \Big( \na_v \cdot ( vF) g^2/2 -(-\Delta)^s (gF)g\,\Big) \d v\\
& = -\int_{\RR^d} \Big( (-\Delta)^s(F) g^2/2 -(-\Delta)^s (gF)g \Big)\d v\\
& = \int_{\RR^d} \Big( -\frac{1}{2}F (-\Delta)^s(g^2) + F g(-\Delta)^s (g)\Big) \d v.
\end{align*}
Inputting the definition \eqref{def:Dels} of the fractional Laplacian we get:
\begin{align*}
\mathcal{D}^s (f)  &= c_s\underset{\RR^d}{\int} P.V. \underset{\RR^d}{\int} \left\{- \frac{1}{2}[ g(v)^2-g(w)^2] + g(v)^2-g(w)g(v)\right\}  \frac{F(v)}{|v-w|^{d+2s}} \d w\d v\\
& = \frac{c_s}{2} \underset{\RR^d\times\RR^d}{\iint}    F(v)  \frac{[ g(v)-g(w)]^2}{|v-w|^{d+2s}} \d v \d w.\\
& = \frac{c_s}{2}\underset{\RR^d\times\RR^d}{\iint}  \left( \frac{f(v)}{F(v)}- \frac{f(w)}{F(w)}\right)^2 \frac{F(v)}{|v-w|^{d+2s}} \d v \d w.
\end{align*}
Since $v$ and $w$ play the same role in the integral, we can write 
\begin{align*}
\mathcal{D}^s (f)  = \frac{c_s}{4}\underset{\RR^d\times\RR^d}{\iint} \Bigg[  \left( \frac{f(v)}{F(v)}- \frac{f(w)}{F(w)}\right)^2 F(v) + \left( \frac{f(v)}{F(v)}- \frac{f(w)}{F(w)}\right)^2 F(w)\Bigg] \frac{ \d v \d w }{|v-w|^{d+2s}}.
\end{align*}
Expending the integrand and grouping the terms adequately, it is not difficult to show that:
\begin{equation}
\mathcal{D}^s (f)  = \underset{\RR^d\times\RR^d}{\iint} \frac{\big( f(v) - f(w) \big)^2}{|v-w|^{d+2s}} \frac{\d v \d w}{F(v)}.
\end{equation}
Finally, the second inequality in \eqref{eq:disscontrol} comes from the modified logarithmic Sobolev inequality of Gentil-Imbert (Theorem 3 in \cite{Gentil+2008}) which we can use here because $F(v)$ is the infinitely divisible law associated with the L\'{e}vy measure $1/|v|^{d+2s}$. We refer the interested reader to \cite{AcevesCesbron} for a proof of this functional inequality in the fractional Laplacian case.
\end{proof}

The dissipation property of $\mathcal{L}^s$ allows us to prove the following:
\begin{repprop}{prop:weakCV} 
Let $f_{in}$ be in $L^2_{F^{-1}(v)} (\Omega\times\RR^d)$ and $s$ be in $(0,1)$. The weak solution $f^\eps$ of the rescaled fractional VFP equation \eqref{eq:vlfp}-\eqref{eq:vlfpid} with absorption \eqref{eq:Aop} or specular reflections \eqref{eq:SRop} on the boundary satisfies
\begin{equation} \tag{\ref{eq:weakCV}}
f^\eps (t,x,v) \rightharpoonup \rho(t,x)F(v) \mbox{ weakly in } L^\infty\big(0,T; L^2_{F^{-1}(v)} (\Omega\times\RR^d) \big)
\end{equation}
where $\rho(t,x)$ is the limit of the macroscopic densities $\rho^\eps = \iRd f^\eps \d v$.
\end{repprop}
\begin{proof}
Multiplying \eqref{eq:vlfp} by $f^\eps / F(v)$ and integrating over $x$ and $v$ one gets, after integrations by parts, for the absorption boundary condition:
\begin{align*}
\eps^{2s-1} \frac{\mbox{d}}{\d t} \underset{\Omega\times\RR^d}{\iint} \big( f^\eps \big)^2 \frac{\d x \d v}{F(v)} + \underset{\Sigma_+}{\iint} |\gamma_+ f^\eps |^2 |n(x)\cdot v| \frac{\d \sigma(x) \d v}{F(v)} + \frac{1}{\eps} \underset{\Omega}{\int} \mathcal{D}^s( f^\eps ) \d x =0
\end{align*}
and in the specular reflections case:
\begin{align*}
\eps^{2s-1} \frac{\mbox{d}}{\d t} \underset{\Omega\times\RR^d}{\iint} \big( f^\eps \big)^2 \frac{\d x \d v}{F(v)}  + \frac{1}{\eps} \underset{\Omega}{\int} \mathcal{D}^s( f^\eps ) \d x =0.
\end{align*}
In both cases, since the dissipation in non-negative, we see that
\begin{align*}
\frac{\text{d}}{\d t} \lVert f^\eps \lVert_{L^2_{F^{-1}(v)} (\Omega\times\RR^d)} \leq 0
\end{align*}
so $f^\eps(t,\cdot,\cdot)$ is bounded in $L^2_{F^{-1}(v)} (\Omega\times\RR^d)$. Moreover, we have
\begin{align*}
\underset{(0,T)\times\Omega}{\iint} \mathcal{D}^s (f^\eps) \d t \d x &\leq \eps^{2s} \Big( \lVert f_{in} \lVert_{L^2_{F^{-1}(v)} (\Omega\times\RR^d)} -  \lVert f^\eps (T,x,v) \lVert_{L^2_{F^{-1}(v)} (\Omega\times\RR^d)} \Big)\\
& \underset{\eps\rightarrow 0}{\longrightarrow} 0 
\end{align*}
and furthermore, by definition of $\rho^\eps$, we see that 
\begin{align*}
 \rho^\eps \leq  \bigg( \underset{\RR^d}{\int} \big(f^\eps\big)^2 \frac{\d v}{F(v)} \bigg)^{1/2} \bigg( \underset{\RR^d}{\iint} F(v) \d v \bigg)^{1/2} = \lVert f^\eps \lVert_{L^2_{F^{-1}(v)} (\RR^d)} 
\end{align*}
so that $\rho^\eps$ is also bounded in $L^\infty(0,T; L^2(\Omega))$. The boundedness of $f^\eps$ in \\
$L^\infty\big(0,T; L^2_{F^{-1}(v)} (\Omega\times\RR^d)\big)$ gives us the existence of a weak limit $\of$. Since the dissipation goes to $0$, \eqref{eq:disscontrol} implies that the limit is in the kernel of the fractional Fokker-Planck operator, i.e. there exists a function $\rho$ such that $\of(t,x,v) = \rho(t,x)F(v)$. And finally, the boundedness of $\rho^\eps$ gives us existence of a weak limit $\bar{\rho}$ and by uniqueness of the limit $\bar{\rho} = \rho$, which concludes the proof.
\end{proof}

\section{Absorption in a smooth convex domain} \label{section:abs}

We focus in this section on the absorption boundary condition \eqref{eq:Aop} and show how we can easily adapt the method developed in \cite{Cesbron12} for the anomalous diffusion limit of the fractional Vlasov-Fokker-Planck equation to this bounded domain case. \\
According to Definition \ref{def:weaksolDir}, if $f_\eps$ is a weak solution of the rescaled equation \eqref{eq:vlfp}-\eqref{eq:vlfpid} with absorption \eqref{eq:Aop} on the boundary then for all $\phi$ satisfying \eqref{eq:weakphiA} we have
\begin{subequations} 
\begin{align}
&\underset{Q_T}{\iiint} f^\eps \Big( \eps^{2s-1}\pa_t \phi - \eps^{-1}\Delsv \phi \Big) \d t \d x\d v \label{eq:wfAmain}\\
&+ \underset{Q_T}{\iiint} f^\eps \Big( v \cdot \na_x \phi - \eps^{-1} v\cdot \na_v \phi \Big) \d t\d x\d v \label{eq:wfAcc}\\
&+ \eps^{2s-1}\underset{\Omega\times\RR^d}{\iint} f_{in} (x,v) \phi (0,x,v) \d x\d v =0. \label{eq:wfAit}
\end{align}
\end{subequations}
We recognize, in \eqref{eq:wfAcc}, the {\it characteristic lines} of \eqref{eq:Fourier}. In order to take advantage of the scalar-hyperbolic structure of \eqref{eq:Fourier} we want to consider test functions which are constant along those lines. This is the purpose of the auxiliary problem.

\subsection{Auxiliary problem}

In the absorption case, it is rather simple to adapt the auxiliary problem introduced in \cite{Cesbron12} to the domain $\Omega$. For any $\psi \in\mathcal{D}([0,T)\times\Omega)$ we introduce the auxiliary problem:
\begin{subequations}
\begin{align}
&\eps v\cdot \na_x\phi^\eps - v\cdot \na_v \phi^\eps = 0 \hspace{1cm} &\forall (t,x,v) \in \RR^+\times\Omega\times\RR^d, \tag{\ref{eq:APabsmain}}\\
&\phi^\eps(t,x,0) = \psi(t,x) &\forall (t,x)\in\RR^+\times\Omega, \tag{\ref{eq:APabsit}}\\
&\gamma_+ \phi^\eps(t,x,v) = 0 &\forall (t,x,v)\in\RR^+\times\Sigma_+.\tag{\ref{eq:APabsbc}}
\end{align}
\end{subequations}
Since the boundary condition \eqref{eq:APabsbc} is immediately compatible with the assumption of compact support in $\Omega$ for the test function $\psi$, the construction of the solution $\phi_\eps$ is rather straightforward:
\begin{prop}
For any $\psi\in\mathcal{D}([0,T)\times\Omega)$, $\phi^\eps$ defined as:
\begin{equation*}
\phi^\eps(t,x,v) =\bar{\psi}(t,x+\eps v)
\end{equation*}
where $\bar{\psi}$ is the extension of $\psi$ by $0$ outside $\Omega$, is a solution of \eqref{eq:APabsmain}-\eqref{eq:APabsit}-\eqref{eq:APabsbc}.
\end{prop}
\begin{proof}
The proof is almost immediate. For \eqref{eq:APabsmain} we write:
\begin{align*}
\eps v\cdot \na_x\phi^\eps - v\cdot \na_v \phi^\eps &= \eps v\cdot \na_x[\bar{\psi}(t,x+\eps v)] - v\cdot \na_v[\bar{\psi}(t,x+\eps v)]\\
&= \eps v\cdot \na\bar{\psi}(t,x+\eps v)- \eps v\cdot \na\bar{\psi} (t,x+\eps v)=0.
\end{align*}
Moreover, the definition of $\phi^\eps$ ensures \eqref{eq:APabsit} and, thanks to the compact support of $\psi$ in $\Omega$ we also see that $\phi^\eps(t,x,v)=0$ for any $(x,v)\in\Sigma_+$ since it means that $x+\eps v \notin \Omega$.
\end{proof}
For such a $\phi^\eps$ we see that:
\begin{equation} \label{eq:scalingDels}
\begin{aligned}
\Delsv \phi^\eps(t,x,v) &= c_{d,s} P.V. \underset{\RR^d}{\int} \frac{\phi^\eps(t,x,v)-\phi^\eps(t,x,w)}{|v-w|^{d+2s}} \d w \\\
&= c_{d,s} P.V.\underset{\RR^d}{\int} \frac{\bar{\psi}(t,x+\eps v)-\bar{\psi}(t,x+\eps w)}{|v-w|^{d+2s}} \d w \\
&= c_{d,s} P.V.\underset{\RR^d}{\int} \frac{\bar{\psi}(t,x+\eps v)-\bar{\psi}(t,w)}{\eps^{-d-2s}|x+\eps v-w|^{d+2s}} \eps^{-d} \d w \\
&= \eps^{2s} \Dels\bar{\psi}(t,x+\eps v) 
\end{aligned}
\end{equation}
so that the weak formulation \eqref{eq:wfAmain}-\eqref{eq:wfAcc}-\eqref{eq:wfAit} becomes
\begin{equation}\label{eq:wfAbs}
\underset{Q_T}{\iiint} f^\eps \bigg( \pa_t \bar{\psi} - \Dels \bar{\psi}(t,x+\eps v)\bigg) \d t \d x \d v + \underset{\Omega\times\RR^d}{\iint} f_{in} (x,v) \bar{\psi}\big(0,x+\eps v\big) \d x \d v =0. 
\end{equation}
\subsection{Macroscopic Limit} \label{sec:MLabs}

In Section \ref{section:apriori} we proved that $f^\eps$ converges weakly in $L^\infty\big(0,T; L^2_{F^{-1}(v)} (\Omega\times\RR^d) \big)$. Hence, in order to pass to the limit in the weak formulation \eqref{eq:wfAbs} we need to show that
\begin{align} \label{eq:strongcvabs}
\pa_t \bar{\psi}(t,x+\eps v) - \Dels \bar{\psi}(t,x+\eps v) \underset{\eps \rightarrow 0}{\longrightarrow} \pa_t \bar\psi(t,x)-\Dels \bar\psi(t,x) 
\end{align} 
at least strongly in $L^\infty\big(0,T; L^2_{F(v)} (\Omega\times\RR^d) \big)$. The proof of this convergence is rather similar to its equivalent in the unbounded case presented in \cite{Cesbron12}. As a consequence we will not give any unnecessary details and instead we briefly recall the main arguments. First, we note that the continuity of $\bar\psi$ readily implies the convergence of the second term in \eqref{eq:wfAbs}:
\begin{align*}
\underset{\Omega\times\RR^d}{\iint} f_{in} (x,v) \bar{\psi}\big(0,x+\eps v\big) \d x \d v \underset{\eps \rightarrow 0}{\longrightarrow} \underset{\Omega}{\int} \rho_{in}(x) \bar\psi(0,x) \d x  .
\end{align*}
Secondly, the strong convergence of \eqref{eq:strongcvabs} follows from the fact that if $\bar\psi$ is in $\mathcal{D}([0,T)\times\Omega)$ then
\begin{align*}
\pa_t \bar\psi \in \mathcal{D}([0,T)\times\Omega) \hspace{0.5cm} \text{and} \hspace{0.5cm} \Dels \bar\psi \in \mathcal{D}([0,T)\times\RR^d)\cap L^2([0,T)\times\RR^d)
\end{align*}
because the pseudo-differential operator $\Dels$ can be defined as an operator from the Schwartz space to $L^2(\RR^d)$, see e.g. Proposition 3.3 in \cite{DiNezza+}. As a consequence, it is straightforward to use dominated convergence on both terms and prove the strong convergence of \eqref{eq:strongcvabs} in $L^\infty\big(0,T; L^2_{F(v)} (\Omega\times\RR^d) \big)$, noticing that $\int F(v) \d v = 1$. \\
Hence, we can take the limit in the weak formulation and find that $\rho$ satisfies:
\begin{equation} \label{eq:wfFracHeatAbs}
\underset{(0,T)\times\Omega}{\iint} \rho(t,x) \big( \pa_t \psi(t,x) - \Dels \psi(t,x)\big) \d t \d x + \underset{\Omega}{\int} \rho_{in}(x) \psi(0,x) \d x = 0.
\end{equation}
Since $\rho$ is the limit of $\rho^\eps$ it is only defined on $\Omega$. If we extend it by $0$ on the complementary $\RR^d\setminus\Omega$, then we can integrate over $\RR^d$ instead of $\Omega$ and that concludes the proof of Theorem \ref{thm:mainA}.

\section{Specular Reflection in a bounded domain} \label{section:sr}

We now turn to the more challenging case of the specular reflection boundary conditions \eqref{eq:SRop}. From Definition \ref{def:weaksolSR} we know that if $f_\eps$ is a weak solution of fractional Vlasov-Fokker-Planck equation with specular reflection on the boundary \eqref{eq:vlfp}-\eqref{eq:vlfpid}-\eqref{eq:SRop} then for any $\phi$ satisfying 
\begin{equation}  \tag{\ref{eq:weakphiSR}} 
\begin{aligned}
&\phi \in C^\infty ( Q_T ) \hspace{1cm} \phi(T,\cdot,\cdot)=0 \\
&\gamma_+\phi(t,x,v) = \gamma_-\phi\big(t,x,\mathcal{R}_x(v)\big) \hspace{1cm} \forall (t,x,v)\in [0,T)\times\Sigma_+
\end{aligned}
\end{equation}
we have, analogously to the absorption case:
\begin{align}
&\underset{Q_T}{\iiint} f^\eps \Big( \eps^{2s-1}\pa_t \phi - \eps^{-1}\Delsv \phi \Big) \d t \d x\d v \tag{\ref{eq:wfAmain}} \\
&+ \underset{Q_T}{\iiint} f^\eps \Big( v \cdot \na_x \phi - \eps^{-1} v\cdot \na_v \phi \Big) \d t\d x\d v \tag{\ref{eq:wfAcc}}\\
&+ \eps^{2s-1}\underset{\Omega\times\RR^d}{\iint} f_{in} (x,v) \phi (0,x,v) \d x\d v =0. \tag{\ref{eq:wfAit}}
\end{align}
Once again, we would like to take advantage of the scalar-hyperbolic structure of \eqref{eq:Fourier} in order to define a sub-class of test function $\phi$ that will allow us to identify the anomalous diffusion limit of this equation. This is the purpose of the following auxiliary problem. 

\subsection{Auxiliary problem} \label{sec:AuxProbSR}

For a smooth function $\psi$, we define $\phi_\eps$ as the solution of 
\begin{subequations}
\begin{align}
&\eps v\cdot \na_x\phi^\eps - v\cdot \na_v \phi^\eps = 0 \hspace{1cm} &\forall (t,x,v) \in \RR^+\times\Omega\times\RR^d, \tag{\ref{eq:APSRmain}}\\
&\phi^\eps(t,x,0) = \psi(t,x) &\forall (t,x)\in\RR^+\times\Omega, \tag{\ref{eq:APSRit}}\\
&\gamma_+ \phi^\eps(t,x,v) = \gamma_-\phi^\eps\big(t,x,\mathcal{R}_x(v)\big) &\forall (t,x,v)\in\RR^+\times\Sigma_+.\tag{\ref{eq:APSRbc}}
\end{align}
\end{subequations}
with $\mathcal{R}_x(v) = v - 2 \big(n(x)\cdot v \big)n(x)$ for $x$ in $\dO$. \\
Because of the specular reflection boundary condition \eqref{eq:APSRbc}, it is much more challenging to construct a solution $\phi_\eps$ to this problem than it was in the absorption case. In fact, we will see later on that if we want to have enough regularity estimates on $\phi_\eps$ in order to take the limit in the weak formulation of the fractional Vlasov-Fokker-Planck equation, we will need an additional assumption on the initial condition $\psi$. Setting aside these considerations for the moment, let us show how we can construct $\phi_\eps$ from a smooth function $\psi$ through the definition of a function $\eta : \Omega\times\RR^d \mapsto \bO$ in the following sense:
\begin{repprop}{prop:solAPSR}
If $\Omega$ is either a half-space or smooth and strongly convex, then there exists a function $\eta : \Omega\times\RR^d \rightarrow \bO$ such that 
\begin{equation}
\phi^\eps (t,x,v) = \psi\big(t,\eta(x,\eps v) \big)
\end{equation}
is a solution of the auxiliary problem \eqref{eq:APSRmain}-\eqref{eq:APSRit}-\eqref{eq:APSRbc}.
\end{repprop}
\begin{proof}
The proof will consist of two steps. First we construct an appropriate $\eta$ by identifying the characteristic lines underlying the hyperbolic problem \eqref{eq:APSRmain}-\eqref{eq:APSRbc}, and then we check that $\phi^\eps$ defined as above is indeed solution of the auxiliary problem. 

\subsubsection{Construction of $\eta$}

The purpose of $\eta$ is to follow the characteristic lines defined by \eqref{eq:APSRmain} and \eqref{eq:APSRbc}. Those lines $(x(s),v(s))$, parametrised by $s\in[0,\infty)$, are given by:
\begin{equation}\tag{\ref{eq:cc}}
  \left\{
      \begin{array}{llll}
        &\dot{x}(s) = \eps v(s) \hspace{15mm} &x(0)=x^{in},  \\
        &\dot{v}(s) = -v(s) &v(0)=v^{in},\\
        & \text{If } x(s)\in\dO \text{ then } v(s^+)= \mathcal{R}_{x(s)}(v(s^{-})).
      \end{array}
    \right.
\end{equation}
Solving this system of ODEs, we see that this trajectory $x(s)$ consists of straight lines with exponentially decreasing velocity $v(s)$ reflected upon hitting the boundary. More precisely, if we denote $s_i$ the times of reflection, i.e. the times for which $x(s_i)\in\dO$, with the convention $s_0=0$, we have for the velocity:
\begin{equation}
  \left\{
      \begin{aligned}
        &v(s) = e^{-s} v_0 &\text{ for } s\in[0,s_1),\\
		&v(s_i^+) = \mathcal{R}_{x(s_i)} v(s_i^-),\\        
        &v(s) = e^{-(s-s_i)} v(s_i^+)  &\text{ for } s\in(s_i,s_{i+1}),
      \end{aligned}
    \right.
\end{equation}
which gives the trajectory, for $s\in(s_i,s_{i+1})$:
\begin{align*}
x(s) &= x_0 + \eps \int_0^s v(\tau) \text{d}\tau\\
	 &= x_0 + \eps \underset{k=0}{\overset{i-1}{\sum}} \int_{s_{k}}^{s_{k+1}} v(\tau) \text{d}\tau + \eps \int_{s_{i}}^{s} v(\tau) \text{d}\tau  \\
	 &= x_0 + \eps \underset{k=0}{\overset{i-1}{\sum}} \left(1-e^{-(s_{k+1}-s_k)}\right) v(s_k^+) + \eps \left(1-e^{-(s-s_k)}\right)v(s_i^+).
\end{align*}
Instead of considering an exponentially decreasing velocity $v(s)$ on an infinite interval $s\in[0,\infty)$, we would like to consider a constant speed on a finite interval $[0,1)$. To that end, we notice that the reflection operator $\mathcal{R}$ is isometric in the sense that:
\begin{align*}
v(s_i^+) &= \mathcal{R}_{x(s_i)}\big( v(s_i^-)\big)\\
		  &= \mathcal{R}_{x(s_i)} \big( e^{-(s_i-s_{i-1})} v(s_{i-1}^+) \big)\\
		  &= e^{-(s_i-s_{i-1})} \mathcal{R}_{x(s_i)} \circ \mathcal{R}_{x(s_{i-1})} \big( e^{-(s_{i-1}-s_{s-2})} v(s_{i-2}^+) \big)\\
		  &= e^{-(s_i-s_{i-2})} \mathcal{R}_{x(s_i)} \circ \mathcal{R}_{x(s_{i-1})} \circ \mathcal{R}_{x(s_{i-2})} \big( e^{-(s_{i-2}-s_{s-3})} v(s_{i-3}^+)\big) \\
		  &= e^{-(s_i-0)} \mathcal{R}_{x(s_i)} \circ \mathcal{R}_{x(s_{i-1})} \circ \dots \circ \mathcal{R}_{x(s_1)} \big( v_0 \big).
\end{align*}
Furthermore, we introduce the notation $R^i$ denoting: 
\begin{equation} \label{def:Ri}
  \left\{
      \begin{aligned}
        & R^0 = Id,\\
        & R^i= \mathcal{R}_{x(s_i)} \circ R^{i-1},
      \end{aligned}
    \right.
\end{equation}
and a new velocity $w(s):= e^s v(s)$ which then satisfies:
\begin{equation}
  \left\{
      \begin{aligned}
        &w(s) = v_0 &\text{ for } s\in(0,s_1),\\
		& w(s_i) = R^i v_0,\\        
        &w(s) = R^i w(s_i)  &\text{ for } s\in[s_i,s_{i+1}).
      \end{aligned}
    \right.
\end{equation}
It is easy to check that for any $s$, $|w(s)| = |v_0|$. The trajectory $x(s)$ can be written, with the velocity $w(s)$ as:
\begin{align*}
x(s) &= x_0 + \eps \int_0^s e^{-\tau} w(\tau) \text{d}\tau\\
	 &= x_0 + \eps \underset{k=0}{\overset{i-1}{\sum}} \left(e^{-s_k}-e^{-s_{k+1}}\right) w(s_k) + \eps \left(e^{-s}-e^{-s_i}\right)w(s_i).
\end{align*}
Finally, we introduce a new parametrisation $\tau = 1 - e^{-s} \in [0,1)$ and the corresponding reflection times $\tau_i := 1-e^{-s_i}$ with which we have, for any $\tau \in [\tau_i,\tau_{i+1})$ with $i\geq 1$:
\begin{equation} \label{eq:etaexplicit}
  \left\{
      \begin{aligned}
        & x(\tau) = x_0 + \eps \underset{k=0}{\overset{i-1}{\sum}} \left(\tau_{k+1}-\tau_{k}\right) w(\tau_k) + \eps \left(\tau-\tau_i\right)w(\tau_i),\\
        & w(\tau) = w(\tau_i) = R^i w_0.
      \end{aligned}
    \right.
\end{equation}

\begin{figure}[h] 
\centering
\includegraphics[width=11cm,height=11cm]{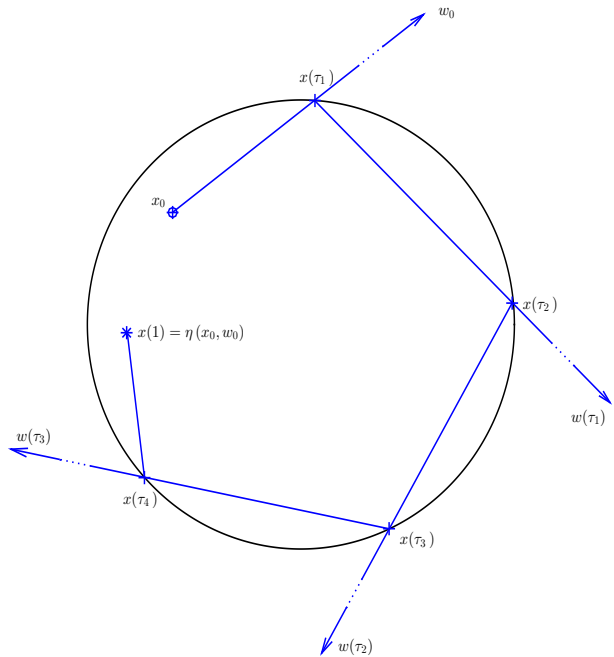}
\caption{Example of trajectory of $\Omega$ is a disk}
\label{fig:extraj}
\end{figure}

These trajectories can be seen as geodesic trajectory in a Hamiltonian billiard, as illustrated by Figure \ref{fig:extraj}. In order to solve \eqref{eq:APSRmain}-\eqref{eq:APSRbc} using a characteristic method we would like to define a function $\eta^\eps$ that relates $(x_0,w_0)$ to $x(\tau$=$1)$ (or $x(s$=$\infty)$ for the initial parametrization). It is natural to construct $\eta^\eps$ by induction on the number of reflections. Such a construction is already well known in the field of mathematical billiards. We refer for instance to the Chapter 2 of the monograph of Chernov and Markarian \cite{ChernovMarkarian} for the construction in dimension 2 and the paper of Halpern \cite{Halpern} where he defines a function $F_t(x,v)$ which gives the position and forward direction of motion of a particle in the billiard, in relation to which our $\eta^\eps(x,v)$ is just the first component of $F_{t=\eps}(x,v)$. To make sure $F_t$, hence $\eta^\eps$, is well defined, we just need to make sure that there are no accumulations of reflection times, i.e. that there is only a finite number of reflections occurring during a finite time interval. To that end, we consider the point on the boundary at which these accumulations would happen. Chernov and Markarian explain that it cannot happen on a flat surface and, moreover, in dimension two, Halpern gives a result which can be stated as follows
\begin{thm*}
Let us call $\zeta$ the function such that 
$$ \Omega = \lbrace x\in\RR^d / \zeta(x)<0 \rbrace \hspace{0.2cm} \text{and} \hspace{0.2cm} \dO =\lbrace x\in\RR^d / \zeta(x)=0 \rbrace.$$
If $\zeta$ has a bounded third derivative and nowhere vanishing curvature on $\dO$ in the sense that there exists a constant $C_\zeta >0$ such that for all $\xi\in\RR^d$:
\begin{align*}
\underset{i,j=1}{\overset{d}{\sum}} \xi_i \frac{\pa^2 \zeta}{\pa x_i \pa x_j} \xi_j \geq C_\zeta |\xi|^2
\end{align*} 
then $F_t(x,v)$ is well defined for all $(x,v)\in\Omega\times\RR^d$.
\end{thm*}
We call \textit{strongly convex} such domains. This result was later extended by Safarov and Vassilev to higher dimension as stated in Lemma 1.3.17 of \cite{SafarovVassilev}. We will consider $\Omega$ to be a half-space or a ball, neither of which allows for the accumulation of reflection times hence $\eta^\eps$ can be defined as:
\begin{equation}
\eta^\eps(x_0,w_0) = x(\tau\text{=}1) = x_0 + \eps \underset{k=0}{\overset{M-1}{\sum}} \left(\tau_{k+1}-\tau_{k}\right) w(\tau_k) + \eps \left(1-\tau_M\right)w(\tau_M)
\end{equation}
where $M=M(x_0,w_0)$ is the (finite) number of reflections undergone by the trajectory that starts at $(x_0,w_0)$. Note that this expression yields immediately that for any $(x,v)\in\Omega\times\RR^d$:
$$ \eta^\eps(x,v) = \eta^1(x,\eps v)$$
so that, from now on, we will forgo the superscript $1$ and always consider $\eta(x,\eps v)$. 
\begin{rmq}
Note that in general domain, possibly non-convex, it has been proved by Briant in the appendix of \cite{Briant2015} and by Kim and Lee in \cite{Kim2017} for non-convex cylindrical domains, that the set of all $(x,v)$ from which the trajectory described above undergoes infinitely many reflections in finite time is of measure zero in the phase space so in general domains $\eta$ is well defined almost-everywhere. 
\end{rmq}

\subsubsection{$\phi^\eps$ solution of the auxiliary problem} 

We now define, for any given smooth function $\psi$:
$$ \phi^\eps(t,x,v) = \psi\big(t,\eta(x,\eps v)\big).$$
By construction, we know that $\phi^\eps$ satisfies \eqref{eq:APSRit} and \eqref{eq:APSRbc}. For \eqref{eq:APSRmain} we differentiate along the characteristic curves:
$$\frac{d}{ds} \phi^\eps (t,x(s),v(s)) = \frac{d}{ds }\psi \Big(t, \eta\big(x(0),\eps v(0)\big) \Big) = 0$$
which yields by \eqref{eq:cc}
\begin{align*}
\dot{x}(s) \cdot \na_x \phi^\eps(x(s),v(s)) + \dot{v}(s) \cdot\na_v\phi^\eps(x(s),v(s)) &= 0\\
\eps v(s) \cdot\na_x \phi^\eps(x(s),v(s)) - v(s) \cdot\na_v \phi^\eps(x(s),v(s)) &=0. 
\end{align*}
Take $s=0$ and you get:
$$ \eps v \cdot\na_x \phi^\eps(x,v) - v \cdot\na_v \phi^\eps(x,v) = 0$$
which concludes the proof of Proposition \ref{prop:solAPSR}.
\end{proof}

The solution $\phi^\eps$ has a scaling property similar to \eqref{eq:scalingDels} for the solution of the auxiliary problem in the absorption case, namely :
\begin{align*}
\Delsv \Big[ \phi^\eps(t,x,v)\Big] &= c_{d,s} P.V. \iRn \frac{\psi\big(t,\eta(x,\eps v)\big)-\psi\big(t,\eta(x,\eps w)\big)}{|v-w|^{N+2s}} \d w\\
&= \eps^{2s} c_{d,s} P.V. \iRn \frac{\psi\big(t,\eta(x,\eps v)\big)-\psi\big(t,\eta(x,w)\big)}{|\eps v-w|^{N+2s}} \d w\\
&= \eps^{2s} \Delsv \Big[ \psi\big( t,\eta(x, \cdot)\big)\Big] (\eps v)
\end{align*}
Hence, the weak formulation of \eqref{eq:vlfp}-\eqref{eq:vlfpid}-\eqref{eq:SRop} becomes:
\begin{equation} \label{eq:wfSReta}
\begin{aligned}
&\underset{Q_T}{\iiint} f^\eps \bigg(  \pa_t \psi -   \Delsv \Big[ \psi\big( t,\eta(x, \cdot)\big)\Big] (\eps v) \bigg) \d t \d x \d v \\
&\hspace{1.5cm} + \underset{\Omega\times\RR^d}{\iint} f_{in} (x,v) \psi\big(0,\eta(x,\eps v)\big) \d x \d v =0. 
\end{aligned}
\end{equation}

\subsection{Macroscopic limit}

Using the same arguments as in the unbounded or the absorption case, one can show that if $\psi \in \mathcal{D}([0,T)\times\bO)$ then
\begin{align*}
\lim_{\eps \searrow 0} \underset{Q_T}{\iiint} f^\eps \pa_t \psi\big(t,\eta(x,\eps v\big) \d t\d x \d v = \underset{(0,T)\times\Omega}{\iint} \rho(t,x)\psi(t,x) \d t \d x
\end{align*}
and 
\begin{align*}
\lim_{\eps \searrow 0} \underset{\Omega\times\RR^N}{\iint} f_{in} (x,v) \phi^\eps(0,x,v) \d x \d v = \underset{\Omega}{\int} \rho_{in}(x)  \psi(0,x) \d x.
\end{align*}
For the last term, we prove the following Lemma:
\begin{lemma} \label{lemma:CVSR}
If $\Omega$ is a half-space or a ball in $\RR^d$ then for any $\psi\in \mathcal{C}^\infty_c(Q_T)$ if $s<1/2$ and any $\psi \in \mathfrak{D}_T(\Omega)$ if $s\geq 1/2$ where we recall that $\mathfrak{D}_T$ is defined as 
\begin{equation} \tag{\ref{eq:defDT}}
\mathfrak{D}_T(\Omega) = \Big\{ \psi\in\mathcal{C}^\infty ([0,T)\times\bO) \text{ s.t. } \psi(T,\cdot)=0 \text{ and } \na_x \psi(t,x)\cdot n(x) =0 \text{ on } \dO \Big\},
\end{equation} 
we have 
\begin{equation} \label{eq:CVSR}
\begin{aligned}
\lim_{\eps \searrow 0} \underset{Q_T}{\iiint} f^\eps  \Delsv &\Big[ \psi\big( t,\eta(x, \cdot)\big)\Big] (\eps v) \d t \d x \d v \\
&=  \underset{(0,T)\times\Omega}{\iint} \rho(t,x)\DelsSR \psi(t,x) \d t \d x 
\end{aligned}
\end{equation}
where $\DelsSR$ is given in Definition \ref{def:DelsSR} and can equivalently be written as:
\begin{equation}
\DelsSR \psi (t,x ) = \Delsv \Big[ \psi\big( t,\eta(x, \cdot)\big)\Big] (0).
\end{equation}
\end{lemma}

Before proving this lemma, which we will do separately for each $\Omega$, let us conclude that with this convergence we can take the limit in \eqref{eq:wfSReta} and see that the macroscopic density $\rho(t,x)$ satisfies
\begin{equation} \label{eq:weakFHSR}
\underset{(0,T)\times\Omega}{\iint} \rho(t,x) \Big( \pa_t \psi(t,x) - \DelsSR \psi(t,x) \Big) \d t \d x + \underset{\Omega}{\int} \rho_{in}(x)\psi(0,x) \d x = 0.
\end{equation}
for any $\psi\in \mathcal{C}^\infty_c(Q_T)$ if $s<1/2$ and any $\psi \in \mathfrak{D}_T(\Omega)$ if $s\geq 1/2$, which ends the proof of Theorem \ref{thm:mainSR}.

\subsubsection{Lemma \ref{lemma:CVSR} in a half-space} \label{subsubsec:lemmaCVSRhalfspace}
Consider the half-space $\lbrace x=(x',x_d) \in \RR^d : x_d >0 \rbrace$. We will focus on the case $s\geq 1/2$ because, as will be explained in Remark \ref{rmq:spetit}, the case $s<1/2$ can be handled by a simpler version of the same proof.\\
The function $\eta$ associated with the half-space can be written explicitly as:
\begin{equation} \label{eq:etaexplicitHS}
\eta(x,v) = \left|  \begin{aligned}  &x+v &\mbox{if } x_d+v_d \geq 0\\
													 & (x'+v' , -x_d -v_d ) &\mbox{if } x_d+v_d \leq 0
							\end{aligned} \right.										 
\end{equation} 
as illustrated by Figure \ref{fig:halfspace}. 
\begin{figure}[h]
\centering
\includegraphics[width=11cm,height=10cm]{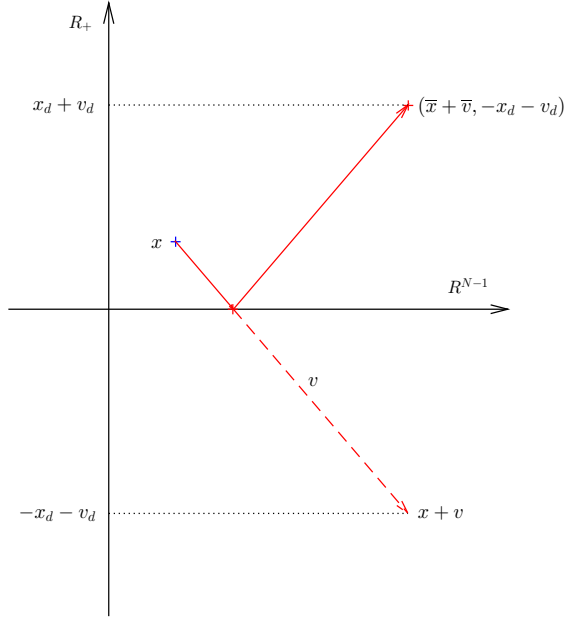}
\caption{Example of trajectory in the half-space}
\label{fig:halfspace}
\end{figure}

We can differentiate $\eta(x,v)$ to see that its Jacobian matrix reads
\begin{align} \label{eq:naetahalfspace}
\na_v \eta(x,v) = Id + \big(H(x_d+v_d)-1\big) E_{d,d} 
\end{align}
where $E_{d,d}$ is the matrix with $0$ everywhere except the last coefficient (of index $d,d$) which is $1$ and $H$ is the Heaviside function equal to $1$ if $x_d+v_d>0$ and $-1$ if $x_d+v_d<0$. Furthermore, the second derivative of $\eta(x,v)$, which we will see as an element of $\mathcal{M}_d (\RR^d)$, i.e. a vector valued matrix, reads
\begin{align*}
D^2_v \eta(x,v) = 2 \big( n \times E_{d,d}\big) \delta_{\eta(x,v)\in\dO}
\end{align*}
where $n$ is the outward unit vector of $\dO$ (which is constant in the half-space), $\delta_{\eta(x,v)\in\dO}$ is the dirac measure of the boundary surface and $\times$ is a multiplication between a vector $u\in \RR^d$ and a matrix $M= (m_{i,j})_{1\leq i,j\leq d} \in \mathcal{M}_d(\RR)$ whose result is the vector-valued matrix given by $u\times M = (  m_{i,j} u )_{1\leq i,j\leq d} \in \mathcal{M}_d(\RR^d)$. \\
A straightforward differentiation yields
\begin{align*}
D_v^2 \Big[ \psi\big(t,\eta(x,v)\big)\Big] &= \big( \na_v \eta(x,v) \big)^T D^2 \psi\big(t,\eta(x,v)\big) \big( \na_v \eta(x,v) \big) \\
&\quad + D_v^2\eta(x,v) \na \psi\big(t,\eta(x,v)\big).
\end{align*}
where for any $\psi\in \mathfrak{D}_T$ we have
\begin{align*}
D_v^2\eta(x,v) \na \psi\big(t,\eta(x,v)\big) = 2 \Big( n \cdot \na \psi\big(t,\eta(x,v)\big) \Big) E_{d,d} \delta_{\eta(x,v)\in\dO}= 0
\end{align*}
since for all $y=\eta(x,v) \in\dO$ we have $ n(y) \cdot \na \psi(t,y) = 0$. \\
To prove Lemma \ref{lemma:CVSR} we will show that $\Delsv \Big[ \psi\big( t,\eta(x, \cdot)\big)\Big] (\eps v)$ converges strongly in $L^\infty( 0,T ; L^2_{F(v)}(\Omega\times\RR^d)$ by a dominated convergence argument. Since $f^\eps$converges weakly in $L^\infty( 0,T ; L^2_{F^{-1}(v)}(\Omega\times\RR^d)$ we can then pass to the limit in the left-hand-side of \eqref{eq:CVSR} and Lemma \ref{lemma:CVSR} follows. \\
We begin by the proof of point-wise convergence. We introduce the function $\chi_x:\RR^d\times\RR^d \mapsto \RR$ given by (omitting the $t$ variable for the sake of clarity)
\begin{align} \label{eq:defchi}
\chi_x(v,w) = \psi\big(\eta(x,v+w)\big) - \psi\big(\eta(x,w)\big).
\end{align}
For any $(t,x,v) \in Q_T$ we then have
\begin{align}
 \Delsv &\Big[ \psi\big(t,\eta(x,\cdot)\big)\Big]  (\eps v) -  \DelsSR \psi(x)  \nonumber \\
 &= c_{d,s} P.V. \iRn \frac{\psi\big(t,\eta(x,\eps v)\big)-\psi\big(t,\eta(x, \eps v + w)\big)}{|w|^{N+2s}} \d w \nonumber \\
&\quad -  c_{d,s} P.V. \iRn \frac{\psi\big(t,x)-\psi\big(t,\eta(x, w)\big)}{|w|^{N+2s}} \d w \nonumber\\
 &= c_{d,s} P.V. \iRn \frac{\chi_x(\eps v, 0) - \chi_x (\eps v , w)}{|w|^{d+2s}} \d w. \label{eq:delschix}
\end{align}
For $\delta >0$, we split the integral as follow 
\begin{align*}
 c_{d,s} P.V. \iRn \frac{\chi_x(\eps v, 0) - \chi_x (\eps v , w)}{|w|^{d+2s}} \d w &=  c_{d,s} P.V. \int_{|w|\leq\delta} \frac{\chi_x(\eps v, 0) - \chi_x (\eps v , w)}{|w|^{d+2s}} \d w \\
 &\hspace{0.5cm} +  c_{d,s} \int_{|w|\geq\delta} \frac{\chi_x(\eps v, 0) - \chi_x (\eps v , w)}{|w|^{d+2s}} \d w.
\end{align*}
On the one hand we see that
\begin{align*}
\bigg| \int_{|w|\geq\delta} \frac{\chi_x(\eps v, 0) - \chi_x (\eps v , w)}{|w|^{d+2s}} \d w \bigg| &\leq 2 \lVert \chi_x(\eps v,\cdot) \lVert_{L^\infty(\RR^d)} \int_{|w|\geq \delta} \frac{1}{|w|^{d+2s}} \d w \\
& \leq 2 \delta^{-2s} \lVert \chi_x(\eps v,\cdot) \lVert_{L^\infty(\RR^d)}  
\end{align*}
and by definition of $\chi_x$ 
\begin{align*}
\sup_w | \chi_x(\eps v, w ) | = \sup_w \Big| \psi\big(\eta(x,\eps v+w)\big) - \psi\big(\eta(x,w)\big) \Big| \underset{\eps \rightarrow 0}{\longrightarrow} 0 
\end{align*}
so the integral over $|w|\geq \delta$ vanishes. On the other hand, using the symmetry of the set $\lbrace |w|\leq \delta \rbrace$ we write
\begin{align*}
&P.V. \int_{|w|\leq\delta} \frac{ \chi_x(\eps v,0) - \chi_x(\eps v,w) }{|w|^{N+2s}} \d w  \\
&\hspace{1cm} = \frac{1}{2} P.V. \int_{|w|\leq \delta} \frac{ 2\chi_x(\eps v, 0) - \chi_x(\eps v, w) - \chi_x (\eps v, -w) }{|w|^{d+2s}} \d w 
\end{align*}
where we can expand $\chi_x (\eps v, \pm w)$ using a second-order Taylor-Lagrange expansion which yields, for some $\theta$ and $\tilde{\theta}$ in the ball $B(\delta)$ centred at the origin with radius $\delta$
\begin{align*}
&2\chi_x(\eps v, 0) - \chi_x(\eps v, w) - \chi_x (\eps v, -w)  \\
&=-  \na_w \chi_x (\eps v, 0) \cdot w - w \cdot D^2 \chi_x (\eps v , \theta) w \\
&\hspace{1cm} - \na_w \chi_x (\eps v, 0) \cdot (-w) - (-w) \cdot D^2 \chi_x (\eps v , \tilde{\theta}) (-w)\\
&= - w \cdot \Big( D^2 \chi_x (\eps v , \theta) + D^2 \chi_x (\eps v , \tilde{\theta}) \Big) w 
\end{align*}
therefore
\begin{align} \label{eq:D2chismallw}
&\bigg|P.V. \int_{|w|\leq\delta} \frac{ \chi_x(\eps v,0) - \chi_x(\eps v,w) }{|w|^{N+2s}} \d w \bigg| \\
&\quad = \frac{1}{2}\bigg| \int_{|w|\leq \delta} \frac{w\big( D^2 \chi_x (\eps v , \theta) + D^2 \chi_x (\eps v , \tilde{\theta}) \big) w }{|w|^{d+2s}} \d w \bigg|
\end{align}
where the P.V. is not needed any more since $s<1$. For any fixed $\theta \in B(\delta)$, we have 
\begin{align*}
D^2 \chi_x (\eps v , \theta ) &= \big(  \na_v \eta(x,\eps v + \theta )\big)^T D^2 \psi\big(\eta(x,\eps v+ \theta)\big) \big( \na_v \eta(x,\eps v + \theta)\big) \\
&\hspace{1cm} - \big( \na_v \eta(x, \theta)\big)^T D^2 \psi\big(\eta(x,\theta )\big) \big( \na_v \eta(x,\theta)\big). 
\end{align*}
If $x + \eps v + \theta$ and $x+\theta$ are either both in $\Omega$ or both outside $\Omega$ then thanks to \eqref{eq:naetahalfspace} we know that $\na_v \eta(x,\eps v + \theta) = \na_v \eta(x, \theta)$. We denote $M$ this matrix and we have 
\begin{align*}
D^2 \chi_x (\eps v , \theta ) &= M^T \Big( D^2 \psi\big(\eta(x,\eps v+ \theta)\big) - D^2 \psi\big(\eta(x,\theta )\big) \Big) M 
\end{align*}
in which case the regularity of $\psi$ yields
\begin{align*}
 \lim_{\eps \rightarrow 0} D^2 \chi_x (\eps v , \theta )  =0 .
\end{align*}
If $x$ is in the interior of $\Omega$, then for $\eps$ and $\delta$ small enough, we will obviously have $x+\theta$ and $x+\eps v + \theta$ inside $\Omega$. Moreover, if $x$ is on the boundary $\dO$ then for any fixed $\theta$ in $B(\delta)$, when $\eps $ is small enough we will also have $x+\theta$ and $x+\eps v + \theta$ either both inside $\Omega$ if $\theta \cdot n(x) <0$ or outside $\Omega$ if $\theta \cdot n(x) \geq 0$. As a consequence, we have point-wise convergence of the integrand in the left side of \eqref{eq:D2chismallw} therefore \eqref{eq:naetahalfspace} and the regularity of $\psi$ ensure that we can use dominated convergence in $L^1(B(\delta))$ to write
\begin{align*}
&\lim_{\eps\rightarrow 0} \bigg|P.V. \int_{|w|\leq\delta} \frac{ \chi_x(\eps v,0) - \chi_x(\eps v,w) }{|w|^{N+2s}} \d w \bigg| \\
& \hspace{1cm} = \frac{1}{2}\bigg| \int_{|w|\leq \delta} \lim_{\eps \rightarrow 0} \frac{w\big( D^2 \chi_x (\eps v , \theta) + D^2 \chi_x (\eps v , \tilde{\theta}) \big) w }{|w|^{d+2s}} \d w \bigg| =0 .
\end{align*}
Now that we have proven the point-wise convergence, let us show that 
\begin{align*}
v \mapsto \Delsv \Big[ \psi\big(t,\eta(x,\cdot)\big)\Big] (\eps v)
\end{align*}
is bounded uniformly in $\eps$ by a function in $L^2_{F(v)}(\Omega\times\RR^d)$. The regularity of $\psi$ and the above computation of the jacobian matrix of $\eta$ yield in particular that for all $t\in [0,T)$
\begin{align} \label{eq:D2halfspacebound}
\underset{v\in\RR^d}{\sup} D^2_v \Big[ \psi\big(t,\eta(x,v)\big)\Big]  \in L^2 (\Omega).
\end{align}
Therefore, for any $t\in[0,T)$ we introduce $G_t(x)$ given by
\begin{align*}
G_t(x) =  \lVert \psi(t,\cdot) \lVert_{L^\infty(\Omega)} +  \Big\lVert D^2_v \Big[ \psi\big(t,\eta(x,\cdot)\big)\Big]\Big\lVert_{L^\infty(\RR^d)} .
\end{align*}
As we did before, we can split the integral expression of the fractional Laplacian into a integral on a ball of radius $\delta$ around the singularity and an integral on the complement of that ball. For the latter, we write for some constant $C>0$
\begin{align*}
&\bigg| c_{d,s} \underset{\RR^d\setminus B(\delta)}{\int} \frac{\psi\big(\eta(x,\eps v)\big) - \psi\big(\eta(x,\eps v + w)\big) }{|w|^{d+2s}} \d w \bigg| \\
&\quad \leq C \lVert \psi (t,\cdot) \lVert_{L^\infty(\Omega)} \underset{\RR^d\setminus B(\delta)}{\int} \frac{1}{|w|^{d+2s}} \d w \\
&\quad \leq C  \lVert \psi (t,\cdot) \lVert_{L^\infty(\Omega)} \delta^{-2s} .
\end{align*}
For the integral over $B(\delta)$, we use a second order Taylor-Lagrange expansion like we did for $\chi_x$ and write
\begin{align*}
&\bigg| c_{d,s} \underset{ B(\delta)}{\int} \frac{\psi\big(\eta(x,\eps v)\big) - \psi\big(\eta(x,\eps v + w)\big) }{|w|^{d+2s}} \d w \bigg| \\
&\quad \leq C \underset{B(\delta)}{\int} \frac{w \cdot\Big(  D^2 \Big[ \psi \big(\eta(x,\cdot)\big)\Big](\eps v + \theta) +  D^2 \Big[ \psi \big(\eta(x,\cdot)\big)\Big](\eps v + \tilde{\theta})\Big) w }{|w|^{d+2s}} \d w \\
&\quad \leq \Big\lVert  D^2 \Big[ \psi \big(\eta(x,\cdot)\big)\Big] \Big\lVert_{L^\infty(\RR^d)} \delta^{2-2s} .
\end{align*}
Put together we see that for $\delta=1$ we have for all $\eps >0$ and $v\in\RR^d$
\begin{align*}
\bigg| \Delsv \Big[ \psi\big(t,\eta(x,\cdot)\big)\Big] (\eps v) \bigg| \leq G_t(x)
\end{align*}
and $G_t(x)$ is in $L^2(\Omega)\subset L^2_{F(v)}(\Omega\times \RR^d)$ by the previous estimates on the second derivative. Hence, we have proven that $\Delsv \Big[ \psi\big(t,\eta(x,\cdot)\big)\Big] (\eps v)$ converges strongly in $L^\infty(0,T; L^2_{F(v)}(\Omega\times\RR^d))$ to $\DelsSR \psi(t,x)$ and Lemma \ref{lemma:CVSR} in the half-space follows. 

\subsubsection{Lemma \ref{lemma:CVSR} in a ball}

We consider, without loss of generality, that $\Omega$ is the unit ball in $\RR^d$. For $\psi$ in $\mathfrak{D}_T(\Omega)$, we will again prove Lemma \ref{lemma:CVSR} by establishing the strong convergence of $\Delsv \Big[ \psi\big( t,\eta(x, \cdot )\big)\Big] (\eps v)$ in $L^\infty(0,T; L^2_{F(v)}(\Omega\times\RR^d))$ to $\DelsSR \psi(t,x)$. \\

First, let us point out that the arguments we presented in the half-space to prove the point-wise convergence still hold in the ball. Indeed, we can introduce the function $\chi_x$ defined in \eqref{eq:defchi} and split \eqref{eq:delschix} over $|w| \leq \delta$ and $|w|\geq \delta$ for some $\delta>0$. On the one hand, if we bound the integral over $|w| \geq \delta$ by the product of the $L^\infty$-norm of $\chi_x$ in $\Omega$ and the integral of the kernel away from its singularity, it follows that this term goes to $0$ by definition of $\chi_x$ and regularity of $\psi$. On the other hand, the integral over $|w|\leq \delta$ can be handled exactly the same way as in the half-space. More precisely, if $x$ is away from the boundary then for $\delta$ and $\eps$ small enough $\eta(x,\eps v+ w) =x+\eps v + w$ and there is no issue; and if $x$ is on $\dO$ then we use the fact that locally the boundary of the ball is isomorphic to the hyperplane $\lbrace x_d = 0\rbrace$ so we recover the previous setting and a dominated convergence argument in $L^1(B(\delta))$ will show that the integral over $|w|\leq \delta$ goes to 0. Together, these two controls and \eqref{eq:delschix} prove the point-wise convergence. \\
The rest of our proof of Lemma \ref{lemma:CVSR} requires some estimates on the derivatives of $\eta$. These estimates can be established by a detailed analysis of the trajectories described by $\eta$ and we have devoted the Appendix \ref{app:FreeTransport} of this paper to this analysis. In particular, in Section \ref{app:subsecDels}, we prove the following Lemma:
\begin{lemma} \label{lemma:delsint}
For all $\psi \in \mathfrak{D}_T$ there exists $p>2$ such that 
\begin{align*}
\Delsv \Big[ \psi\big( t,\eta(x, v) \big)\Big]  \in L^p_{F(v)} (\Omega\times\RR^d).
\end{align*}
\end{lemma}
The strong convergence of $\Delsv \Big[ \psi\big( t,\eta(x, \cdot )\big)\Big] (\eps v)$ in $L^2_{F(v)}(\Omega\times\RR^d)$ then follows from the following result
\begin{lemma}
If $(h_\eps)_{\eps>0}$ converges point-wise to $h$ and is bounded in $L^p_{F(v)} (\Omega\times\RR^d)$ for some $p>2$ uniformly in $\eps$ then $h_\eps$ converges strongly to $h$ in $L^2_{F(v)} (\Omega\times\RR^d)$. 
\end{lemma}
\begin{proof}
Consider $R>0$ and the ball $B(R)$ of radius $R$ centred at $0$ in $\RR^d$. The Egorov theorem states that, since $\Omega\times B(R)$ is a bounded domain, for any $\delta>0$ one can find a subset $A_\delta \subset \Omega\times B(R)$ such that $| \lbrace \Omega\times B(R) \rbrace\setminus A_\delta | \leq \delta$ and $h_\eps$ converges uniformly on $A_\delta$ which means in particular
\begin{align*}
\int_{A_\delta} | h_\eps - h |^2 F(v) \d x \d v \rightarrow 0.
\end{align*}
As a consequence, we split the norm as follows
\begin{align*}
&\underset{\Omega\times\RR^d}{\iint} |h_\eps - h|^2 F(v) \d x \d v = \underset{A_\delta}{\iint}  |h_\eps - h|^2 F(v) \d x\d v \\
&\quad + \underset{\lbrace\Omega\times B(R) \rbrace\setminus A_\delta}{\iint}  |h_\eps - h|^2 F(v) \d x \d v + \underset{\Omega\times\lbrace\RR^d\setminus B(R)\rbrace}{\iint}  |h_\eps - h|^2 F(v) \d x \d v.
\end{align*}
 The first term is handled by Egorov's theorem. For the second, we write using the boundedness in $L^p_{F(v)}(\Omega\times\RR^d)$
\begin{align*}
&\bigg| \underset{\lbrace\Omega\times B(R) \rbrace \setminus A_\delta}{\iint}  |h_\eps - h|^2 F(v) \d x \d v \bigg| \\
& \leq \bigg( \underset{\lbrace\Omega\times B(R) \rbrace \setminus A_\delta}{\iint}  |h_\eps - h|^p F(v) \d x \d v \bigg)^{2/p} \bigg( \underset{\lbrace\Omega\times B(R) \rbrace \setminus A_\delta}{\iint} F(v) \d x \d v \bigg)^{1-2/p} \\
&\leq C |\lbrace\Omega\times B(R) \rbrace \setminus A_\delta|^{1-2/p} \\
&\leq C \delta^{1-2/p}
\end{align*}
and for the third, using Proposition \ref{prop:eqLFP}
\begin{align*}
&\bigg| \underset{\Omega\times\lbrace\RR^d\setminus B(R)\rbrace}{\iint}  |h_\eps - h|^2 F(v) \d x \d v \bigg| \\
& \hspace{1cm} \leq \bigg( \underset{\Omega\times\lbrace\RR^d\setminus B(R)\rbrace}{\iint}  |h_\eps - h|^p F(v) \d x\d v \bigg)^{2/p} \bigg( \underset{\Omega\times\lbrace\RR^d\setminus B(R)\rbrace}{\iint} F(v) \d x \d v \bigg)^{1-2/p} \\
&\hspace{1cm} \leq C \Big( \frac{1}{R^{2s}}\Big)^{1-2/p} .
\end{align*}
Hence, for any $\tilde{\delta}>0$ we can find $R$ such that $R^{-2s(1-2/p)} \leq \tilde{\delta} /3$, $\delta$ such that $\delta^{1-2/p}\leq \tilde{\delta}/3$ and $\eps_0$ such that for all $\eps \leq \eps_0$
\begin{align*}
\int_{A_\delta} | h_\eps - h |^2 F(v) \d x\d v \leq \frac{\tilde{\delta}}{3}
\end{align*}
and the lemma follows.
\end{proof}

\begin{rmq} \label{rmq:spetit}
In both the half-space and the ball, when $s<1/2$, we do not need to assume that $\na \psi (x) \cdot n(x) = 0$ for all $x$ on the boundary which means we can actually extend the set of test functions to $\psi\in \mathcal{C}^\infty ([0,T)\times \bar{\Omega} )$ with $\psi(T,\cdot) =0$. Indeed, in those cases, $\eta$ is regular enough to ensure that $\psi\big( t, \eta(x,v)\big)$ is in $H^{1}(\RR^d)$ with respect to the velocity and since $H^{2s}(\RR^d) \subset H^1(\RR^d)$, the fractional Laplacian of order $s$ of $\psi\big( t, \eta(x,v)\big)$ will be in $L^2_{F(v)}(\Omega\times\RR^d)$. Moreover, in our proof of point-wise convergence above, if $2s<1$ then we can control the singularity for small $w$ in \eqref{eq:delschix} with a first-order Taylor Lagrange expansion which mean we do not require any assumption on $\na \psi$ at the boundary. 
\end{rmq}

\section{Well posedness of the specular diffusion equation} \label{sec:wellposedness}

This last section is devoted to the proof of Theorem \ref{thm:wellposedness} and is divided in three steps. First, we establish some properties of the specular diffusion operator $\DelsSR$. Secondly, we handle the first part of Theorem \ref{thm:wellposedness} which is the existence and uniqueness of a weak solution to the specular diffusion equation \eqref{eq:FracHeatSR}-\eqref{eq:FracHeatSRit}. Thirdly, we will show that the distributional solution $\rho$ that we constructed in the previous section is precisely this unique weak solution when $\Omega$ is either the half-space $\RR^d_+ = \lbrace (\bar{x},x_d)\in\RR^d : x_d>0\rbrace$ or the unit ball $B_1$ in $\RR^d$. \\
Note that although the theorem holds in both domains and the steps are similar in both cases, the techniques we use at each step often differ so we will have to treat the cases separately several times. 

\subsection{Properties and estimates of the specular diffusion operator}

\subsubsection{$\DelsSR$ on the half-space}

When $\Omega$ is the half-space $\RR^d_+$, $\DelsSR$ can be written as a kernel operator using the notations of Section \ref{subsubsec:lemmaCVSRhalfspace}
\begin{prop}
Let us define $K_{\RR^d_+}$ as
\begin{align} \label{eq:Khalfspace}
& K_{\RR^d_+} (x,y) = c_{d,s} \bigg(\frac{1}{|x-y|^{d+2s}} + \frac{1}{|(\bar{x}-\bar{y}, x_d+y_d)|^{d+2s}}\bigg)
\end{align}
Then we have
\begin{equation} \tag{\ref{def:kernelDelSR}}
\DelsSR \psi (x) = P.V. \underset{\RR^d_+}{\int} \big( \psi(x) - \psi(y) \big) K_{\RR^d_+}(x,y) \d y.
\end{equation}
Moreover, this kernel is symmetric: $K_{\RR_+^d}(x,y) = K_{\RR_+^d}(y,x)$ for all $x$ and $y$ in $\RR^d_+$ and satisfies 
\begin{equation} \label{eq:estimateRRd+}
c_{d,s} \frac{1}{|x-y|^{d+2s}} \leq K_{\RR_+^d}(x,y) \leq c_{d,s} \frac{2}{|x-y|^{d+2s}}
\end{equation}
\end{prop}
\begin{proof}
The expression for $\eta(x,v)$ in the half-space is given in \eqref{eq:etaexplicitHS} and \eqref{def:kernelDelSR} follows immediately from it. $K_{\RR^d_+}$ is obviously well defined, although singular, and moreover we have:
\begin{align*}
K_{\RR^d_+} (x,y) &= c_{d,s} \bigg(\frac{1}{|x-y|^{d+2s}} + \frac{1}{|(\bar{x}-\bar{y}, x_d+y_d)|^{d+2s}}\bigg) \\
&= c_{d,s} \bigg( \frac{1}{|y-x|^{d+2s}} + \frac{1}{|(\bar{y}-\bar{x}, y_d+x_d)|^{d+2s}} \bigg) = K_{\RR^d_+} (y,x) .
\end{align*}
Finally, since $1/|(\bar{y}-\bar{x}, y_d+x_d)|^{d+2s} \geq 0$, the left-hand-side of \eqref{eq:estimateRRd+} holds and by a basic geometry argument, illustrated for instance in Figure \ref{fig:halfspace}, for any $x,y$ in $\RR^d_+$: $|(\bar{x}-\bar{y}, x_d+y_d)| \geq |x-y|$ which yields the right-hand-side of \eqref{eq:estimateRRd+}. 
\end{proof}
In more general domains $\Omega$, we can also try to write $\DelsSR$ as a kernel operator. The general form of this kernel is given by a generalized change of variable formula, c.f. \cite{LukesMaly} and reads
\begin{equation} \label{eq:genkernel}
K_\Omega (x,y) = c_{d,s} \sum_{ v \in \eta_x^{-1}(y)} \frac{\big|\det \na_v \eta(x,v) \big|^{-1}}{|v|^{d+2s}} 
\end{equation}
where $\eta_x^{-1}(y) = \lbrace v\in\RR^d : \eta(x,v) = y\rbrace$. For instance, when $\Omega$ is a stripe and a cube, one can show that the Jacobian determinant of $\eta$ in those domains is bounded away from $0$, that the sum is infinite but countable and as a consequence that the kernel will be well defined, symmetric and its singularity will be comparable with the singularity of $\Dels$ as expressed in \eqref{eq:estimateRRd+} for the half-space. Although we won't dwell on those domains in this paper, we will make sure not to use the explicit expression of the kernel in the half-space when ever possible in order to establish results that will also hold in any domains where the kernel is well defined, symmetric and $2s$-singular. In particular, we can establish an integration by parts formula for $\DelsSR$ from which we will deduce its symmetry.
\begin{prop}
The operator $\DelsSR$ satisfies an integration by parts formula: for any $\psi$ and $\phi$ smooth enough:
\begin{align}\label{eq:IPPDelSR}
&\underset{\Omega}{\int} \phi(x) \DelsSR \psi(x) \d x = \frac{1}{2} \underset{\Omega\times\Omega}{\iint} \big(\phi(x) - \phi(y) \big)\big( \psi(x) - \psi(y) \big) K_{\Omega} (x,y) \d x \d y.  
\end{align}
\end{prop}
\begin{proof}
First, we use the kernel operator expression \eqref{def:kernelDelSR} for the $\DelsSR$ operator and inverse the variables $x$ and $y$, using the symmetry of the kernel $K_\Omega$, in order to write the following:
\begin{align*}
\underset{\Omega}{\int} \phi(x) \DelsSR\psi(x) \d x &= \frac{1}{2} \underset{x\in\Omega}{\int} \phi(x) P.V.  \underset{y\in\Omega}{\int} \big( \psi(x) - \psi(y) \big) K_{\Omega}(x,y) \d y \d x \\
&\quad - \frac{1}{2}  \underset{y\in\Omega}{\int} \phi(y) P.V.  \underset{x\in\Omega}{\int} \big( \psi(x) - \psi(y) \big) K_{\Omega}(x,y) \d y \d x.
\end{align*}
In first integral, we add and subtract $(x-y)\na\psi(x) \mathds{1}_{B(x)}(y)$ where $\mathds{1}_{B(x)}(y)$ is the indicator function of a ball around $x$ included in $\Omega$, and we notice that since $\psi$ is smooth it satisfies for any $x\in\Omega$ and $y\in B(x)$:
\begin{align*}
\psi(x) - \psi(y) - (x-y)\na\psi(x) \mathds{1}_{B(x)}(y) = O\big( |x-y|^2\big)
\end{align*}
so that the integral
\begin{align*}
\underset{\Omega\times\Omega}{\iint} \phi(x) \Big( \psi(x) - \psi(y) - (x-y) \na\psi(x) \mathds{1}_{B(x)}(y) \Big) K_{\Omega}(x,y) \d x \d y
\end{align*}
is well defined without need of a principal value because the kernel is $2s$-singular with $2s<2$. We do the same in the second integral, adding and subtracting  $(x-y)\na\psi(y) \mathds{1}_{B(y)}(x)$ where $\mathds{1}_{B(y)}(x)$ is the indicator function of a ball around $y$ included in $\Omega$ so that we get:
\begin{align*}
&\underset{\Omega}{\int} \phi(x) \DelsSR\psi(x) \d x \\
&\hspace{1cm}= \frac{1}{2} \underset{\Omega\times\Omega}{\iint} \phi(x) \Big( \psi(x) - \psi(y) - (x-y) \na\psi(x) \mathds{1}_{B(x)}(y) \Big) K_{\Omega}(x,y) \d x \d y \\
&\hspace{1cm} + \frac{1}{2} \underset{x\in\Omega}{\int} \phi(x) \na\psi(x) P.V. \underset{y\in\Omega}{\int} (x-y) \mathds{1}_{B(x)}(y) K_{\Omega} (x,y) \d y \d x  \\
&\hspace{1cm} - \frac{1}{2} \underset{\Omega\times\Omega}{\iint} \phi(y) \Big( \psi(x) - \psi(y) - (x-y) \na\psi(y) \mathds{1}_{B(y)}(x) \Big) K_{\Omega}(x,y) \d x \d y  \\
&\hspace{1cm}- \frac{1}{2} \underset{y\in\Omega}{\int} \phi(y) \na\psi(y) P.V. \underset{x\in\Omega}{\int} (x-y) \mathds{1}_{B(y)}(x) K_{\Omega} (x,y) \d y \d x .
\end{align*}
Since we can use Fubini's theorem in the first and the third term, we sum both of them and notice that $\big(\phi(x) - \phi(y) \big) \big(\psi(x)-\psi(y) \big)  = O\big( |x-y|^2\big)$ in order to write  

\begin{align*}
&\frac{1}{2} \underset{\Omega\times\Omega}{\iint} \phi(x) \Big( \psi(x) - \psi(y) - (x-y) \na\psi(x) \mathds{1}_{B(x)}(y) \Big) K_{\Omega}(x,y) \d x \d y \\
&\hspace{1cm} - \frac{1}{2} \underset{\Omega\times\Omega}{\iint} \phi(y) \Big( \psi(x) - \psi(y) - (x-y) \na\psi(y) \mathds{1}_{B(y)}(x) \Big) K_{\Omega}(x,y) \d x \d y \\
&= \frac{1}{2}\underset{\Omega\times\Omega}{\iint} \bigg[ \big(\phi(x) - \phi(y) \big) \big(\psi(x)-\psi(y) \big) - \phi(x) \na\psi(x) \mathds{1}_{B(x)}(y) (x-y) \\
&\hspace{1cm} + \phi(y) \na\psi(y) \mathds{1}_{B(y)}(x) (x-y) \bigg] K_{\Omega}(x,y) \d x \d y \\
&= \frac{1}{2}\underset{\Omega\times\Omega}{\iint} \big(\phi(x) - \phi(y) \big) \big(\psi(x)-\psi(y) \big) K_{\Omega}(x,y)\d x\d y \\
&\hspace{1cm} -\frac{1}{2} \underset{x\in\Omega}{\int}  \phi(x) \na\psi(x) P.V. \underset{y\in\Omega}{\int} (x-y) K_{\Omega}(x,y)  \mathds{1}_{B(x)}(y)\d y \d x \\
&\hspace{1cm} + \frac{1}{2} \underset{y\in\Omega}{\int}  \phi(y) \na\psi(y) P.V. \underset{x\in\Omega}{\int}  (x-y) \mathds{1}_{B(y)}(x) K_{\Omega}(x,y) \d x \d y 
\end{align*}
which concludes the proof. 
\end{proof}
As a direct corollary of this proof, we see that since the kernel $K_{\Omega}$ is symmetric, the operator is symmetric as well:
\begin{align*}
\underset{\Omega}{\int} \phi(x) \DelsSR \psi(x) \d x = \underset{\Omega}{\int} \psi(x) \DelsSR \phi(x) \d x.
\end{align*}

\subsubsection{$\DelsSR$ on a ball}

In the ball, if we wanted to write $\DelsSR$ as a kernel operator using \eqref{eq:genkernel}, the kernel would only be defined almost everywhere because the determinant of $\na_v \eta$ is not bounded away from $0$. Indeed -- see Appendix \ref{app:FreeTransport} for proof -- for a fixed $x$, a fixed direction $\theta = v/|v| \in\mathbb{S}^{d-1}$ and a fixed number of reflections, we can find one and only one norm $|v|$ such that the determinant of $\na_x \eta(x,|v|\theta)$ is null. This can be seen in the expression \eqref{eq:naveta} because finding this norm is equivalent to solving $\det \na_v \eta(x,v) =0$ after fixing all the variables except $l_{end}$ and, in that setting, the Jacobian determinant is a monotonous function of $l_{end}$ that passes through $0$. However, for each fixed $x$, the set of velocities $v$ such that the determinant is null is a countable sum of curves since for each fixed number of reflections $k$ there is exactly one $v$ in that set per direction $\theta$ in $\mathbb{S}^{d-1}$. Therefore, the kernel is defined almost everywhere.\\
Nevertheless, even if we can't rigorously write it with a kernel, the specular diffusion operator still has interesting properties, as for instance:
\begin{prop}
When $\Omega$ is a ball $B$, the operator $\DelsSR$ admits the following integration by parts formula: for all $\phi$ and $\psi$ smooth enough
\begin{align} \label{eqball:ipp}
&\underset{\Omega}{\int} \phi(x) \DelsSR \psi(x) \d x \\
&\quad = \frac{1}{2} c_{d,s}\underset{\Omega\times\RR^d}{\iint} \Big( \phi(x)-\phi\big(\eta(x,v)\big)\Big)\Big( \psi(x)-\psi(\eta(x,v)\big)\Big) \frac{\d v \d x}{|v|^{d+2s}} .
\end{align}
\end{prop}
From which we readily deduce its symmetry
\begin{align} \label{eqball:sym}
\underset{\Omega}{\int} \phi(x) \DelsSR \psi(x) \d x = \underset{\Omega}{\int} \psi(x) \DelsSR \phi(x) \d x 
\end{align}
\begin{proof}
We write
\begin{align}
&\underset{\Omega}{\int}  \phi(x) \DelsSR \psi(x) \d x \\
&\quad=  c_{d,s}\underset{\Omega\times\RR^d}{\iint} \Big( \phi(x)-\phi\big(\eta(x,v)\big)\Big)\Big( \psi(x)-\psi(\eta(x,v)\big)\Big) \frac{\d v \d x}{|v|^{d+2s}}  \nonumber \\
&\quad -  c_{d,s} P.V.\underset{\Omega\times\lbrace\RR^d}{\iint} \phi\big(\eta(x,v)\big)\Big( \psi(x)-\psi\big(\eta(x,v)\big) \Big) \frac{\d v \d x}{|v|^{d+2s}} \label{eq:ippPV}.
\end{align}
In the second term on the right-hand-side we want to do a change of variable $F(x,v) = (y,w)$ such that the trajectory described by $\eta$ from $(y,w)$ is exactly the trajectory from $(x,v)$ backwards. In particular, that means $\eta(y,w)=x$ and $\eta(x,v)=y$. We have the following result on this change of variable which will be proven in Section \ref{subsec:etaballcov} of the appendices:
\begin{lemma} \label{lem:changeofvariable}
The change for variable $F$ given by 
\begin{align}
F \begin{pmatrix} x \\ v \end{pmatrix} = \begin{pmatrix} \eta(x,v) \\ - \big[\na_v \eta(x,v)\big] v \end{pmatrix}
\end{align}
is precisely the change of variable such that $\eta ( F(x,v) ) = x$ and the trajectory described by $\eta$ starting at $\eta(x,v)$ with velocity $- \big[\na_v \eta(x,v)\big] v$ is exactly the trajectory from $(x,v)$ backwards. Moreover, for all $(x,v)$:
\begin{equation}
\det \na F(x,v) = 1.
\end{equation}
\end{lemma}
The singularity that requires the principal value in \eqref{eq:ippPV} is at $\{v=0\}$ around which we have explicitly $\eta(x,v)=x+v$ hence it will become, through the change of variable, a singularity at $\{w=0\}$ since we have $w=-v$ in the neighbourhood of $0$. The change of variables yields
\begin{align*}
\underset{\Omega}{\int}  \phi(x) \DelsSR \psi(x) \d x &=  c_{d,s}\underset{\Omega\times\RR^d}{\iint} \Big( \phi(x)-\phi\big(\eta(x,v)\big)\Big)\Big( \psi(x)-\psi(\eta(x,v)\big)\Big) \frac{\d v \d x}{|v|^{d+2s}}  \\
&\quad -  c_{d,s}P.V. \underset{\Omega\times\RR^d}{\iint} \phi(y)\Big( \psi\big(\eta(y,w)\big)-\psi(y) \Big) \frac{\d w \d y}{|w|^{d+2s}} 
\end{align*}
and the integration by parts formula follows. 
\end{proof}
Finally, in relation with \eqref{eq:estimateRRd+}, one can see immediately from looking at the integration by part formula in a ball, that the singularity in the operator is of order exactly $2s$.

\subsubsection{The Hilbert space $\HSRs(\Omega)$}
We conclude the analysis of $\DelsSR$ by introducing the associated Hilbert space $\HSRs(\Omega)$. This comes down to interpreting the integration by parts formula as a type of scalar product and considering the associated semi-norm in the spirit of the Gagliardo (semi-)norm on the fractional Sobolev space $H^s(\RR^d)$ and its relation with the fractional Laplacian as presented e.g. in \cite{DiNezza+}. The natural semi-norm associated with the specular diffusion operator reads in the half-space
\begin{align*}
[\psi]_{\HSRs(\RR^d_+)}^2 = \frac{1}{2} \underset{\RR^d_+\times\RR^d_+}{\iint} \big( \psi(x) - \psi(y) \big)^2 K_{\RR^d_+}(x,y) \d x \d y.
\end{align*}
and in the ball
\begin{align*}
[\psi]_{\HSRs(B)}^2 =   \frac{c_{d,s}}{2} \underset{\RR^d \times B}{\iint} \Big( \psi(x) - \psi\big(\eta(x,v)\big) \Big)^2\frac{1}{|v|^{d+2s} }\d x\d v.
\end{align*}
Consequently, we introduce a Hilbert space associated with the specular diffusion operator.
\begin{defi} 
We define the Hilbert space $\HSRs(\Omega)$ as
\begin{equation} \label{def:HSRs}
\HSRs(\Omega) = \Big\{ \psi \in L^2(\Omega) : [\psi]_{\HSRs(\Omega)} <\infty \Big\}
\end{equation}
associated with a scalar product which, on a half-space, read
\begin{equation}
\begin{aligned}
\langle \psi | \phi \rangle_{\HSRs(\RR^d_+)} &= \underset{\RR^d_+}{\int} \psi \phi \d x \\
&\quad + \frac{1}{2} \underset{\RR^d_+\times\RR^d_+}{\iint} \big(\phi(t,x) - \phi(t,y) \big)\big( \psi(t,x) - \psi(t,y) \big) K_{\RR^d_+} (x,y) \d x\d y
\end{aligned}
\end{equation} 
and on the ball becomes
\begin{equation}
\begin{aligned}
\langle \psi | \phi \rangle_{\HSRs(B)} &= \underset{B}{\int} \psi \phi \d x \\
& \quad + \frac{ c_{d,s}}{2} \underset{\RR^d\times B}{\iint} \Big(\phi(t,x) - \phi\big(t,\eta(x,v)\big) \Big)\Big( \psi(t,x) - \psi\big(t,\eta(x,v)\big)\Big) \frac{ \d x\d v}{|v|^{d+2s}}
\end{aligned}
\end{equation}
hence the norm associated with $\HSRs(\Omega)$ is naturally
\begin{align*}
\lVert \psi \lVert_{\HSRs(\Omega)}^2 = \lVert \psi \lVert_{L^2(\Omega)}^2 + [\psi]_{\HSRs(\Omega)}^2.
\end{align*}
\end{defi}
This functional space is strongly linked with the Sobolev space $H^s(\Omega)$ and we refer the interested reader to \cite{DiNezza+} for more details. We notice right away that $\DelsSR$ is self-adjoint on the Hilbert space $\HSRs(\Omega)$ and also, by the estimates on the singularity of the operator established above, we see that $ \HSRs(\Omega) \subset H^s(\Omega)$. 

\subsection{Existence and uniqueness of a weak solution for the macroscopic equation}
We now turn to the specular diffusion equation \eqref{eq:FracHeatSR}-\eqref{eq:FracHeatSRit}.
\begin{repthm}{thm:wellposedness}[Part I]
Let $\Omega$ be a half-space or a ball in $\RR^d$, $u_{in}$ be in $L^2((0,T)\times\Omega)$ and $s$ be in $(0,1)$. For any $T>0$, there exists a unique weak solution $u \in L^2(0,T;\HSRs(\Omega))$ to
\begin{subequations}
\begin{align}
&\pa_t u + \DelsSR u = 0 & (t,x)\in [0,T)\times\Omega  \tag{\ref{eq:FracHeatSR}}\\
&u(0,x) = u_{in} (x) & x\in\Omega \tag{\ref{eq:FracHeatSRit}}
\end{align}
\end{subequations}
in the sense that for any $\psi\in \mathcal{C}^\infty_c(Q_T)$ if $s<1/2$ and any $\psi \in \mathfrak{D}_T(\Omega)$ if $s\geq 1/2$, $u$ satisfies if $\Omega$ is a half-space:
\begin{equation} \tag{\ref{eq:weaksolFracHeathalfspace}}
\begin{aligned}
&\underset{(0,T)\times\Omega}{\iint} u \pa_t \psi \d t\d x+ \underset{\Omega}{\int} u_{in}(x)\psi(0,x) \d x   \\
& -\frac{1}{2} \underset{(0,T)\times\Omega\times\Omega}{\iiint} \big( u(t,x)-u(t,y) \big)\big( \psi(t,x) - \psi(t,y) \big) K(x,y) \d t\d x \d y=0 .
\end{aligned}
\end{equation}
and if $\Omega$ is the unit ball
\begin{equation} \tag{\ref{eq:weaksolFracHeatball}}
\begin{aligned}
&\underset{(0,T)\times\Omega}{\iint} u \pa_t \psi \d t \d x + \underset{\Omega}{\int} u_{in}(x)\psi(0,x) \d x \\
&-\frac{1}{2} \underset{(0,T)\times\Omega\times\RR^d}{\iiint} \big( u(t,x)-u\big(t,\eta(x,v)\big) \big)\big( \psi(t,x) - \psi\big(t,\eta(x,v)\big) \big) \frac{ \d t \d x \d v }{|v|^{d+2s}} =0 .
\end{aligned}
\end{equation}
\end{repthm}

\begin{proof}[Proof of Theorem \ref{thm:wellposedness}, (Part I)] 
We focus on the case $s\geq 1/2$ for the sake of clarity, the proof for $s<1/2$ is similar. The following proof is strongly inspired by the method of Carrillo in \cite{Carrillo98}. We consider an associated problem which comes formally from deriving \eqref{eq:FracHeatSR} for $\bar{u}(t,x) = e^{-\lambda t} u (t,x) $ for some $\lambda >0$:
\begin{equation} \label{eq:fracheatbar}
\begin{aligned}
&\pa_t \bar{u}(t,x) + \lambda \bar{u}(t,x) + \DelsSR \bar{u}(t,x) = 0 & (t,x)\in (0,T)\times\Omega \\
& \bar{u}(0,x) = \bar{u}_{in} (x) & x\in\Omega.  
\end{aligned}
\end{equation}
Note that we do not prescribe any explicit boundary condition on $\dO$. A weak solution of \eqref{eq:fracheatbar} is a function $\bar{u} \in L^2(0,T;\HSRs(\Omega))$ such that for any $\psi\in \mathfrak{D}_T$, 
\begin{align*}
\underset{(0,T)\times\Omega}{\iint} \Big( -\bar{u} \pa_t \psi + \lambda \bar{u} \psi + \bar{u} \DelsSR \psi \Big) \d t \d x - \underset{\Omega}{\int} \bar{u}_{in}(x) \psi(0,x) \d x =0.
\end{align*}
We first prove existence of weak solutions to this problem using a Lax-Milgram argument and we will show afterwards that it implies existence for \eqref{eq:FracHeatSR}-\eqref{eq:FracHeatSRit}. We consider on $\mathfrak{D}_T$ the prehilbertian norm
\begin{align*}
|\psi|^2_{\mathfrak{D}_T} = \lVert \psi \lVert^2_{\HSRs(\Omega)}  + \frac{1}{2} \lVert \psi(0,\cdot) \lVert^2_{L^2(\Omega)}.
\end{align*}
We then introduce the bilinear form $a$ from $L^2(0,T;\HSRs(\Omega)) \times \mathfrak{D}_T$ to $\RR$ defined as
\begin{align*}
a(\bar{u},\psi) = \underset{(0,T)\times\Omega}{\iint} \Big( -\bar{u} \pa_t \psi + \lambda \bar{u} \psi + \bar{u} \DelsSR \psi \Big) \d t \d x 
\end{align*}
and the continuous bounded linear operator $L$ on $\mathfrak{D}_T$:
\begin{align*}
L(\psi) = \underset{\Omega}{\int} \bar{u}_{in}(x) \psi(0,x) \d x.
\end{align*}
From Lemma \ref{lemma:delsint} we know in particular that $\mathfrak{D}_T$ is a subset of $L^2(0,T;\HSRs(\Omega))$ with a continuous injection. Moreover, it is easy to see that $a$ is continuous and it is also coercive since:
\begin{align*}
a(\psi,\psi) = \underset{(0,T)\times\RR^d}{\iint} \lambda \psi^2 + \psi\DelsSR \psi \d t\d x + \frac{1}{2} \underset{\Omega}{\int} \psi(0,x)^2 \d x \geq \min(1,\lambda) |\psi|^2_{\mathfrak{D}_T}
\end{align*} 
hence, the Lax-Milgram theorem gives us existence of a weak solution of \eqref{eq:fracheatbar} in $L^2(0,T;\HSRs(\Omega))$. From this weak solution $\psi$ we define $\bar\psi (t,x) = e^{-\lambda t} \psi(t,x)$ which is obviously in $L^2(0,T;\HSRs(\Omega))$ and weak solution of \eqref{eq:FracHeatSR}-\eqref{eq:FracHeatSRit}. Since the equation is linear, to show uniqueness is equivalent to proving that the only weak solution with initial data $u_{in} = 0$ is the zero function. Call $u_0$ this weak solution. Multiplying \eqref{eq:FracHeatSR} by $u_0$ and integrating over $\Omega$ we have:
\begin{align*}
\underset{\Omega}{\int} \frac{1}{2 }\pa_t \big( u_0^2 \big) \d x =-  \underset{\Omega}{\int} u_0 \DelsSR u_0 \d x \leq 0.
\end{align*}
Hence $\lVert u_0(t,\cdot) \lVert_{L^2(\Omega)}$ is decreasing. Since it was $0$ to start with, that means $u_0 \equiv 0$ and that concludes the proof of uniqueness of solution. Finally, we notice that the integration by parts formula \eqref{eq:IPPDelSR} concludes the proof existence and uniqueness of a weak solution of \eqref{eq:FracHeatSR}-\eqref{eq:FracHeatSRit} in the sense given in Theorem \ref{thm:wellposedness}. \\
\end{proof}

\subsection{Identifying the macroscopic density as the unique weak solution}

Finally, we turn to the last part of Theorem \ref{thm:wellposedness}
\begin{repthm}{thm:wellposedness}[Part II]
If $\Omega$ is a ball or a half-space, the macroscopic density $\rho$ who satisfies \eqref{eq:FracHeatSRweak} for all $\psi\in\mathfrak{D}_T(\Omega)$ is the unique weak solution of \eqref{eq:FracHeatSR}-\eqref{eq:FracHeatSRit}.
\end{repthm}

\begin{proof}
In order to prove this theorem we will show that there is a unique distributional solution of \eqref{eq:FracHeatSRweak}, i.e. a unique $\rho$ such that \eqref{eq:FracHeatSRweak} holds for all $\psi\in \mathcal{C}^\infty_c(Q_T)$ if $s<1/2$ and any $\psi \in \mathfrak{D}_T(\Omega)$ if $s\geq 1/2$. Indeed, since it is obvious that the weak solution of \eqref{eq:FracHeatSR}-\eqref{eq:FracHeatSRit} is also a distributional solution of \eqref{eq:FracHeatSRweak}, if we prove its uniqueness then Theorem \ref{thm:wellposedness} Part II will follow immediately. \\
As usual, to prove uniqueness for linear PDEs, we assume that there are two distributional solutions $\rho_1$ and $\rho_2$ or \eqref{eq:FracHeatSRweak} and we consider their difference $\bar\rho = \rho_1-\rho_2$ which satisfies for any $\psi$ 
\begin{equation} \label{eq:distsoluniq}
\underset{[0,T)\times\Omega}{\iint} \bar\rho \Big( \pa_t \psi - \DelsSR \psi \Big) \d t  \d x = 0
\end{equation}
with $\int_\Omega \bar\rho \d x = 0$ thanks to the conservation of mass. We want to prove that $\bar\rho$ is null. In order to do so, we first introduce the following reverse evolution problem and show its wellposedness:
\begin{prop}
For any $\bar\rho \in L^\infty( [0,T); L^2(\Omega))$ there exists a unique $\psi_{\bar\rho}$ weak solution in $L^2((0,T)\times\Omega)$ of:
\begin{equation} \label{eq:evoHY}
\left\{ \begin{aligned}   &\pa_t \psi_{\bar\rho} - \DelsSR \psi_{\bar\rho} = \bar\rho &\hspace{1cm} (t,x)\in [0,T)\times\Omega \\
									&\psi_{\bar\rho} (T,x) = 0 & x\in\Omega
		\end{aligned} \right.
\end{equation}
\end{prop}
\begin{proof}
The proof of part 1 of Theorem \ref{thm:wellposedness} above can easily be adapted to show existence of uniqueness of weak solution in $L^2(0,T; \HSRs(\Omega))$ of \eqref{eq:FracHeatSR}-\eqref{eq:FracHeatSRit} with a source term $S$, namely:
\begin{align*}
&\pa_t u + \DelsSR u = S(t,x) & (t,x)\in [0,T)\times\Omega \\
&u(0,x) = u_{in} (x) & x\in\Omega .
\end{align*}
To do so, one only needs to change the continuous bounded linear map $L$ to 
\begin{align*}
L(\psi) = \underset{\Omega}{\int} \bar{u}_{in} (x) \psi(0,x) \d x + \underset{(0,T)\times\Omega}{\iint} \bar{S} \psi  \d t \d x
\end{align*}
where $\bar{S}(t,x) = e^{-\lambda t} S(t,x)$, and the rest of the proof holds. Hence, if we consider this weak solution $u$ and define $\psi_{\bar\rho}(t,x) = u( T-t, x)$ as well as choose $S$ such that $\bar\rho(t,x) = -S(T-t, x)$ and take $u_{in} = 0$, this gives us the unique $\psi_{\bar\rho}$ weak solution of \eqref{eq:evoHY} in $L^2(0,T ; \HSRs(\Omega))$.
\end{proof}
We see now that if we can use $\psi_{\bar\rho}$ as a test function in \eqref{eq:distsoluniq} then we will have 
\begin{align*}
\underset{[0,T)\times\Omega}{\iint} \bar\rho^2 \d x \d t = 0 
\end{align*}
which concludes the proof of uniqueness of the distributional solution $\rho$ of \eqref{eq:weakFHSR}. It remains to show that $\psi_{\bar\rho}$ is an admissible test function for \eqref{eq:FracHeatSRweak}.
 
When $s<1/2$, since $\mathcal{C}^{\infty}([0,T)\times\bO)$ is dense in $L^\infty([0,T); \HSRs(\Omega))$ with respect to the $\HSRs$-norm, the result is immediate. \\

When $s>1/2$, however, the test functions in \eqref{eq:distsoluniq} need to be in $\mathfrak{D}_T$ so we need to understand the behaviour of $\psi_{\bar\rho}$ on the boundary. Let us recall that $\mathfrak{D}_T$ is defined as:
\begin{equation*}
\mathfrak{D}_T(\Omega) = \Big\{ \psi\in\mathcal{C}^\infty ([0,T)\times\bO) \text{ s.t. } \psi(T,\cdot)=0 \text{ and } \forall x\in\dO : \na_x \psi(t,x)\cdot n(x) =0 \Big\}.
\end{equation*} 
The interaction between the singularity in the specular diffusion operator and the boundary leads us to believe that $\psi_{\bar\rho}$ satisfies a rather strong, non-local boundary condition but we are unable to write this condition explicitly since it is contained in the action of $\DelsSR$. As a consequence, we will show instead that $\psi_{\bar\rho}$ satisfies, in particular, an homogeneous Neumann condition. To that end, we first regularize with respect to time the right hand side of \eqref{eq:evoHY}, and call $n$ the regularizing parameter. Since the operator $\DelsSR$ is self-adjoint and dissipative it generates a strongly continuous semi-group of contractions and as a consequence one can prove, see \cite{Pazy} Section 4.2 for more details, that for each $n$ there exists a unique strong solution $\psi_n$ of \eqref{eq:evoHY} which, in particular, satisfies for any $t$
\begin{equation} \label{eq:linfpsin}
\DelsSR \psi_n (t,x) \in L^\infty(\Omega).
\end{equation}
Moreover, we have the following lemma:
\begin{lemma} \label{lem:Neumann}
Let $\Omega$ be a ball or a half-space and $s>1/2$. For any $\psi$ such that $\DelsSR \psi (x) \in L^\infty (\Omega)$, we have
\begin{equation} \label{eq:Neumannpsi}
\na_x \psi(t,x) \cdot n(x) = 0  \hspace{2cm} \forall x\in \dO.
\end{equation}
\end{lemma}
Postponing the proof of this lemma, let us conclude the proof of Theorem \ref{thm:wellposedness}. For each $n$, $\psi_n$ satisfies the homogeneous Neumann boundary condition and belongs at least in $\HSR2s(\Omega)$ since it is a strong solution of \eqref{eq:evoHY}. As a result, we can approach $\psi_n$ by functions in $\mathfrak{D}_T$ with respect to the  $\HSR2s(\Omega)$-norm, which is strong enough to take the limit in \eqref{eq:distsoluniq}. Hence, $\psi_{\bar\rho}$ is an admissible test function for \eqref{eq:distsoluniq}, which yields the uniqueness of the distributional solution of \eqref{eq:weakFHSR}.
\end{proof}

\begin{proof}[Proof of Lemma \ref{lem:Neumann}]
For the half-space, we notice that $\DelsSR \psi$ can be interpreted as the fractional Laplacian acting on its mirror-extension $\widetilde{\psi}$ defined as:
\begin{equation} \label{eq:widetildepsi}
\widetilde{\psi}(t,x) = \left| \begin{aligned}   &\psi(t,x) &\mbox{if } x_d \geq 0\\
													 & \psi(t,[ \bar{x} , -x_d] ) &\mbox{if } x_d \leq 0
							\end{aligned} \right.
\end{equation}
with the notations from Section \ref{subsubsec:lemmaCVSRhalfspace}. The boundary behaviour of $\psi$ follows readily because we know that in order for $\Dels \widetilde{\psi}$ to be bounded, $\widetilde{\psi}$ has to be at least $\mathcal{C}^{1,2s-1}$ on $\RR^d$. Since it is a mirror-extension that means $\psi$ has to satisfy an homogeneous Neumann condition on the boundary:
\begin{align*}
\na_x \psi(t,x) \cdot n(x) = 0  \hspace{2cm} \forall x\in \dO.
\end{align*}
Note that the same line of argument would also hold in a stripe or a cube since we can define in those cases an extension that consists of a composition of mirror extensions and such that $\DelsSR \psi$ coincides with the action of $\Dels$ on that extension.\\

When $\Omega$ is a ball, since  $\DelsSR \psi (x) \in L^\infty (\Omega)$, we have
\begin{align*}
&\underset{\RR^d}{\int} \Big[ \psi (x) - \psi\big(\eta(x,v)\big) - \na \psi (x) \cdot \big(\eta(x,v) -x\big)  \Big] \frac{\d v}{|v|^{d+2s}} \\
& \hspace{1cm} + P.V. \underset{\RR^d}{\int}  \na \psi (x) \cdot \big(\eta(x,v) -x\big) \frac{\d v}{|v|^{d+2s}}
\end{align*}
in $L^\infty(\Omega)$. In the first integral 
\begin{align*} 
\psi (x) - \psi \big(\eta(x,v)\big) -  \na \psi (x) \cdot \big(\eta(x,v) -x\big) = O( |x-\eta(x,v)|^2)
\end{align*}
which means the integral is finite since $2s<2$. Hence, we have
\begin{align} \label{eqball:neumannpsi}
 \na\psi (x) \cdot P.V. \underset{\RR^d}{\int} \big(\eta(x,v) -x\big) \frac{\d v}{|v|^{d+2s}} \in L^\infty(\Omega).
\end{align}
Let us show that there is a function $f(x)$ such that 
\begin{equation} \label{eqball:neumannK}
 P.V. \underset{\RR^d}{\int} \big(\eta(x,v) -x\big) \frac{\d v}{|v|^{d+2s}}  =  f(x) n(x) \hspace{0.5cm} \text{ with }\hspace{0.5cm} f(x) \underset{x\rightarrow \dO}{\rightarrow} -\infty 
\end{equation}
where $n(x)$ denotes the extended outward normal vector: $n(x) = x/|x|$ if $x\neq 0$. We write the integral in a orthonormal coordinates system that starts with $e_1 = n(x)$ and with the notation $\eta(x,v) = \sum \eta_i(x,v) e_i$. We have:
\begin{align*}
P.V.\underset{\RR^d}{\int} \big(\eta(x,v)-x\big) \frac{\d v}{|v|^{d+2s}}&=  \bigg(P.V. \underset{\RR^d}{\int} \big(\eta_1(x,v) - |x|\big) \frac{\d v}{|v|^{d+2s}} \bigg) e_1 \\
& \quad + \underset{2\leq i\leq d}{\sum} \bigg( P.V.\underset{\RR^d}{\int} \eta_i(x,v) \frac{\d v}{|v|^{d+2s}} \bigg) e_i \\
&:= I_1 n(x) + \underset{2\leq i \leq d}{\sum} I_i e_i.
\end{align*}
For the coefficient $I_2$ we notice that if we call $T_2 :y\in\RR^d \mapsto y -  2y_2 e_2$, the mirror image of $y$ with respect to the hyperplane $\{y_2=0\}$, then it is easy to see that the ball is invariant by $T_2$: $T_2 (B_1) = B_1$ which means that $\eta$ acts in $T_2(B_1)$ exactly as it acts on $B_1$. As a consequence, $T_2$ and $\eta$ commute: $\eta(x,T_2 v)= T_2 \eta(x,v)$ which yields when we write explicitly the principle value:
\begin{align*}
& I_2 = \underset{\eps\rightarrow 0}{\lim} \underset{\{|v_1|\geq \eps\}\times\RR^{d-2}}{\iint} \Bigg( \underset{v_2>0}{\int} \eta_2(x,v) \frac{\d v_2}{|v|^{d+2s}} + \underset{v_2<0}{\int}  \eta_2(x,v)\frac{\d v_2}{|v|^{d+2s}} \Bigg) \d v_1 \d v_3\cdots \d v_d\\
&=\underset{\eps\rightarrow 0}{\lim} \underset{\{|v_1|\geq \eps\}\times\RR^{d-2}}{\iint} \Bigg( \underset{v_2>0}{\int} \eta_2(x,v) \frac{\d v_2}{|v|^{d+2s}} + \underset{v_2>0}{\int}  (\eta_2(x,Tv))\frac{\d v_2}{|v|^{d+2s}} \Bigg) \d v_1 \d v_3\cdots \d v_d\\
&=\underset{\eps\rightarrow 0}{\lim}\underset{\{|v_1|\geq \eps\}\times\RR^{d-2}}{\iint} \Bigg( \underset{v_2>0}{\int} \eta_2(x,v) \frac{\d v_2}{|v|^{d+2s}} + \underset{v_2>0}{\int}  (-\eta_2(x,v))\frac{\d v_2}{|v|^{d+2s}} \Bigg) \d v_1 \d v_3\cdots \d v_d\\
&= 0 .
\end{align*}
The same holds for all $I_i$, $i\geq 2$ so that we can define a function $f(x)= I_1$ with which
\begin{align*}
P.V.\underset{\RR^d}{\int} \big(\eta(x,v)-x\big) \frac{\d v}{|v|^{d+2s}}&=  \bigg(P.V. \underset{\RR^d}{\int} \big(\eta_1(x,v) - |x|\big) \frac{\d v}{|v|^{d+2s}} \bigg) n(x) := f(x) n(x) .
\end{align*}
To understand the behaviour of $f$ as $x$ goes to the boundary we split the integral as follows, for some $R>0$ fixed, writing $B_{1-|x|}$ the ball centred at $0$ of radius $1-|x|$ and $C_R$ the cube centred at $0$ of side $2R$ (assuming w.l.o.g. that $R>1-|x|$):
\begin{align*}
f(x) =& P.V. \underset{B_{1-|x|}}{\int} \big((\eta_1(x,v) - |x|\big) \frac{\d v}{|v|^{d+2s}} \\
&\hspace{1cm} + \underset{C_R\setminus B_{1-|x|}}{\int} \big(\eta_1(x,v) - |x|\big) \frac{\d v}{|v|^{d+2s}} +\underset{\RR^d\setminus C_R}{\int} \big(\eta_1(x,v) - |x|\big) \frac{\d v}{|v|^{d+2s}} .
\end{align*}
For the first term on the right-hand-side, we use the explicit expression of $\eta$ when there are no reflections: $\eta(x,v) = x+v$ in order to write
\begin{align*}
P.V. \underset{B_{1-|x|}}{\int} \big((\eta_1(x,v) - |x|\big) \frac{\d v}{|v|^{d+2s}} = \underset{\eps\rightarrow 0}{\lim} \underset{\eps < |v|< 1-|x|}{\int} v_1 \frac{\d v}{|v|^{d+2s}} = 0
\end{align*}
because the integrand is an odd function and the domain is radially symmetric. For the last term in the expression of $f(x)$ we write
\begin{align*}
\bigg| \underset{\RR^d\setminus C_R}{\int} \big(\eta_1(x,v) - |x|\big) \frac{\d v}{|v|^{d+2s}}\bigg| \leq \underset{|v|>R}{\int} \frac{|v|}{|v|^{d+2s}} \d v = \frac{1}{2s R^{2s-1}} 
\end{align*}
which is fixed with $R$.  Finally, for the second term in the expression of $f(x)$, we want to identify a sign in the integrand to which end we introduce 
\begin{align*}
\mathcal{E}(x) = \Big( C_R\setminus B_{1-|x|} \Big) \cap \Big( \big\{ -R \leq v_1 \leq 0 \big\} \cup \big\{ 2(1-|x|) \leq v_1 \leq R \big\} \Big)
\end{align*}
so that for any $v\in \mathcal{E}(x)$ we have $\eta_1(x,v)-|x| \leq 0$ (note that the set of all velocities such that $\eta_1(x,v)-|x|\leq 0$ is actually a little bigger that $\mathcal{E}(x)$ because of the curvature of $\dO$, if $\dO$ was a straight line that it would be precisely $\mathcal{E}(x)$). We also write $\mathcal{E}^c(x) =  ( C_R\setminus B_{1-|x|} ) \setminus \mathcal{E}(x)$ its complement in $ C_R\setminus B_{1-|x|}$ with which we have
\begin{align*}
&\underset{C_R\setminus B_{1-|x|}}{\int} \big(\eta_1(x,v) - |x|\big) \frac{\d v}{|v|^{d+2s}} \\
&\quad = - \underset{\mathcal{E}(x)}{\int} \big| \eta_1(x,v) - x \big| \frac{\d v}{|v|^{d+2s}} + \underset{\mathcal{E}^c(x)}{\int} \big( \eta_1(x,v) - x \big) \frac{\d v}{|v|^{d+2s}}.
\end{align*}
We introduce the notations 
\begin{align*}
\mathcal{E}(x,v_1)= \Big\{ (v_2,\cdots,v_d) : \sqrt{ (1-|x|)^2-v_1^2} \leq |v_2|,\cdots,|v_d|\leq R \Big\}
\end{align*} and 
\begin{align*}
\mathcal{E}^c(x,v_1)= \Big\{ (v_2,\cdots,v_d) : \sqrt{(1-x)^2-v_1^2 } \leq  |v_2|,\cdots,|v_d| \leq 2(1-|x|) \Big\}
\end{align*}
such that for a fixed $v_1$ the projection of $\mathcal{E}(x)$ on $\lbrace w\in\RR^d : w_1=v_1\rbrace$ is $\mathcal{E}(x,v_1)$ if $-R\leq v_1\leq 0$ and $[-R,R]^{d-1}$ if $2(1-|x|) \leq v_1 \leq R$, and the projection of $\mathcal{E}^c(x)$ is $\mathcal{E}^c(x,v_1)$. With those, we have on the one hand
\begin{align*}
\underset{\mathcal{E}(x)}{\int} \big| \eta_1(x,v) - x \big| \frac{\d v}{|v|^{d+2s}} &= \underset{v_1=-R}{\overset{|x|}{\int}} \bigg( \underset{\mathcal{E}(x,v_1)}{\int} \frac{\big| \eta_1(x,v) - x \big| }{|v|^{d+2s}} \d v_2 \cdots \d v_d \bigg) \d v_1 \\
&\quad + \underset{v_1=2(1-|x|)}{\overset{R}{\int}} \bigg( \underset{[-R,R]^{d-1}}{\int} \frac{\big| \eta_1(x,v) - x \big| }{|v|^{d+2s}} \d v_2 \cdots \d v_d \bigg) \d v_1 
\end{align*}
and on the other hand
\begin{align*}
\bigg| \underset{\mathcal{E}^c(x)}{\int} \big( \eta_1(x,v) - x \big) \frac{\d v}{|v|^{d+2s}}\bigg| \leq \underset{v_1=0}{\overset{2(1-|x|)}{\int}} \bigg( \underset{\mathcal{E}^c(x,v_1)}{\int} \frac{\big|\eta_1(x,v)-x\big|}{|v|^{d+2s}}\d v_2 \cdots \d v_d \bigg) \d v_1.
\end{align*}
We see that it is the same integrand but in the integral over $\mathcal{E}^c(x)$, the volume of the domain of integration $\big( 0, 2(1-|x|)\big)\times \mathcal{E}^c(x,v_1)$ goes to $0$ as $x$ approaches the boundary whereas the domains $(-R,|x|)\times \mathcal{E}(x,v_1)$ and $\big( 2(1-|x|),R)\times(-R,R)^2$ do not, hence the first term is negligible in the limit before the second and we have
\begin{align*}
|f(x)| \underset{x\rightarrow \dO}{\sim}& \underset{\mathcal{E}(x)}{\int} \big| \eta_1(x,v) - x \big| \frac{\d v}{|v|^{d+2s}} \geq \underset{v_1=-R}{\overset{|x|}{\int}} \bigg( \underset{\mathcal{E}(x,v_1)}{\iint} \frac{v_1 }{|v|^{d+2s}} \d v_2 \d v_3 \bigg) \d v_1 \\
&\geq \underset{1-|x|\leq |v_2|,|v_3|\leq R}{\iint} \bigg( \underset{-R}{\overset{|x|-1}{\int}} \frac{v_1 }{|v|^{d+2s}}\bigg)  \d v_2 \d v_3\\
&\geq \underset{1-|x|\leq |v_2|,|v_3|\leq R}{\iint} \frac{C}{\Big( (1-|x|)^2 + v_2^2 + v_3^2 \Big)^{(d+2s-2)/2}} \d v_2 \d v_3.
\end{align*}
As $x$ approaches the boundary, the integrand tends to $1/(v_2^2 +v_3^2)^{d-1+2s-1}$ and the domain to $[-R,R]^{d-1}$ so the integral diverges since $2s-1>0$.\\
We have proved \eqref{eqball:neumannK} which, together with \eqref{eqball:neumannpsi}, yields Lemma \eqref{lem:Neumann}. 
\end{proof}

\appendix

\section{Free transport equation in a sphere}  \label{app:FreeTransport}

In this appendix, we call $\Omega$ the unit ball in $\RR^d$ and we consider the trajectories in $\Omega$ described by \eqref{eq:etaexplicit} and the associated $\eta$ function. We recall that what we name "\textit{trajectory that starts from $x\in\Omega$ with velocity $v\in\RR^d$}" the trajectory that consists of straight lines, specularly reflected upon hitting the boundary, and that stops when the length of the trajectory is $|v|$, as illustrated by Figure \ref{fig:extraj} in Section \ref{section:sr}. \\
We first note that a trajectory in $\Omega$ is necessarily included in a plane of dimension 2. Indeed, by definition of the specular reflection, when the trajectory hits the boundary, the reflected velocity is a linear combination of the initial velocity and the normal vector: for $t\in [0,1]$ such that $|x+tv|=1$, $Rv = v - 2 ( n(x+tv) \cdot v ) n(x+tv)$ where is $n(x+tv)=x+tv$ because $\Omega$ is the unit ball. Since the normal vector belongs to the plane generated by $x$ and $v$ we see that the reflected velocity also belongs to that same plane, and every reflected velocities along this trajectory. As a consequence, we restrict the study of the regularity of $\eta$ in a ball to the case of a disk in dimension $d=2$.\\

%

\subsection{Explicit expression of the trajectories} \label{subsec:etaballexp}

Consider $(x,v)$ in $\Omega \times \RR^2$, we call $k$ the number of reflections that the trajectory which starts at $x$ with velocity $v$ undergoes. We also introduce
\begin{itemize}
\item $\theta$ such that $[\cos\theta,\sin\theta]$ is the first point of reflection,
\item $A$ the angle between the vector $v$ and the outward normal to $\pa \Omega$ at $[\cos\theta,\sin\theta]$ (which, in the unit ball, is $[\cos\theta,\sin\theta]$ itself),
\item $z_j = \big[\cos\big(\theta + j(\pi -2A)\big), \sin\big(\theta + j(\pi -2A)\big)\big]$ for any $j\in\mathbb{Z}$. Note that $z_0$ is the first point of reflection.
\end{itemize} 
\begin{prop}
For any $k\geq 0$ we have
\begin{equation} \label{eq:expliciteta}
\eta(x,v) = k \big( z_{k-1}-z_k \big) + R_{k(\pi-2A)} (x+v)
\end{equation}
where $R_{k(\pi-2A)}$ is the matrix of the rotation of angle $k(\pi-2A)$.
\end{prop}

\begin{figure}[h]
\centering
\caption{Trajectory with 1 reflection in the circle}
\includegraphics[width=13cm,height=10cm]{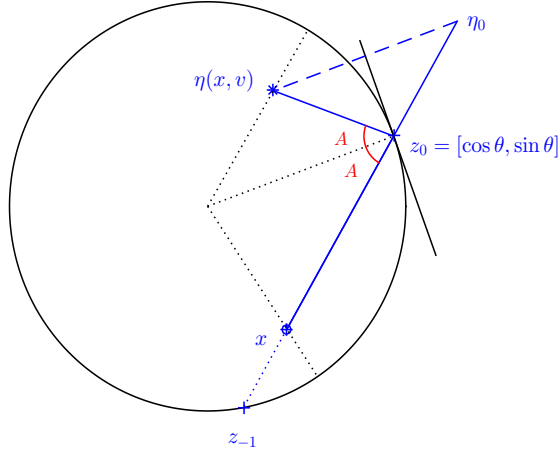}
\label{fig:1refcricle}
\end{figure}

\begin{proof}
We will prove the expressions \eqref{eq:expliciteta} by induction on the number of reflections. When $k=0$, by definition of $\eta$ we have $\eta(x,v) = x+v$ so that \eqref{eq:expliciteta} holds. \\
Let us assume \eqref{eq:etaexplicit} holds for some $k\geq 0$. Then, if we write $\eta_k = k ( z_{k-1}-z_k) + R_{k(\pi-2A)} (x+v) $, we can compute $\eta(x,v)$ after $k+1$ reflections using the relation
\begin{align*}
\eta(x,v) - z_k = R_{\pi-2A} (\eta_k - z_k)
\end{align*}
as illustrated in Figure \ref{fig:1refcricle} in the case $k=0$. By definition of $z_j$ we notice that $R_{\pi-2A} \,z_j = z_{j+1}$ hence:
\begin{align*}
\eta(x,v) &= z_k + R_{\pi-2A} \Big(  k ( z_{k-1}-z_k ) + R_{k(\pi-2A)} (x+v) - z_k \Big)\\
			   &= z_k + k ( z_k - z_{k+1}) + R_{(k+1)(\pi-2A)} (x+v) - z_{k+1}\\
			   &= (k+1) (z_k - z_{k+1}) + R_{(k+1)(\pi-2A)} (x+v)
\end{align*}
which is exactly \eqref{eq:expliciteta}. 
\end{proof}

\subsection{First and second derivatives} \label{subsec:etaballder}

We recall that $\mathfrak{D}_T$ is defined as:
\begin{equation*}
\mathfrak{D}_T(\Omega) = \Big\{ \psi\in\mathcal{C}^\infty ([0,T)\times\bO) \text{ s.t. } \psi(T,\cdot)=0 \text{ and } \forall x\in\dO : \na_x \psi(t,x)\cdot n(x) =0 \Big\}.
\end{equation*} 
This section is devoted to the proof of the following estimates on the Jacobian matrix and the second derivative of $\eta$:
\begin{lemma} \label{lemma:EtaBall}
Consider the unit ball $\Omega$. The associated function $\eta$, defined in Section \ref{sec:AuxProbSR}, satisfies
\begin{align} \label{eq:controlnaveta}
 \lVert \na_v \eta (x,v) \lVert \in L^\infty (\Omega\times\RR^d) 
\end{align}
and for all $\psi$ is in $\mathfrak{D}_T$ 
\begin{align} \label{eq:controlD2psieta}
\Big\lVert D^2_v \Big[ \psi\big( \eta(x, v )\big)\Big] \Big\lVert \in L_{F(v)}^p (\Omega\times\RR^d)
\end{align}
for $p<3$ where $\lVert \cdot\lVert$ is a matrix norm. Moreover, 
\begin{align} \label{eq:controlD2psietadelta}
\underset{v\in\RR^d}{\sup} \Big\lVert D^2_v \Big[ \psi\big( \eta(x, v )\big)\Big] \Big\lVert \in L^{2-\delta} (\Omega)
\end{align}
for any $\delta>0$. 
\end{lemma}

\begin{proof}
When $k=0$, we have immediately $\na_v \eta= Id$ and the controls stated in the Lemma follow. When $k\geq 1$ we notice that for all $j$, $ z_j = R_{k(\pi-2A)} [z_{j-k}]$ so that we have
\begin{align*}
\eta(x,v) = R_{k(\pi-2A)} \big(x + v - k(z_0-z_{-1}) \big)
\end{align*}
where $z_0$ and $z_{-1}$ are illustrated in Figure \ref{fig:1refcricle}. Also, we introduce the matrix $S = 
\begin{pmatrix} 0 & -1 \\ 1 &0 \end{pmatrix}$ for the rotation of angle $\pi/2$ -- note that it commutes with the rotation matrix $R_{k(\pi-2A)}$ -- and with which the Jacobian matrix of $\eta$ with respect to $v$ takes the form
\begin{align*}
\na_v \eta(x,v) &= \Big[ S R_{k(\pi-2A)} \big( x+v - k(z_0-z_{-1}) \big) \Big] \\
&\quad \quad  \otimes \big( k \na_v (\pi-2A) \big) + R_{k(\pi-2A)} \na_v \Big( x + v - k(z_0-z_{-1}) \Big)\\
&= \Big[ S R_{k(\pi-2A)} \big( x+v - k(z_0-z_{-1}) \big) \Big] \\
&\quad \quad \otimes \big( -2k \na_v A \big) - k R_{k(\pi-2A)} \na_v\big(z_0-z_{-1}\big) + R_{k(\pi-2A)}.
\end{align*}
Now, to differentiate the angles $\theta$ and $A$ with respect to $v=(v_1,v_2)$, let us recall that for $t$ such that $|x+tv|=1$ we have
\begin{equation*}
\left\{ \begin{aligned} &x_1 + t v_1 = \cos \theta\\
									&x_2 + t v_2= \sin \theta \end{aligned} \right. 
\end{equation*}
so that $ v_2 (\cos \theta -x_1) = v_1 (\sin\theta - x_2)$, hence:
\begin{equation*}
\frac{\pa \theta}{\pa v_1} = \frac{x_2-\sin\theta}{v_1\cos\theta + v_2\sin\theta} = \frac{- tv_2}{|v|\cos A}, \hspace{0.9cm} \frac{\pa \theta}{\pa v_2} = \frac{\cos\theta-x_1}{v_1\cos\theta + v_2\sin\theta} = \frac{tv_1}{|v| \cos A}.
\end{equation*}
Moreover, $t$ satisfies $|v| \cos A = (x+tv)\cdot v = x\cdot v + t|v|^2$ which means
\begin{equation} \label{eq:patheta}
\frac{\pa \theta}{\pa v_1} = \frac{-v_2}{|v|} \frac{1}{|v|\cos A} \bigg( \cos A  -  \frac{ x\cdot v}{|v|}\bigg) , \hspace{1cm}  \frac{\pa \theta}{\pa v_2} = \frac{v_1}{|v|} \frac{1}{|v|\cos A} \bigg( \cos A  -  \frac{ x\cdot v}{|v|} \bigg).
\end{equation}
Also, by definition of $A$ we have: $|v| \sin A = (x+tv)\times v = x_1 v_2 - x_2 v_1$ therefore:
\begin{align} 
\frac{\pa A}{\pa v_1} &= \frac{ -v_1(x_1v_2-x_2 v_1 ) - x_2( v_1^2 + v_2^2)}{|v|^3\cos A}, \hspace{0.1cm} \frac{\pa A}{\pa v_2}= \frac{ -v_2(x_1v_2-x_2 v_1 ) + x_1( v_1^2 + v_2^2) }{|v|^3 \cos A}  \nonumber \\
&= \frac{-v_2}{|v|} \frac{1}{|v|\cos A}  \bigg( \frac{ x\cdot v}{|v|}\bigg) \hspace{4cm} = \frac{v_1}{|v|} \frac{1}{|v|\cos A}  \bigg( \frac{ x\cdot v}{|v|}\bigg). \label{eq:paA}
\end{align}
We now introduce the notations $l_{in}$, $L$ and $l_{end}$ defined as follows and illustrated in Figure \ref{fig:muxv}
\begin{itemize} 
\item $l_{in}$ is the distance between $x$ and the first point of reflection $z_0$:
\begin{align*}
l_{in} = t|v| = \cos A - \frac{x\cdot v}{|v|}.
\end{align*}
\item $L$ is the length between two consecutive reflections (note that it is constant because $\Omega$ is a ball):
\begin{align*}
L = 2\cos A
\end{align*}
\item $l_{end}$ is the length between the last point of reflection and the end of the trajectory, $\eta(x,v)$:
\begin{align*}
l_{end} = |v| - (k-1) L - l_{in}.
\end{align*}
\end{itemize}
\begin{figure}[h]
\centering
\caption{Notations $l_{in}$, $L$ and $l_{end}$}
\includegraphics[width=12cm,height=9cm]{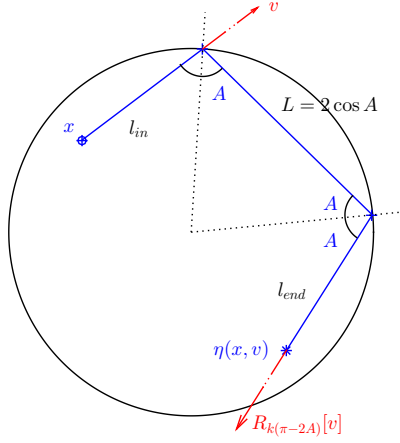}
\label{fig:muxv}
\end{figure}
With these notations, the gradients of $\theta$ and $A$ read
\begin{align} \label{eq:gradthA}
\na_v \theta = \bigg(\frac{2l_{in}}{|v| L} \bigg) S\frac{v}{|v|}, \hspace{2cm} \na_v A = \bigg( \frac{L-2l_{in}}{|v| L}  \bigg)S\frac{v}{|v|}
\end{align}
hence the Jacobian matrices of $z_0$ and $z_1$ as functions of $v$ are
\begin{align*}
&\na_v z_0 = S z_0 \otimes \na_v \theta =  \bigg( \frac{2l_{in}}{|v| L}  \bigg) Sz_0 \otimes S\frac{v}{|v|} ,\\
&\na_v z_{-1} = S z_{-1} \otimes \na_v \big( \theta - (\pi-2A)\big) =  \bigg( \frac{2(L-l_{in})}{|v| L} \bigg) S z_{-1} \otimes S \frac{v}{|v|}.
\end{align*}
Therefore, we have
\begin{align*}
\na_v \eta(x,v) &= SR_{k(\pi-2A)} \Bigg[ \bigg(\frac{2k (2l_{in}-L)}{|v|L} \bigg) \big( x + v - k (z_0-z_{-1}) \big) \Bigg] \otimes S\frac{v}{|v|} \\
&\quad - k S R_{k(\pi-2A)} \Bigg[ \bigg( \frac{2l_{in}}{|v|L}\bigg)  z_0 - \bigg( \frac{2(L-l_{in})}{|v|L} \bigg) z_{-1} \Bigg] \otimes S\frac{v}{|v|} + R_{k(\pi-2A)}\\
&=  \frac{2k}{|v|L} SR_{k(\pi-2A)} \bigg[  (2 l_{in} -L) \big(x +v - k(z_0 - z_{-1})\big)  \\
&\quad -(l_{in}-\frac{L}{2}) (z_0+z_{-1})  - \frac{L}{2} (z_0-z_{-1}) \bigg] \otimes S\frac{v}{|v|} + R_{k(\pi-2A)}\\
&= \frac{2k}{|v|L} SR_{k(\pi-2A)} \bigg[  \frac{1}{2} (2 l_{in} -L) (2x - z_0 - z_{-1}) \\
&\quad + (2l_{in} -L)\big(v-k(z_0-z_{-1})\big)- \frac{L}{2}(z_0-z_{-1} ) \bigg] \otimes S\frac{v}{|v|} + R_{k(\pi-2A)}.
\end{align*}
Finally, by definition of $z_0$ and $z_{-1}$ we see that
\begin{equation} \label{eq:distxzi}
\left\{ \begin{aligned}
&z_0-z_{-1} = L \frac{v}{|v|},\\
&x-z_0 = - l_{in} \frac{v}{|v|},\\
&x-z_{-1} = (L-l_{in} ) \frac{v}{|v|}
\end{aligned} \right.
\end{equation}
which yields
\begin{equation} \label{eq:naveta} 
\na_v \eta(x,v) = \frac{2kL}{|v|} \bigg[ 2\frac{l_{in}}{L} \frac{l_{end}}{L} - \frac{l_{in}+ l_{end}}{L} \bigg] SR_{k(\pi-2A)}\frac{v}{|v|} \otimes S \frac{v}{|v|}  + R_{k(\pi-2A)} .
\end{equation}
Introducing the notation 
\begin{align*}
\bv = \frac{v}{|v|}
\end{align*}
as well as the angular function $\Theta : \mathbb{S}^1 \mapsto \mathcal{M}_2(\RR)$ and the function $\mu_x:\RR^2\mapsto \RR$ as
\begin{align}
&\Theta (\bv) = S R_{k(\pi-2A)} \bv \otimes S\bv. \label{eq:thetafct} \\
&\mu_x(v) = \frac{2kL}{|v|} \bigg[ 2 \frac{l_{in}}{L} \frac{l_{end}}{L} -\frac{ l_{in} + l_{end}}{L}  \bigg] .\label{eq:muxv}
\end{align} 
we have 
\begin{align} \label{eq:naveta2}
\na_v \eta(x,v) = \mu_x(v) \Theta(\bv) + R_{k(\pi-2A)}.
\end{align}
Now, since $|v| = l_{in} + (k-1)L + l_{end} $ we see that when $k>1$:
\begin{align*}
&\frac{kL}{|v|} = \frac{|v|+L-l_{in}-l_{end}}{|v|} \leq 1 + \frac{|L-l_{in}-l_{end}|}{|v|} \leq 2 
\end{align*}
and also, since $0\leq l_{in}, l_{end} \leq L$ we have
\begin{align*}
-1 \leq 2\frac{l_{in}}{L} \frac{l_{end}}{L} - \frac{l_{in}+ l_{end}}{L} \leq 0
\end{align*}
so that 
\begin{equation} \label{eq:mubounded}
-4\leq -2 - 2 \frac{|L-l_{in}-l_{end}|}{|v|}  \leq \mu_x(v)  \leq 0.
\end{equation}
Since $\lVert R_{k(\pi-2A)} \lVert =\lVert S\lVert= 1$, $\na_v \eta $ is bounded uniformly in $x$ and $v$ which concludes the proof of the control of $\na_v \eta$ stated in Lemma \ref{lemma:EtaBall}. Notice that it also yields an explicit expression for the determinant:
\begin{align} \label{eq:detnaveta}
\det \na_v \eta((x,v) = 1 + \frac{2kL}{|v|} \bigg[ 2 \frac{l_{in}}{L} \frac{l_{end}}{L} -\frac{ l_{in} + l_{end}}{L}  \bigg] 
\end{align}
from which is it easy to see that
\begin{align*}
-3 \leq -1 - 2 \frac{| L-l_{in}-l_{end}| }{|v|} \leq \det \na_v \eta (x,v)\leq 1.
\end{align*}

For the second derivative, we first notice that the expression of the Jacobian matrix above depends strongly on $k$ and is not continuous when we go from $k$ to $k+1$ which is equivalent to $l_{end}$ going to $L$. Hence, we introduce the sets $E_k$ defined as
\begin{align*}
E_k &= \lbrace (x,v)\in\Omega\times\RR^d \text{ s.t. the trajectory from } (x,v) \text{ undergoes}\\ 
&\hspace{1cm} \text{ exactly } k \text{ reflections} \rbrace
\end{align*}
and the Jacobian of $\eta$ actually reads
\begin{align*}
\na_v \eta(x,v) = \underset{k\in\mathbb{N}}{\sum} \na_v \eta^k (x,v) \mathds{1}_{E_k}
\end{align*}
where $\eta^k$ is the expression \eqref{eq:naveta}. The second derivative of $\eta$ will involve a derivative of the indicator functions of the $E_k$ sets, i.e. the dirac measure of the boundary $\pa E_k$ in the direction of the discontinuity. However, the boundary of $E_k$ corresponds, by definition, to the $(x,v)$ such that $\eta(x,v)$ is on $\dO$. Hence, similarly to the half-space case (see Section \ref{subsubsec:lemmaCVSRhalfspace}) if we consider $\psi\in\mathfrak{D}_T$ then the direction of the jump will be orthogonal to $\na \psi$ at that point on $\dO$ and their product will be naught. \\
For the rest of this proof, we omit the dependence of $\eta^k$ with respect to $k$. Before computing $D^2_v \eta$ which we define as usual as:
\begin{align} \label{eq:D2formula}
D^2_v \eta (x,v) = \begin{pmatrix} \pa_{11}^2 \eta & \pa_{12}^2 \eta \\ \pa_{21}^2\eta & \pa_{22}^2 \eta  \end{pmatrix}
\end{align}
where $\pa_{ij}^2$ means the second order partial derivative with respect to $v_i$ and $v_j$, we feel it is simpler, given the form of the Jacobian matrix,  to compute $\na_v \times \na_v \eta$ where we define the product $\times$ between a vector $u$ in $\RR^2$ and a matrix $M = (m_{ij})_{1\leq i,j\leq 2}$ in $\mathcal{M}_2(\RR)$ as
\begin{equation*}
 u \times M = \left( \begin{matrix} m_{11} u & m_{12} u \\ m_{21} u & m_{22} u \end{matrix} \right) 
\end{equation*}
which means the product $u\times M $ is a vector valued matrix in $\mathcal{M}_2(\RR^2)$. We write $\na_v \eta = (\pa_j \eta_i)_{i,j}$ and have
\begin{align} \label{eq:nanaetaformula}
\na_v \times \na_v \eta = \begin{pmatrix} \na_v \pa_1 \eta_1 & \na_v \pa_2 \eta_1 \\ \na_v \pa_1 \eta_2 & \na_v \pa_2 \eta _2   \end{pmatrix}.
\end{align}
Using expression \eqref{eq:naveta} we have:
\begin{equation} \label{eq:nanaetaexplicit}
\na_v \times \na_v \eta(x,v) = \na_v \mu_x(v) \times \Theta (\bv) + \mu_x(v) \na_v \times \Theta (\bv) + 2k \na_v A \times S R_{k(\pi-2A)}
\end{equation}
Let us look at each of the terms individually and focus on singularities that might cause trouble for the integrability in $L^2_{F(v)}(\Omega\times\RR^2)$, which in fact will arise when we get close to the grazing set, i.e. when $L$ (as well as $l_{in}$ and $l_{end}$) goes to 0 or, equivalently, when $k$ goes to infinity. The simplest term to handle is the last one since we have, using \eqref{eq:gradthA}:
\begin{align} \label{eq:D2navA}
2k \na_v A  =  \frac{1}{L}\bigg( \frac{2k\big(L-2l_{in}\big) }{|v|} \bigg) S \bv
\end{align}
so that 
\begin{align*}
2k\na_v A \times SR_{k(\pi-2A)} := \frac{\alpha_A}{L} S\bv \times SR_{k(\pi-2A)}
\end{align*}
where $\alpha_A$ is uniformly bounded in $x$ and $v$. For the second term, 
\begin{align*}
\na_v \times \Theta (\bv) &= \na_v\times  \Big( S R_{k(\pi-2A)} \bv\otimes S \bv \Big)
\end{align*}
we introduce the extension of the dyadic product defined, for $u\in\RR^2$ and $M\in\mathcal{M}_2(\RR)$ as:
\begin{align*}
u \otimes M = \begin{pmatrix}  u_1 \left[ \begin{matrix} m_{11} \\ m_{12} \end{matrix} \right] & u_1 \left[ \begin{matrix} m_{21} \\ m_{22} \end{matrix} \right]  \\[8pt]   u_2 \left[ \begin{matrix} m_{11} \\ m_{12} \end{matrix} \right]  & u_2 \left[ \begin{matrix} m_{21} \\ m_{22} \end{matrix} \right]     \end{pmatrix}
\end{align*}
which is rather natural if one notices that for two vectors $u$ and $v$, $u\otimes v = u \, v^T$, and we also define its commuted form $M\otimes u = ( u \otimes M)^T$. With these notation, we have
\begin{align*}
\na_v \times \Theta (\bv) = \big( \na_v SR_{k(\pi-2A)} \bv \big) \otimes S\bv + SR_{k(\pi-2A)} \bv \otimes \big( \na_v S\bv\big)
\end{align*}
where on the one hand
\begin{align*}
\na_v S\bv = \frac{-1}{|v|} \bv \otimes S\bv
\end{align*}
and on the other hand
\begin{align*}
\na_v \big( SR_{k(\pi-2A)} \bv\big) &= S\Big( SR_{k(\pi-2A)} \bv \otimes \na_v \big( k(\pi-2A) \big) + R_{k(\pi-2A)} \na_v \bv \Big)\\
														&= \frac{1}{L} \bigg( \frac{2k(L-2l_{in})}{|v|} - \frac{L}{|v|} \bigg) R_{k(\pi-2A)} \bv \otimes S\bv.
\end{align*}
We get
\begin{align} \label{eq:D2navTheta}
\na_v \times \Theta(\bv) &= \frac{1}{L} \bigg( \frac{2k(L-2l_{in})}{|v|} - \frac{L}{|v|} \bigg) \Big( R_{k(\pi-2A)} \bv \otimes S\bv\Big) \otimes S\bv \nonumber \\
& - \frac{1}{|v|} SR_{k(\pi-2A)} \bv \otimes \Big( \bv \otimes S\bv \Big) 
\end{align}
so that 
\begin{align*}
\mu_x (v) \na_v \times \Theta(\bv) = \frac{\alpha_\theta}{L} ( R_{k(\pi-2A)} \bv \otimes S\bv ) \otimes S\bv + O(1)
\end{align*}
when we are close to the grazing where, once again, $\alpha_\theta$ is uniformly bounded in $x$ and $v$. Note, in fact, that $\alpha_\theta = \alpha_A + O(L)$. Let us also note that the extension of the dyadic we defined is not quite associative in the sense that if $u$, $v$ and $w$ are vectors then
\begin{align*}
(u \otimes v ) \otimes w = u \otimes ( w\otimes v) 
\end{align*}
which we will keep in mind when we compute $D^2\eta(x,v)$. Finally, for the first term in \eqref{eq:nanaetaexplicit} we notice that since $l_{in} = |x-z_0|=\sqrt{1+|x|^2-2x\cdot z_0}$ we have
\begin{align*}
\na_v l_{in} &= \frac{-2 \na_v ( x \cdot z_0) }{l_{in}} = \frac{-2 x\cdot Sz_0}{l_{in}} \na_v \theta = \frac{-4 x\cdot Sz_0}{|v| L	} S\bv
\end{align*} 
where $x\cdot S z_0 = x\cdot S(x+tv) = t|v| x\cdot Sv/|v| = l_{in} \sin A$ so that in fact
\begin{align*}
\na_v l_{in} = \frac{-4 l_{in} \sin A}{|v| L	} S\bv.
\end{align*} 
Moreover, $L=2\cos A$ so we have 
\begin{align*}
\na_v L = \frac{-2(L-2l_{in})\sin A}{|v|L} S\bv
\end{align*}
and finally, $l_{end} = |v| - (k-1) L - l_{in}$ therefore
\begin{align*}
\na_v l_{end} &=\bv + \frac{1}{|v| L} \big( 2(k-1) (L-2 l_{in} ) \sin A + 4 l_{in} \sin A \big) S\bv\\
						&= \bv + \frac{2 \sin A }{L } \bigg( \frac{(k-1)(L-2l_{in})}{|v|} + \frac{2 l_{in}}{|v|} \bigg) S\bv.
\end{align*}
Note that unlike $\na_v L$ and $\na_v l_{in}$, the gradient of $l_{end}$ diverges in norm for small $L$ (i.e. close to the grazing set) because the coefficient $\sin A / L$ goes to infinity. Differentiating $\mu_x(v)$ we get
\begin{align}
&\na_v \mu_x(v) = \na_v \bigg( \frac{2kL}{|v|} \bigg) \bigg[ 2 \frac{l_{in}}{L} \frac{l_{end}}{L} -\frac{ l_{in} + l_{end}}{L} \bigg] \nonumber \\ 
&+ \frac{2kL}{|v|} \bigg[ \na_v l_{end} \Big( 2\frac{l_{in}}{L^2} - \frac{1}{L} \Big) + \na_v l_{in} \Big(2\frac{l_{end}}{L^2} - \frac{1}{L} \Big) + \na_v L \Big( \frac{l_{in}+l_{end}}{L^2} - 4 \frac{l_{in} l_{end}}{L^3} \Big) \bigg] \nonumber\\
&= \frac{-1}{|v|L} \mu_x(v) \Big( 2\Big(1-2\frac{l_{in}}{L} \Big) S\bv - L \bv \Big) \nonumber \\
&\quad + \frac{1}{L^2} \bigg(\frac{2kL}{|v|}\bigg) \bigg(2\frac{l_{in}}{L} -1 \bigg) \bigg[L \bv + 2 \sin A  \bigg( \frac{(k-1)(L-2l_{in})}{|v|} + \frac{2 l_{in}}{|v|} \bigg) S\bv \bigg] \nonumber \\
&\quad + \frac{1}{L} \frac{2kL}{|v|^2 } \bigg[ \frac{4l_{in} \sin A}{L}  \bigg( 1 - 2 \frac{l_{end}}{L} \bigg) - 2\sin A \bigg(1 - 2\frac{l_{in}}{L}\bigg) \bigg( \frac{l_{in}+l_{end}}{L} - 4 \frac{l_{in} l_{end}}{L^2} \bigg) \bigg] S\bv .\label{eq:D2navmu}
\end{align}
Introducing uniformly bounded functions $\alpha^i_\mu$, $i\in\lbrace1,2,3\rbrace$ we get
\begin{align*}
\na_v \mu_x (v) \times \Theta(\bv) &= \frac{1}{L} \Big( \alpha_\mu^1 S\bv + \alpha_\mu^2 \bv \Big)  \times ( SR_{k(\pi-2A)} \bv \otimes S\bv ) \\
& + \frac{1}{L^2} \alpha_\mu^3 S\bv \times ( SR_{k(\pi-2A)} \bv \otimes S\bv ) .
\end{align*}
Together, all three terms yields 
\begin{align*}
\na_v \times \na_v \eta(x,v)& = \frac{1}{L} \Big( \alpha_\mu^1 S\bv + \alpha_\mu^2 \bv \Big)  \times ( SR_{k(\pi-2A)} \bv \otimes S\bv )  \\
&+ \frac{1}{L^2} \alpha_\mu^3 S\bv \times ( SR_{k(\pi-2A)} \bv \otimes S\bv )  \\
&+   \frac{\alpha_\theta}{L} ( R_{k(\pi-2A)} \bv \otimes S\bv ) \otimes S\bv  + \frac{\alpha_A}{L} S\bv \times SR_{k(\pi-2A)} + O(1)
\end{align*}
Identifying the terms in \eqref{eq:D2formula} with those of \eqref{eq:nanaetaformula} we get
\begin{align}\label{eq:D2explicitgrazing}
D^2 \eta(x,v) &= \frac{1}{L} SR_{k(\pi-2A)} \bv \times \Big( \alpha_\mu^1 (S\bv \otimes S\bv) + \alpha_\mu^2 (\bv \otimes S\bv) + \frac{1}{L} \alpha_\mu^3 (S\bv \otimes S\bv ) \Big)\nonumber \\
&+ \frac{1}{L} \Big(  \alpha_{\theta} R_{k(\pi-2A)}\bv \times (S\bv \otimes S\bv) + \alpha_A S\bv \otimes SR_{k(\pi-2A)} C  \Big) + O(1)
\end{align} 
where $C$ is the conjugation matrix: $C = \begin{pmatrix} 1 & 0 \\ 0 & -1 \end{pmatrix}$. Now, if we want to integrate $1/L$ in $L^p_{F(v)}(\Omega\times\RR^2)$ for some $p>0$ we first write $L$ in terms of $x$ and $v$ using the relations $L=2\cos A$, $|v| \cos A = x\cdot v + t |v|^2$ and the fact that $t$ solve $|x+tv|^2=1$ which yield
\begin{align} \label{eq:Lgrazing}
L = 2 \sqrt{ \big( x\cdot \bv \big)^2 + \big(1-|x|^2\big) }.
\end{align}
Therefore using polar change of variables
\begin{align*}
& \underset{\Omega\times\RR^2}{\iint} \Big(\frac{2}{L}\Big)^p F(v) \d xv =  \underset{\Omega\times\RR^2}{\iint} \frac{1}{\Big( (x\cdot\bv)^2 + (1-|x|^2) \Big)^{p/2}} F(v)\d xv \\
&\hspace{1cm} = \int_0^{1} \int_0^{2\pi} \int_0^{2\pi} \frac{\rho_x}{\big( 1- \rho_x^2 \sin^2(\theta_v-\theta_x) \big)^{p/2}} \, \text{d}\rho_x \, \text{d}\theta_x \, \text{d}\theta_v  \int_{\RR^+} F(\rho_v) \rho_v \text{d}\rho_v 
\end{align*}
where, since $F$ is radial and normalized, $\int_{\RR^2} F(v) \d v = 2\pi \int_{\RR} \rho_v F(\rho_v) \text{d}\rho_v = 1$. Note, in fact, that since $L$ does not depend on the norm of $|v|$, the integrability in $L^2_{F(v)}(\Omega\times\RR^2)$ is equivalent to the integrability in $L^2(\Omega\times\mathbb{S}^{1} )$ where $\mathbb{S}^1$ is the unit circle in $\RR^2$. Expanding the denominator, we have
\begin{align}
&\underset{\Omega\times\RR^2}{\iint} \Big(\frac{2}{L}\Big)^{p} F(v) \d xv \\
&= \frac{1}{2\pi} \int_0^{1} \int_0^{2\pi} \int_0^{2\pi} \frac{\rho_x}{\big( 1- \rho_x| \sin(\theta_v-\theta_x)| \big)^{p/2} \big( 1+ \rho_x| \sin(\theta_v-\theta_x)| \big)^{p/2}} \, \text{d}\rho_x \, \text{d}\theta_x \, \text{d}\theta_v\nonumber\\
&\leq  C  \int_0^{1} \int_0^{2\pi} \int_0^{2\pi} \frac{\rho_x}{\big( 1- \rho_x| \sin(\theta_v-\theta_x)| \big)^{p/2}} \, \text{d}\rho_x \, \text{d}\theta_x \, \text{d}\theta_v\nonumber\\
&\leq   2 \pi C  \int_0^{1} \int_0^{2\pi} \frac{\rho_x}{\big( 1- \rho_x| \sin\alpha| \big)^{p/2}} \, \text{d}\rho_x \, \text{d}\alpha \nonumber\\
&\leq  \tilde{C}  \int_0^1 \int_0^{\sqrt{1-x_2^2}} \frac{1}{\big( 1 - x_2)^{p/2}} \d x_1 \d x_2 \nonumber \\
&\leq  \tilde{C} \int_0^1 \frac{1}{(1-x_2)^{p/2-1/2}} \d x_2  \label{eq:Lintegrability}
\end{align}
hence $1/L$ will be in $L^p_{F(v)}(\Omega\times\RR^2)$ if $p < 3$. \\
Moreover, if we take $\psi$ in $\mathfrak{D}_T$ (defined in beginning of this section) then we have 
\begin{align*}
D^2_v \Big[ \psi\big( \eta(x,v)\big)\Big]  = D^2\eta(x,v) \na\psi\big(\eta(x,v)\big) + \na_v\eta(x,v)^T D^2_v \psi\big(\eta(x,v)\big) \na_v\eta(x,v).
\end{align*}
The second term is uniformly bounded in $x$ and $v$ by \eqref{eq:controlnaveta} so it belongs to $L^p_{F(v)}(\Omega\times\RR^2)$ for any $p\leq \infty$. Furthermore, for the first term, we notice that for any $u\in\RR^2$ and $M\in \mathcal{M}_2(\RR)$ we have
\begin{align*}
u \times M \na \psi = (u \cdot \na \psi ) M.
\end{align*}
Thus the first term reads
\begin{align*}
&D^2\eta(x,v)  \na\psi\big(\eta(x,v)\big) \\
&=\frac{1}{L} \Big( \alpha_A (S\bv \otimes SR_{k(\pi-2A)} C)\na\psi\big(\eta(x,v)\big) + \alpha_\theta \big[ R_{k(\pi-2A)}\bv \cdot \na\psi\big(\eta(x,v)\big) \big] S\bv \otimes S\bv \Big)\\
&\quad + \frac{1}{L} \big[ SR_{k(\pi-2A)}\bv \cdot \na\psi\big(\eta(x,v)\big) \big] \Big( \alpha_\mu^1 (S\bv \otimes S\bv) + \alpha_\mu^2 (\bv \otimes S\bv) + \frac{1}{L} \alpha_\mu^3 (S\bv \otimes S\bv ) \Big) \\
& \quad + O(1).
\end{align*}
Recall that on the boundary, $\na \psi(x,v) \cdot n(x) =0$ hence, by the regularity of $\psi$, when $\eta(x,v)$ is close the boundary we have
\begin{align*}
\na \psi\big(\eta(x,v)\big) = \tilde{\tau}\big( \eta(x,v) \big) + O\Big(\text{dist}\big( \eta(x,v) , \dO \big)\Big)
\end{align*}
where $\tilde{\tau}$ is the extension of the tangent $\tau(x)$ of $\dO$ at $x\in\dO$ which, since we are in the unit ball, is explicitly $\tilde{\tau}\big(\eta(x,v)\big) = \tau\big( \eta(x,v)/ |\eta(x,v)| \big)$ when $|\eta(x,v)| \neq 0$. Moreover, when we start close to the grazing set the trajectory stays close to the grazing set (because $A$ is constant close to $\pi/2$), which means $R_{k(\pi-2A)}\bv$ stays close to $\tilde{\tau}(\eta(x,v))$ and in fact it will be furthest from the tangent when $\eta(x,v)$ is on the boundary where we have
\begin{align*}
R_{k(\pi-2A)} \bv &= \big( \cos A \big)  n\big(\eta(x,v)\big) + \big( \sin A  \big) \tau\big(\eta(x,v)\big)\\
			&= \Big( \frac{1}{2} L \Big) n\big(\eta(x,v)\big) + \Big(1- \frac{L^2}{4}\Big)^{1/2} \tau \big(\eta(x,v)\big)
\end{align*}
so that we have
\begin{align*}
SR_{k(\pi-2A)}\bv = n\big(\eta(x,v)\big)  + O(L).
\end{align*}
Finally, we can bound the distance between $\eta(x,v)$ and the boundary in terms of $L$ because we are in a circle so the $\eta(x,v)$ is furthest from the boundary when $l_{end} = L/2$ and the Pythagorean theorem tells us in that case
\begin{align*}
\Big( 1 - \text{dist}\big(\eta(x,v), \dO \big) \Big)^2 + \Big( \frac{L}{2} \Big)^2 = 1
\end{align*}
so that we have all along the trajectory
\begin{align*}
\text{dist}\big(\eta(x,v), \dO\big) = 1 - \sqrt{1- \frac{L^2}{4}}  \underset{L\ll 1}{=} \frac{L^2}{4} + o(L^2).
\end{align*}
All together, these estimates yields
\begin{align*}
SR_{k(\pi-2A)}\bv \cdot \na \psi\big(\eta(x,v)\big) \underset{L\ll 1}{=} O(L)
\end{align*}
so that 
\begin{align} \label{eq:D2etanapsi}
D^2\eta(x,v) &\na\psi\big(\eta(x,v)\big) \underset{L\ll 1}{=} \frac{1}{L} \Big( \alpha_A ( S\bv \otimes SR_{k(\pi-2A)}C ) \na\psi\big(\eta(x,v)\big) \\
&\hspace{1cm} + \alpha_\theta \Big(R_{k(\pi-2A)} \bv \cdot \na\psi\big(\eta(x,v)\big) \Big)   S\bv \otimes S\bv + \alpha_\mu^3 S\bv \otimes S\bv \Big)  + O(1). \nonumber
\end{align}
and from \eqref{eq:Lintegrability} it follows in particular that $\big\lVert D^2\eta(x,v) \na \psi \big(\eta(x,v)\big)\big\lVert   \in L_{F(v)}^{p} (\Omega\times\RR^2)$ for all $p<3$ where $\lVert \cdot \lVert$ is any matrix norm. \\
However this integrability does not hold uniformly in $v$. Indeed, if we take the supremum over $v$ in $\RR^d$ of the second derivative then, close to the boundary, it behave like $1/L = 1/ \sqrt{ 1-|x|^2}$ which is in $L^{2-\delta}(\Omega)$ for any $\delta>0$ but not in the limit when $\delta=0$, as stated in \eqref{eq:controlD2psietadelta}. 
\end{proof}

\subsection{Fractional Laplacian along the trajectories} \label{app:subsecDels}
This section of the Appendix is devoted to the proof of the following Lemma which follows from Lemma \ref{lemma:EtaBall}:
\begin{replemma}{lemma:delsint}
There exists $p>2$ such that 
\begin{align*}
\Delsv \Big[ \psi\big( t,\eta(x, v) \big)\Big]  \in L^p_{F(v)} (\Omega\times\RR^d).
\end{align*}
\end{replemma}
\begin{proof}
As we did several times before in this paper, we can split the integral formulation of the fractional Laplacian, for $R>0$, as follows
\begin{align*}
\Delsv \Big[ \psi\big( t,\eta(x, v) \big)\Big] &= c_{d,s} P.V. \underset{|w|\leq R}{\int} \frac{\psi\big(\eta(x,v)\big) - \psi\big( \eta(x,v+w)\big)}{|w|^{d+2s}} \d w \\
&\hspace{1cm}+ c_{d,s} \underset{|w|\geq R}{\int} \frac{\psi\big(\eta(x,v)\big) - \psi\big( \eta(x,v+w)\big)}{|w|^{d+2s}} \d w 
\end{align*}
and the integral over $|w|\geq R$ is immediately integrable in $L^p_{F(v)}(\Omega\times\RR^d)$ for any $p$ thanks to the boundedness of $\psi$ in $L^\infty(\Omega)$ and the fact that $F$ is normalized. For the integral over $w\leq R$ we do a second order Taylor-Lagrange expansion, as we did for $\chi_x$ in section \ref{subsubsec:lemmaCVSRhalfspace}, in order to write for some $z$ and $\tilde{z}$ in the ball centred at $v$ of radius $|w|$:
\begin{align*}
&P.V. \underset{|w|\leq R}{\int} \frac{\psi\big(\eta(x,v)\big) - \psi\big( \eta(x,v+w)\big)}{|w|^{d+2s}} \d w \\
\quad &= \frac{1}{2} \underset{|w|\leq R}{\int} \frac{w \cdot \Big( D^2\big[ \psi(\eta(x,\cdot))\big] (z) + D^2\big[ \psi(\eta(x,\cdot))\big] (\tilde{z}) \Big) w  }{|w|^{d+2s}} \d w.
\end{align*}
Let us focus on the term with $z$, the one with $\tilde{z}$ can obviously be handled similarly. Using \eqref{eq:D2etanapsi} we have through straightforward computation
\begin{align*}
w \cdot  D^2\Big[ \psi\big(\eta(x,\cdot)\big)\Big] (z)  w &= \frac{1}{L} \bigg[ \alpha_A (w\cdot S\bz ) \Big( \na \psi(\eta(x,z)) \cdot SR_{k(\pi-2A)} w \Big) \\
& +  \Big(\alpha_\theta R_{k(\pi-2A)} \bz \cdot  \na\psi\big(\eta(x,z)\big) + \alpha_\mu^3 \Big) (S \bz \cdot w )^2 \bigg] + C|w|^2
\end{align*}
where $\bz =z/|z|$ and $C=C(x,z)$ is uniformly bounded in $x$ and $z$. Introducing $\bw = w/|w|$ as well as $\lambda_1$ and $\lambda_2$ to simplify the notations, this yields
\begin{align*}
&w \cdot  D^2\Big[ \psi\big(\eta(x,\cdot)\big)\Big] (z)  w \\
\quad &= \frac{|w|^2 }{L} \Big(\bw\cdot S\bz\Big) \Big( \lambda_1 (S\bz\cdot \bw) + \lambda_2 (SR_{k(\pi-2A)} \bw \cdot  \na\psi\big(\eta(x,z)\big)  \Big) + C|w|^2.
\end{align*}
Therefore, using \eqref{eq:Lgrazing} we have
\begin{align*}
&\bigg| \underset{|w|\leq R}{\int} \frac{w \cdot D^2\big[ \psi(\eta(x,\cdot))\big] (z) w  }{|w|^{d+2s}} \d w \bigg|\\
\quad &=\bigg| \underset{|w|\leq R}{\int} \frac{(\bw\cdot S\bz)\big(\lambda_1 (\bw\cdot S\bz) + \lambda_2 \big(SR\bw \cdot \na\psi\big(\eta(x,z)\big)\big) }{\sqrt{x\cdot \bz + 1-|x|^2}} \frac{\d w}{|w|^{d+2s-2}} \bigg| \\
&\leq R^{2s}\underset{\mathbb{S}^{1}}{\int} \frac{C_\psi}{\sqrt{x\cdot \bz + 1-|x|^2}} \text{d}\bz 
\end{align*}
where $C_\psi = \sup_{|w|\leq R} \Big( (\bw\cdot S\bz)\big( \lambda_1 (\bw\cdot S\bz) + \lambda_2 \big(SR\bw \cdot \na\psi\big(\eta(x,z)\big)\big)\Big)$ is uniformly bounded in $x$ and $\bz$. Thus, we have for $p>0$:
\begin{align*}
\underset{\Omega\times\RR^d}{\iint} \Big| \Delsv \Big[ \psi\big( t,\eta(x, v) \big)\Big] \Big|^{p}& F(v) \d xv  \leq  \underset{\Omega\times\RR^d}{\iint} \bigg| \underset{\mathbb{S}^{1}}{\int} \frac{2R^{2s} C_\psi }{\sqrt{x\cdot \bz + 1-|x|^2}} \text{d}\bz \bigg|^{p} F(v) \d xv \\
&\leq \big( 2R^{2s} C_\psi \big)^p \underset{\Omega\times\RR^d\times\mathbb{S}^{1}}{\iiint} \frac{1}{(x\cdot \bz + 1-|x|^2)^{p/2}} F(v) \text{d}\bz  \d xv 
\end{align*}
which we know to be finite if $p<3$ by \eqref{eq:Lintegrability} since $F$ is radial.
\end{proof}

\subsection{Change of variable} \label{subsec:etaballcov}

\begin{replemma}{lem:changeofvariable}
The change for variable $F$ given by 
\begin{align}
F \begin{pmatrix} x \\ v \end{pmatrix} = \begin{pmatrix} \eta(x,v) \\ - \big[\na_v \eta(x,v)\big] v \end{pmatrix}
\end{align}
is precisely the change of variable such that $\eta ( F(x,v) ) = x$ and the trajectory described by $\eta$ starting at $\eta(x,v)$ with velocity $- \big[\na_v \eta(x,v)\big] v$ is exactly the trajectory from $(x,v)$ backwards. Moreover, for all $(x,v)$:
\begin{equation}
\det \na F(x,v) = 1.
\end{equation}
\end{replemma}

\begin{proof}
From the explicit expression of $\na_v\eta(x,v)$ given above in \eqref{eq:naveta}, we see 
\begin{align*}
- \big[\na_v \eta(x,v)\big] v = - R_{k(\pi-2A)} v
\end{align*}
and by construction, see \eqref{eq:expliciteta}, we know the ending velocity of the trajectory is $R_{k(\pi-2A)} v$, see Figure \ref{fig:muxv} for a representation, so the trajectory from $F(x,v)$ is indeed the backward trajectory from $(x,v)$ which in particular implies that $\eta(F(x,v)) = x$. \\
In order to compute the determinant of $F$ we need the Jacobian with respect to $x$ of $\eta$. Following the same line of arguments as for the Jacobian in $v$ we write
\begin{align*}
\na_x \eta(x,v) &= \Big[ S R_{k(\pi-2A)} \big( x+v - k(z_0-z_{-1}) \big) \Big] \otimes \big( -2k \na_x A \big) \\
						& \quad - k R_{k(\pi-2A)} \na_x\big(z_0-z_{-1}\big) + R_{k(\pi-2A)}.
\end{align*}
where, from the relations we used to derive \eqref{eq:patheta} and \eqref{eq:paA} we have
\begin{align} \label{eq:gradxthA}
\na_x \theta =  \frac{2}{L}  S\frac{v}{|v|}, \hspace{2cm} \na_x A = \frac{-2}{L}  S\frac{v}{|v|}
\end{align}
which yields
\begin{align} \label{eq:gradxzi}
\na_x z_0 = \frac{2}{L} Sz_0 \otimes S\frac{v}{|v|}, \hspace{2cm} \na_x z_{-1}= \frac{-2}{L} Sz_{-1} \otimes S \frac{v}{|v|}.
\end{align}
As a consequence
\begin{align*}
\na_x \eta(x,v) &= SR_{k(\pi-2A)} \bigg[ \frac{4k}{L} \big( v- k(z_0-z_{-1}) \big) + \frac{2k}{L} \big(2s - z_0 - z_{-1}\big) \bigg] \otimes S\frac{v}{|v|} \\
&\quad + R_{k(\pi-2A)}
\end{align*}
and using \eqref{eq:distxzi} we get
\begin{align}\label{eq:naxeta}
\na_x \eta(x,v) = 2k\bigg( 2\frac{l_{end}}{L} -1\bigg) SR_{k(\pi-2A)} \frac{v}{|v|} \otimes S\frac{v}{|v|} + R_{k(\pi-2A)}.
\end{align}
We also need the Jacobian matrices of $-\big[\na_v \eta(x,v)\big] v$ which are
\begin{align} 
&\na_x\bigg( -\big[\na_v \eta(x,v)\big] v \bigg) = \frac{-4k|v|}{L} SR_{k(\pi-2A)}\frac{v}{|v|} \otimes S\frac{v}{|v|}, \label{eq:naxnaveta}\\
&\na_v \bigg(-\big[\na_v \eta(x,v)\big] v\bigg) = 2k\bigg(1-2\frac{l_{in}}{L} \bigg) SR_{k(\pi-2A)}\frac{v}{|v|} \otimes S\frac{v}{|v|} - R_{k(\pi-2A)}.\label{eq:navnaveta}
\end{align}
With appropriate coefficient $\alpha_x$, $\alpha_v$, $\beta_x$, $\beta_v$ (which are functions of $x$ and $v$), using the angular function $\Theta$ defined in \eqref{eq:thetafct} and writing $R$ instead of $R_{k(\pi-2A)}$ we can then write the Jacobian of $F$ as the following sum of block matrices
\begin{align*}
\na F(x,v) &= \begin{pmatrix} \na_x\eta (x,v) & \na_v\eta(x,v) \\ \na_x \bigg( -\big[\na_v \eta(x,v)\big] v \bigg)  & \na_v \bigg( -\big[\na_v \eta(x,v)\big] v \bigg) \end{pmatrix}  \\
&= \begin{pmatrix} \alpha_x \Theta & \alpha_v \Theta \\ \beta_x \Theta & \beta_v \Theta \end{pmatrix} + \begin{pmatrix} R & R \\ 0 & -R \end{pmatrix}.
\end{align*}
Now, we write $R^{-1}\Theta = S\frac{v}{|v|}\otimes S\frac{v}{|v|} := N$ 
which yields the relation
\begin{align*}
\det\Bigg( \begin{pmatrix} R^{-1} &  R^{-1} \\  0 & -R^{-1} \end{pmatrix} \na F(x,v) \Bigg) = \det \Bigg( \begin{pmatrix} (\alpha_x+\beta_x) N  & (\alpha_v+\beta_v) N\\ -\beta_x N & -\beta_v N   \end{pmatrix} + \begin{pmatrix} Id & 0 \\ 0 & Id \end{pmatrix} \Bigg)
\end{align*}
where we also notice that 
\begin{align*}
\det  \begin{pmatrix} R^{-1} &  R^{-1} \\  0 & -R^{-1} \end{pmatrix} = \det \Big( - R^{-2} \Big) = 1
\end{align*} 
because it is a rotation matrix in dimension 2. Therefore,  
\begin{align*}
\det \na F (x,v) = \det \begin{pmatrix} (\alpha_x+\beta_x) N +Id  & (\alpha_v+\beta_v) N\\ -\beta_x N & -\beta_v N +Id  \end{pmatrix} .
\end{align*}
Finally, it is rather simple to find the eigenvalues of this matrix. Indeed, since $N v = (v\cdot S\tfrac{v}{|v|} ) S\tfrac{v}{|v|} = 0$ we see that the 4-dimensional vectors $(v,0)$ and $(0,v)$ are both eigenvectors associated with the eigenvalue 1. Moreover, we notice that $N Sv = Sv$ so we solve for $\lambda$ and $\mu$ the equation
\begin{align*}
\begin{pmatrix} (\alpha_x+\beta_x) N +Id  & (\alpha_v+\beta_v) N\\ -\beta_x N & -\beta_v N +Id  \end{pmatrix} \begin{pmatrix} Sv \\ \lambda Sv \end{pmatrix} = \mu \begin{pmatrix} Sv \\ \lambda Sv\end{pmatrix}
\end{align*}
and find the two remaining eigenvalues:
\begin{align*}
\mu_1 = 1- 2k \big( k + \sqrt{k^2-1} \big) \\
\mu_2 = 1- 2k \big( k - \sqrt{k^2-1} \big). 
\end{align*}
Note that in order to find those values we used the relations $\alpha_x + \beta_x - \beta_v = -4k^2$ and $\beta_v \alpha_x - \beta_x\alpha_v = -4k^2$ which are deduced easily from the expressions \eqref{eq:naveta} \eqref{eq:naxeta} \eqref{eq:naxnaveta} and \eqref{eq:navnaveta}. In the end, we get the determinant of $\na F(x,v)$:
\begin{align*}
\det \na F(x,v) = \Big( 1- 2k \big( k + \sqrt{k^2-1} \big)\Big)\Big( 1- 2k \big( k - \sqrt{k^2-1} \big)\Big) = 1.
\end{align*}

\end{proof}

\bigskip 

\noindent
{\textbf{Conflict of Interest:} The authors have no conflicts of interest to declare.}

\section*{Acknowledgements}
This work has been supported by the European Research Council Grant ERC-2011-StG Mathematical Topics of Kinetic Theory. The author wishes to thank his Ph.D. advisors, Antoine Mellet and Cl\'{e}ment Mouhot, as well as Ariane Trescases, Emeric Bouin and Marc Briant for their simulating and very interesting discussions on the different topics of this paper.

\bibliographystyle{siam}

\bibliography{biblio.bib}

\end{document}